\documentclass[12pt]{amsart}
\usepackage{amsmath,latexsym,amsfonts,amssymb,amsthm}
\usepackage{geometry}
\usepackage{hyperref}
\usepackage{mathrsfs}
\usepackage{graphicx,color}
\usepackage[alphabetic]{amsrefs}
\usepackage{tikz}
\usepackage{tikz-cd}
\usepackage{bbm}
\usepackage{comment}
\usepackage[new]{old-arrows}

\newtheorem{prop}{Proposition}[section]
\newtheorem{thm}[prop]{Theorem}
\newtheorem{cor}[prop]{Corollary}
\newtheorem{conj}[prop]{Conjecture}
\newtheorem{lem}[prop]{Lemma}

\theoremstyle{definition}

\newtheorem{defn}[prop]{Definition}
\newtheorem{expl}[prop]{Example}
\newtheorem{rem}[prop]{\it Remark}
\newtheorem{emp}[prop]{}

\newtheorem*{claim*}{Claim}

\newcommand{\bC}{\mathbb{C}}
\newcommand{\bR}{\mathbb{R}}
\newcommand{\bA}{\mathbb{A}}
\newcommand{\bQ}{\mathbb{Q}}
\newcommand{\bZ}{\mathbb{Z}}
\newcommand{\bN}{\mathbb{N}}

\newcommand{\bT}{\mathbb{T}}
\newcommand{\bG}{\mathbb{G}}
\newcommand{\bk}{\mathbbm{k}}

\newcommand{\tX}{\widetilde{X}}
\newcommand{\tY}{\widetilde{Y}}
\newcommand{\tE}{\widetilde{E}}

\newcommand{\tS}{\widetilde{S}}

\newcommand{\cX}{\mathcal{X}}
\newcommand{\cY}{\mathcal{Y}}

\newcommand{\cO}{\mathcal{O}}
\newcommand{\cL}{\mathcal{L}}
\newcommand{\cI}{\mathcal{I}}

\newcommand{\cN}{\mathcal{N}}
\newcommand{\cF}{\mathcal{F}}

\newcommand{\cE}{\mathcal{E}}
\newcommand{\cD}{\mathcal{D}}

\newcommand{\cR}{\mathcal{R}}

\newcommand{\scX}{\mathscr{X}}
\newcommand{\scY}{\mathscr{Y}}
\newcommand{\scE}{\mathscr{E}}
\newcommand{\scD}{\mathscr{D}}
\newcommand{\scG}{\mathscr{G}}

\newcommand{\fa}{\mathfrak{a}}

\newcommand{\fm}{\mathfrak{m}}
\newcommand{\ft}{\mathfrak{t}}
\newcommand{\fg}{\mathfrak g}

\newcommand{\rd}{\mathrm{d}}

\newcommand{\Spec}{\mathrm{Spec}}
\newcommand{\Proj}{\mathrm{Proj}}
\newcommand{\Supp}{\mathrm{Supp}}
\newcommand{\Hom}{\mathrm{Hom}}
\newcommand{\mult}{\mathrm{mult}}

\newcommand{\Ex}{\mathrm{Ex}}
\newcommand{\Pic}{\mathrm{Pic}}
\newcommand{\Aut}{\mathrm{Aut}}
\newcommand{\vol}{\mathrm{vol}}
\newcommand{\ord}{\mathrm{ord}}
\newcommand{\Val}{\mathrm{Val}}

\newcommand{\gr}{\mathrm{gr}}

\newcommand{\Lie}{\mathrm{Lie}}

\newcommand{\QM}{\mathrm{QM}}

\newcommand{\wt}{\mathrm{wt}}

\newcommand{\an}{\mathrm{an}}

\newcommand{\Cl}{\mathrm{Cl}}

\numberwithin{equation}{section}

\newcommand{\R}{\mathbb R}
\newcommand{\C}{\mathbb C}

\newcommand{\tF}{\widetilde{\mathcal{F}}}

\newcommand{\tDe}{\widetilde \Delta}
\newcommand{\Q}{\mathbb Q}
\newcommand{\ep}{\varepsilon}

\begin{document}

\title[Stable degeneration and Kaledin's Conjecture]{Stable Degeneration, Non-degenerate Forms, and Kaledin's Conjecture}

\date{}

\author{Chenyang Xu}
\author{Ziquan Zhuang}


\dedicatory{With an appendix by Henri Guenancia}

\address{Department of Mathematics, Princeton University, Princeton, NJ 08544, USA}
\email     {chenyang@princeton.edu}

\address{Department of Mathematics, Johns Hopkins University, Baltimore, MD 21218, USA}
\email{zzhuang@jhu.edu}

\address{Univ. Bordeaux, CNRS, Bordeaux INP, IMB, UMR 5251, F-33400 Talence, France}
\email     {henri.guenancia@math.cnrs.fr}

\begin{abstract}
We prove that stable degeneration, the canonical degeneration associated to the normalized volume minimizer of a Kawamata log terminal (klt) singularity, preserves non-degenerate reflexive differential forms. In particular, the stable degeneration of a symplectic singularity is again symplectic. Combining this with a deformation-theoretic rigidity result for symplectic degenerations, we confirm Kaledin's conjecture that the formal completion of any symplectic singularity is conical. As applications, we show that the natural base of any normalized nilpotent orbit closure is a K-semistable Fano variety, and that the normalized volume minimizer of a hypertoric singularity is induced by the standard dilation. 
\end{abstract}

\maketitle

\section{Introduction}

For any real valuation $v$ over a klt singularity $x \in X$, Chi Li \cite{Li-nv} introduced an invariant $\widehat{\mathrm{vol}}_X(v)$ called the normalized volume of the valuation. A fundamental result in local K-stability theory \cite{Blu-existence, LX-higher-rank, Xu-quasimonomial, XZ-uniqueness, XZ-SDC} (see also \cite{Z-survey-klt-stab}) states that up to rescaling, there exists a unique, quasi-monomial valuation $v_0$ over $x\in X$ that minimizes the normalized volume, and the associated graded ring $\mathrm{gr}_{v_0} \cO_{X,x}$ is finitely generated. The affine variety $X_0 := \mathrm{Spec}(\mathrm{gr}_{v_0} \cO_{X,x})$, called the \emph{stable degeneration} of $x \in X$, is a K-semistable Fano cone singularity equipped with a canonical torus action.

A natural question is which additional geometric structures on $X$ are inherited by $X_0$. In this paper, we answer this question affirmatively for non-degenerate differential forms.

\begin{thm}[see Theorem \ref{thm:sigma-admissible nbhd}]\label{main-thm:degenerate}
Let $x \in X$ be a klt singularity with stable degeneration $x_0 \in X_0$. Let $\sigma \in H^0(X, \Omega^{[p]}_X)$ be a non-degenerate (reflexive) $p$-form. Then $\sigma$ specializes to a non-degenerate $p$-form $\sigma_0 \in H^0(X_0, \Omega^{[p]}_{X_0})$ on $X_0$.
\end{thm}

Here a reflexive $p$-form $\sigma$ on a normal variety $X$ of dimension $n$ is called \emph{non-degenerate} if $p \mid n$ and $\sigma^{\frac{n}{p}}$ is a nowhere vanishing section of $\omega_X=\Omega^{[n]}_X$ on the smooth locus. The specialization of $p$-forms in Theorem \ref{main-thm:degenerate} is induced by a test configuration, corresponding to some Koll\'ar component that approximates the normalized volume minimizer, see Theorem \ref{thm:sigma-admissible nbhd} for the more precise statement.  Moreover, one can easily see that closedness of forms is preserved under this specialization. 

Recall that a symplectic singularity is a klt singularity carrying a closed non-degenerate $2$-form \cite{Bea-symp-sing}. As a consequence of Theorem \ref{main-thm:degenerate}, symplectic singularities are preserved under stable degeneration. Combined with a rigidity result for symplectic degenerations (\cite{Namikawa-deformation-terminal,Namikawa-deformation,Namikawa-notes-on-deformation,NO-symp}, see also Theorem \ref{thm:symplectic rigid}), this yields a proof of a conjecture of Kaledin \cites{Kaledin-sym, Kaledin-survey} that the formal completion of any symplectic singularity is conical. Combined further with \cite{Namikawa-finite}, it also implies that there are only countably many formal isomorphism classes of symplectic singularities. 

\begin{thm}[Kaledin's Conjecture, see Corollary \ref{cor-Kaledin}]\label{main-thm:Kaledin}
Every symplectic singularity is formally isomorphic to a conical symplectic singularity, and every conical symplectic singularity is a K-semistable Fano cone.
\end{thm}

We refer to Definition \ref{def:Reeb vector} and Paragraph \ref{say-higherrank} for the definition of K-semistable Fano cones, and to Theorem \ref{thm:Kaledin enhanced} for a stronger statement.  Roughly speaking, being a K-semistable Fano cone is an algebraic condition that is closely related to the existence of Ricci-flat K\"ahler cone metrics on singularities with good torus actions. The second part of Theorem \ref{main-thm:Kaledin} implies among other things that every conical symplectic singularity admits a canonical dilating action by a torus whose rank can a priori be larger than one, and further suggests that every conical symplectic singularity has some hidden hyperk\"ahler structure. 
In fact, combined with the Yau-Tian-Donaldson correspondence between K-polystable Fano cones and Ricci-flat K\"ahler cones \cite{Huang-thesis,Li-Fano-cone-YTD}, the following was essentially conjectured in \cite[Section 3.2]{XZ-open}.


\begin{conj}
Any conical symplectic singularity is  a hyperk\"ahler cone.
\end{conj}


By the Artin approximation theorem \cite[Corollary 2.6]{Artin-approx}, formally isomorphic algebraic singularities are also \'etale locally isomorphic. Moreover, the symplectic structure, if it exists, is unique up to analytic isomorphism on any singularity germ, see Appendix \ref{s:Darboux}. Thus Theorem \ref{main-thm:Kaledin} also implies that every symplectic singularity is analytically conical. 
We also show that any two formally isomorphic conical symplectic singularities are indeed algebraically isomorphic (see Proposition \ref{prop:cone+form formal isom = alg isom}). In particular, this answers the algebraicity question raised in \cite[Page 3]{Namikawa-torichyperkahler} affirmatively.

\medskip

We apply our general theory to explicit symplectic singularities and identify minimizers of their normalized volume function. Finding a minimizer for general klt singularities is a challenging question. However, for a symplectic singularity, Theorem \ref{main-thm:Kaledin} implies that the minimizer can be found in the Reeb cone of the maximal torus of its automorphism group, and the symmetry of the singularity allows us to identify the minimizer in several cases. 

First, we address part of \cite[Problem 23]{XZ-open} by showing that the natural base of the normalization of any nilpotent orbit closure is a K-semistable Fano variety. 
We expect these bases to be K-polystable (see Remark \ref{rem-Kronheimer} for related earlier analytic works). 

\begin{thm}[see Theorem \ref{t-nilpotent}] \label{main thm:nilpotent}
Let $\fg$ be a semisimple Lie algebra and let $\widetilde{O}$ be the normalization of a nilpotent orbit closure with vertex $\tilde{o}$. Then the $\bG_m$ scaling action on $\fg$ gives the minimizer of the normalized volume on $\tilde{o}\in \widetilde{O}$. Consequently, the quotient $\left(\widetilde{O} \setminus \{\tilde{o}\}\right)/\bG_m$ is a K-semistable Fano variety. 
\end{thm}

Another application is to \emph{hypertoric} singularities (see Section \ref{ss-hypertoric} for the definition of related notions). This is a special case of a more general conjecture we made for symplectic singularities arising from Hamiltonian reductions (see \cite[Problem 27]{XZ-open}).

\begin{thm}[= Theorem \ref{t-hypertoric}] \label{main thm:hypertoric}
For any hypertoric singularity $Y(A,0)$, the normalized volume minimizer is given by the descent of the standard $\bG_m$ action on $\mathbb C^{2N}$.
\end{thm}


Theorems \ref{main-thm:degenerate} and \ref{main-thm:Kaledin} are closely related to the recent analytic work of Namikawa-Odaka \cite{NO-symp}. In the setting of symplectic singularities lying on smoothable projective symplectic varieties, they use the Donaldson–Sun theory of Ricci-flat metric tangent cones \cite{DS-degeneration2}, together with Poisson deformation theory, to construct a canonical good torus action; in particular, they prove Kaledin’s conjecture in the global smoothable projective setting, and obtain additional metric information. Our approach is different in both scope and method. We replace the metric tangent cone by the stable degeneration associated to the normalized-volume minimizer, which is available for \emph{arbitrary} klt singularities. The main new point is to prove that this algebraic degeneration preserves non-degenerate reflexive forms. Consequently, when the original singularity is symplectic, its stable degeneration is again symplectic; combining this with the deformation-theoretic rigidity of symplectic degenerations gives Kaledin’s conjecture for arbitrary symplectic singularities.

The key algebraic input behind this non-degeneracy preserving statement is a relation between the vanishing order of a differential form along a Kollár component and the K-stability of the associated log Fano pair. 


More precisely, for any reflexive $p$-form $\sigma$ on a normal variety $X$ and any prime divisor $D$ on a log resolution $Y\to X$, we define the \emph{log discrepancy $A_\sigma (D)$} as the zero or pole order along $D$ of $\sigma$ as a rational section of $\Omega^p_Y(\log D)$. The definition also extends to quasi-monomial valuations, see Definition \ref{defn-log discrepancy with respect to a form}. By \cite{GKKP-extend-forms,KS-extend-forms}, every reflexive differential form $\sigma$ on a klt singularity extends holomorphically to any resolution and in particular $A_\sigma (D)\ge 0$ for any prime divisor $D$ over the singularity. We strengthen this result for divisors/valuations that we consider as follows.

\begin{thm}[see Theorem \ref{thm:estimate dis}] \label{main-thm:estimate dis}
Let $\sigma$ be a non-zero reflexive $p$-form on a klt singularity $x\in X$ of dimension $n$.
\begin{enumerate}
\item Let $D$ be a Koll\'ar component over $x$. Then
\[
\min\{\delta(D,\Delta_D),1\} \cdot \frac{A_X(D)}{n}\le \frac{A_{\sigma}(D)}{p} \, .
\]
\item Let $v$ be the minimizer of the normalized volume function. Then 
\[
\frac{A_X(v)}{n}\le \frac{A_{\sigma}(v)}{p} \, .
\]
\end{enumerate}
\end{thm}

Here the log Fano pair $(D,\Delta_D)$ is obtained by adjunction along the Koll\'ar component (see Definition \ref{def:kc}), and $\delta(D,\Delta_D)$ is its stability threshold (see Definition \ref{def:Fano K-ss}). A log Fano pair $(D,\Delta_D)$ is K-semistable if and only if $\delta(D,\Delta_D)\ge 1$. In a way, Theorem \ref{main-thm:estimate dis} says that differential forms on a klt singularity tend to have higher vanishing order along Koll\'ar components that are closer to being K-stable.

\subsection*{Sketch of proof}

Let us sketch proofs of Theorems \ref{main-thm:degenerate} and \ref{main-thm:estimate dis}, and explain the relations between them. First assume for simplicity that the minimizer of the normalized volume function on the klt singularity $x\in X$ is a divisorial valuation $\ord_D$, given by some Koll\'ar component $D$.

The test configuration induced by a Koll\'ar component $D$ can be constructed by a deformation to normal cone process (see \cite{LX-Kol-comp-stab}). Using this explicit description, we first show that the specialization of a non-degenerate $p$-form $\sigma$ remains non-degenerate on the central fiber of the test configuration if and only if its log discrepancy satisfies 
\begin{eqnarray}\label{e-logdiscrepancy}
\frac{A_{\sigma}(D)}{p}=\frac{A_{X}(D)}{n} \, ,
\end{eqnarray}
where $n$ is the dimension of $X$. We refer to Section \ref{s-differential forms} for the details of the proof, but here is an example that illustrates the intuition behind this: let $X=\mathbb A^n$ for $n=2k$ with the standard symplectic form $\sigma = \rd x_1 \wedge \rd x_2+\cdots+ \rd x_{n-1}\wedge \rd x_{n}$ (thus $p=2$), and consider the $\bG_m$-action on $X$ with weight $\alpha_i\in \bN_{>0}$ on $x_i$. This gives a product test configuration whose corresponding Koll\'ar component $D_{\alpha}$ is the exceptional divisor of the weighted blow up with weights $(\alpha_1,\cdots,\alpha_{n})$. We have 
\[
A_{\sigma}(D_\alpha)=\min_{1\le j \le k}\{\alpha_{2j-1}+\alpha_{2j}\} \quad \mathrm{and}\quad A_X(D_\alpha) = \alpha_1+\dots+\alpha_n\,.
\]
The specialization $\sigma_0$ is the initial term of $\sigma$ with respect to the given weights. It remains non-degenerate if and only if
\[
\alpha_1+\alpha_2=\alpha_3+\alpha_4=\cdots=\alpha_{n-1}+\alpha_n \, ,
\]
and it is not hard to see that this is equivalent to \eqref{e-logdiscrepancy}.

Under our simplifying assumption that the minimizer is given by a Koll\'ar component $D$, it follows from the general local K-stability theory \cite{LX-Kol-comp-stab} that the associated log Fano pair $(D,\Delta_D)$ is K-semistable. For such a Koll\'ar component, Theorem \ref{main-thm:estimate dis}(1) gives 
\[
\frac{A_{\sigma}(D)}{p}\ge \frac{A_{X}(D)}{n},
\]
which holds for any $p$-form $\sigma$ (not necessarily non-degenerate) on $X$. On the other hand, if $\sigma$ is non-degenerate, then $\sigma^{\frac{n}{p}}$ is a nowhere vanishing volume form and by definition 
\[
\frac{A_X(D)}{n} = \frac{A_{\sigma^{n/p}}(D)}{n}\ge \frac{A_{\sigma}(D)}{p}
\]
holds for any Koll\'ar component $D$. Hence \eqref{e-logdiscrepancy} holds when the Koll\'ar component $D$ is K-semistable and the $p$-form is non-degenerate, and Theorem \ref{main-thm:degenerate} follows in this case. 

In general, the normalized volume minimizer is only quasi-monomial, but we can always approximate it using Koll\'ar components. Thanks to Theorem \ref{main-thm:estimate dis}(2) and the local linearity of the log discrepancy function, one can still show that \eqref{e-logdiscrepancy} holds for Koll\'ar components that are sufficiently close to the minimizer in its rational envelope. This gives the general case of Theorem \ref{main-thm:degenerate}. As we mentioned above, Theorem \ref{main-thm:Kaledin} is a direct consequence of Theorem \ref{main-thm:degenerate} and a rigidity property of symplectic singularities.

We next explain the proof of Theorem \ref{main-thm:estimate dis}. The idea is to exploit the (slope) semistability of the canonical extension of the orbifold cotangent sheaf of the log Fano pair $(D,\Delta_D)$. For simplicity, first assume that the Koll\'ar component in Theorem \ref{main-thm:estimate dis}(1) is the normalized volume minimizer; in particular, it is K-semistable and $\delta(D,\Delta_D)\ge 1$. Let $Y\to X$ be its plt blowup. On the index one covering stack $\cY\to Y$ with respect to $D$, with $\cD$ the pullback of $D$, we have an exact sequence
\[
0\to  \Omega^{[1]}_{\cD} \to \Omega^{[1]}_{\cY}(\log \cD)|_{\cD} \to  \mathcal{O}_{\cD} \to 0\,
\]
whose extension class is given by the pullback of $c_1(-(K_D+\Delta_D))$. The sheaf in the middle is known as the canonical extension of the orbifold cotangent sheaf. 

By \cite{DGP-Q-Fano-decomp,Dai-stability} (see also \cite{Tian-extension,Li-extension-stab}), this extension sheaf is semistable if the log Fano pair has standard coefficients and is K-semistable. Therefore, if $(D,\Delta_D)$ is K-semistable, then the sheaves $\Omega^{[p]}_{\cY}(\log \cD)|_{\cD}$, as well as their twists by line bundles, are semistable, and they have non-zero global sections only when the slope is non-negative. On the other hand, if $r:=A_{\sigma}(D)$, then it is not hard to see that the non-zero $p$-form $\sigma$ yields a nonzero section of $\Omega^{[p]}_{\cY}(\log \cD)(-r\cD)|_{\cD}$. Theorem \ref{main-thm:estimate dis} then follows from a calculation of slopes in this case. 

In general, the pair $(D,\Delta_D)$ does not need to be K-semistable, but we can still estimate the maximal slope of the canonical extension in terms of the stability threshold, using the existence of a ``K-semistable complement'' (i.e. a general effective $\bQ$-divisor $G\sim_{\mathbb Q}-(1-\delta)(K_D+\Delta_D)$ on $D$, where $\delta=\delta(D,\Delta_D)$, such that $(D,\Delta_D+G)$ is K-semistable, see \cite[Theorem 1.8]{LXZ-HRFG}). For this to work, we also need a generalization of the semistability of the canonical extension to K-semistable log Fano pairs with general coefficients, proved by Guenancia in Appendix \ref{s:can ext ss} by an analytic argument.

Finally, to tackle Theorem \ref{main-thm:estimate dis}(2), we again approximate the normalized volume minimizer by Koll\'ar components. The key is to show that the approximating Koll\'ar components are close to being K-semistable, in the sense that $\delta(D,\Delta_D)$ can be arbitrarily close to $1$ (see Theorem \ref{thm:delta->1}). This is a consequence of the stable degeneration theory, and ultimately relies on the higher rank finite generation of the minimizer \cite{XZ-SDC} and some tools developed in \cite{XZ-uniqueness}.

\medskip

This paper is organized as follows. In Section \ref{s-stability theory}, we review the local K-stability theory for klt singularities. We also show the approximation result Theorem \ref{thm:delta->1}. In Section \ref{s-differential forms}, we study how differential forms degenerate under a special test configuration. We introduce  the log discrepancy of a quasi-monomial valuation with respect to a reflexive differential form (see Definition \ref{defn-log discrepancy with respect to a form}), and use it to characterize whether the specialization of a non-degenerate $p$-form remains non-degenerate (see Proposition \ref{prop:non-degenerate criterion}). In Section \ref{s-deformation}, we give a self-contained proof of Theorem \ref{thm:symplectic rigid} first established in \cite{NO-symp}, which roughly says that any symplectic special degeneration of a symplectic singularity is analytically trivial. In Section \ref{s-canonical}, we discuss the orbifold cotangent sheaf and its canonical extension. We use the semistability of the canonical extension to relate log discrepancy with stability thresholds and prove Theorems \ref{main-thm:degenerate} and \ref{main-thm:estimate dis}. In Section \ref{s-Kaledin}, we prove a stronger form of Kaledin's conjecture and deduce Theorem \ref{main-thm:Kaledin}. We also establish some structural results for conical symplectic singularities, which can be used to answer the algebraicity question in \cite[Page 3]{Namikawa-torichyperkahler}. In Section \ref{s-open}, we apply our results to study (normalized) nilpotent orbit closures and hypertoric singularities, and obtain Theorems \ref{main thm:nilpotent} and \ref{main thm:hypertoric}. 

\subsection*{Notation and Conventions}

Throughout we work over the field $\bC$ of complex numbers. A singularity $x\in (X,\Delta)$ consists of an affine variety $X$, an effective $\bQ$-divisor $\Delta$ on $X$, and a closed point $x\in X$\footnote{We work in this generality only for possible future reference. Readers are free to assume $\Delta=0$ throughout this paper.}. Symplectic singularities are denoted by $(x\in X,\sigma)$ where $\sigma$ is the symplectic form. Denote by $\Val_{X}$ the set of real valuations of the function field $\bC(X)$ that have a center on $X$, and by $\Val_{X,x}$ the set of those valuations that are centered at the closed point $x\in X$. We follow the standard terminology from \cites{KM98, Kol13, Xu-book}. 

\subsection*{Acknowledgements}

We would like to thank Dori Bejleri and Chi Li for helpful discussions and comments. We are especially grateful to Henri Guenancia for providing the proof of Theorem \ref{thm:can ext ss}. CX wants to thank the organizers of the conference {\it Deformations and Birational Geometry of Algebraic Varieties -- in celebration of the 60th birthday of Professor Yoshinori Namikawa} held at RIMS (Kyoto), from which he learned about recent progress on symplectic singularities. CX is partially supported by NSF Grant DMS-2201349 and a Simons Investigator grant.
ZZ is partially supported by the NSF Grant DMS-2234736, a Sloan research fellowship and a Packard fellowship.
Both authors are also partially supported by the Simons Collaboration Grant on Moduli of Varieties.

\section{Local K-stability theory}\label{s-stability theory}

\subsection{Stable degeneration theorem}

In this subsection, we give a brief overview of the stable degeneration theory of klt singularities. The starting point of the local K-stability theory is the normalized volume of valuations over a klt singularity, defined in \cite{Li-nv}. A deep and surprising phenomenon is that the minimizer of the normalized volume function carries remarkable geometric information. The most important part for us is the canonical K-semistable degeneration of a klt singularity induced by the minimizer. 

The following is the more precise statement; see \cites{LLX-nvsurvey, Z-survey-klt-stab, Xu-book} for more background. For any singularity $x\in X=\Spec(R)$, any valuation $v\in \Val_{X,x}$ and any $\lambda\in \bR$, let $$\fa_{\lambda}(v):=\{f\in R\mid v(f)\geq \lambda\}$$ be the valuation ideal, and let $$\gr_v R: = \bigoplus_{\lambda} \fa_{\lambda}(v)/\fa_{>\lambda}(v)$$ be the associated graded ring (it is not hard to see that this is an integral domain).

\begin{thm}[Stable degeneration]\label{thm-SDC}  
For any klt singularity $x\in (X=\Spec(R),\Delta)$, we have the following.
\begin{enumerate}
\item Up to rescaling, there exists a unique valuation $v\in \Val_{X,x}$ that minimizes the normalized volume, and the minimizer is quasi-monomial. 
\item The graded ring ${\gr}_v R$ is finitely generated.
\item Let $X_0={\rm Spec}({\gr}_v R)$, $\Delta_0$ the degeneration of $\Delta$, and $\xi_v$ the Reeb vector induced by $v$, then $(X_0,\Delta_0;\xi_v)$ is a K-semistable log Fano cone.
\end{enumerate}
\end{thm}

\begin{proof}
These statements are conjectured in \cite{Li-nv}, and proved in \cite{Blu-existence, LX-higher-rank, Xu-quasimonomial, XZ-uniqueness, XZ-SDC}. See also \cite{BLQ-convexity,Che-HRFG} for different proofs of certain parts of the statements. 
\end{proof}

\begin{rem}
Together with \cite{LWX-tangent-cone}, this produces a two-step canonical degeneration of any klt singularity to a K-polystable log Fano cone singularity. Such a two-step degeneration process, noticed first in \cite{DS-degeneration2}, can be constructed using metric geometry if $x\in X$ is contained in the Gromov-Hausdorff limit of a sequence of K\"ahler-Einstein manifolds.  
\end{rem}

We have stated Theorem \ref{thm-SDC} in a compact form. In the remaining part of this subsection, we will unpack the parts needed in our argument.

In general the minimizer $v$ in Theorem \ref{thm-SDC} could have rational rank (see \cite[2.13 and 2.16]{Xu-book} for the definition) larger than one,  but when its rational rank is one, the picture is simpler: the minimizer is given by a Koll\'ar component whose underlying log Fano pair is K-semistable. We first recall the relevant definitions.

\begin{defn}[\cite{Xu-pi_1-finite}] \label{def:kc}
Let $x\in (X,\Delta)$ be a klt singularity. We say that a prime divisor $D$ over $x\in X$ is a \emph{Koll\'ar component} if there exists a proper birational morphism $\pi\colon Y\to X$ such that $D$ is the unique exceptional divisor, $-(K_Y+D+\pi_*^{-1}\Delta)$ is $\bQ$-Cartier and ample over $X$, and $(Y,D+\pi_*^{-1}\Delta)$ is plt. The morphism $\pi\colon Y\to X$ is called the plt blowup of the Koll\'ar component $D$. As $K_Y+D+\pi_*^{-1}\Delta\sim_{\bQ,X} A_{X,\Delta}(D)\cdot D$, these conditions also imply that $-D$ is ample over $X$.
\end{defn}

By adjunction we may write
\begin{equation} \label{eq:adjunction on kc}
(K_Y + D + \pi_*^{-1}\Delta)|_D = K_D + \Delta_D
\end{equation}
for some $\bQ$-divisor on $D$. Then $(D,\Delta_D)$ is a klt log Fano pair.


\begin{defn}[\cite{FO-delta,BJ-delta}] \label{def:Fano K-ss}
Let $(X, \Delta)$ be a log Fano pair, i.e. $(X,\Delta)$ is projective klt and $H:=-(K_X+\Delta)$ is ample. For any valuation $v\in \Val_X$ and any positive integer $m$ such that $mH$ is Cartier and $N_m:=h^0(X,mH)>0$, let 
\[
\cF_v^\lambda H^0(X, mH) := \{s \in H^0(X, mH) \,|\, v(s) \geq \lambda\}
\]
be the filtration induced by $v$.
The corresponding \emph{$S$-invariant} (or expected vanishing order) is defined as
\[
S_{X,\Delta}(v) := \lim_{m \to \infty} \frac{\sum_{\lambda\in\bR} \lambda \cdot \dim \gr_v^\lambda H^0(X, mH)}{m \cdot N_m},
\]
where $\gr_v^\lambda := \cF_v^\lambda / \cF_v^{>\lambda}$ denotes the associated graded pieces. For any prime divisor $E$ over $X$, set $S_{X,\Delta}(E):=S_{X,\Delta}(\ord_E)$.

The \emph{stability threshold} (or \emph{$\delta$-invariant}) of the log Fano pair $(X, \Delta)$ is defined as
\[
\delta(X, \Delta) := \inf_{E} \frac{A_{X,\Delta}(E)}{S_{X,\Delta}(E)},
\]
where the infimum is taken over all prime divisors $E$ over $X$, and $A_{X,\Delta}(E)$ denotes the log discrepancy. We say that $(X, \Delta)$ is K-semistable if $\delta(X, \Delta) \geq 1$.
\end{defn}

For normalized volume minimizers given by Koll\'ar components, Theorem \ref{thm-SDC}(3) can be restated as follows.

\begin{thm} \label{thm:kc K-ss}
Let $x \in (X, \Delta)$ be a klt singularity. Suppose that the normalized volume minimizer is a divisorial valuation $\ord_D$ for some prime divisor $D$ over $x$. Then $D$ is a Koll\'ar component. Moreover, let $\pi\colon Y \to X$ be the plt blowup of $D$, and write 
\[
(K_Y + D + \pi_*^{-1}\Delta)|_D = K_D + \Delta_D
\]
by adjunction, then $(D, \Delta_D)$ is a K-semistable log Fano pair.
\end{thm}

\begin{proof}
See \cite[Theorem 1.2]{LX-Kol-comp-stab} or \cite[Theorems 3.6 and 3.10]{XZ-uniqueness}.
\end{proof}

Next we review the log Fano cone degeneration construction in Theorem \ref{thm-SDC}, which is induced by a test configuration.

\begin{defn}\label{defn-testconfiguration}
A \emph{test configuration} of a variety $X$ is a $\bG_m$-equivariant flat morphism $\scX\to \bA^1$ (for the canonical $\bG_m$-action on $\bA^1$) together with a $\bG_m$-equivariant isomorphism $\scX\setminus \scX_0\cong X \times (\bA^1\setminus\{0\})$. 

A test configuration of a singularity $x\in X$ is a test configuration $\scX\to \bA^1$ of $X$, where $\scX$ is affine, and a $\mathbb G_m$-equivariant section $\bA^1\ni t\mapsto x_t\in \scX_t$ such that the isomorphism $\scX\setminus \scX_0\cong X \times (\bA^1\setminus\{0\})$ identifies $(x_t\in \scX_t)$ with $(x\in X)$ when $t\neq 0$.

\end{defn}

Note that the $\bG_m$-equivariant section, once existence is known, is uniquely determined by the test configuration of $X$. We are mostly concerned with test configurations induced by a Koll\'ar component. The general construction is as follows.

\begin{defn} \label{def:TC}
Let $D$ be a prime divisor over a singularity $x\in X = \Spec(R)$. Let $\fa_m:=\fa_m(\ord_D)$ be the corresponding valuation ideals, and let
\begin{equation} \label{eq:extended Rees}
\cR:=\bigoplus_{m\in\bZ} t^{-m} \fa_m \subseteq R[t,t^{-1}]
\end{equation}
be the extended Rees algebra. Assume that $\cR$ is finitely generated (e.g. if $D$ is a Koll\'ar component). Then
\[
f\colon \scX:=\Spec(\cR)\to \bA^1_t
\]
is called the \emph{test configuration of $x\in X$ corresponding to $D$}. If $\Delta$ is a $\bQ$-divisor on $X$, we may define a $\bQ$-divisor $\Delta_{\scX}$ on $\scX$ as the closure of $\Delta\times (\bA^1\setminus\{0\})$ through the $\bG_m$-equivariant isomorphism $X\times (\bA^1\setminus\{0\})\cong \scX\setminus \scX_0$. Then the test configuration degenerates $x\in (X,\Delta)$ to the singularity
\begin{equation} \label{eq:central fiber}
x_0\in (\scX_0 = \Spec(\gr_D R),\Delta_0=\Delta_{\scX}|_{\scX_0})
\end{equation}
where we use the shorthand notation $\gr_D R$ for $\gr_{\ord_D} R$. When $x\in (X,\Delta)$ is klt and $D$ is a Koll\'ar component over it, we call the corresponding test configuration a \emph{special} test configuration. In this case, the singularity \eqref{eq:central fiber} is klt by \cite[Theorem 4.1]{XZ-SDC} (see also \cite[Section 2.4]{LX-Kol-comp-stab}).
\end{defn}

When $D$ gives the normalized volume minimizer, the klt singularity \eqref{eq:central fiber} is the one that appears in Theorem \ref{thm-SDC}(3).

The following fact is not needed in the rest of the paper but may be of independent interest. It is proved in \cite[Proposition 4.1]{NO-symp} for the stable degeneration. 

\begin{lem}
Let $x\in (X,\Delta)$ be a singularity and let $\scX\to \bA^1$ be a test configuration of $x\in X$ corresponding to some divisor $D$ over $x\in X$. Let $N$ be an integer. Assume that $N(K_X+\Delta)$ is Cartier at $x$. Then $N(K_{\scX}+\Delta_{\scX})$ is Cartier at $x_0$. 
\end{lem}

\begin{proof}
Since $N(K_X+\Delta)$ is Cartier at $x$, after replacing $X$ with a smaller open neighborhood of $x$ we may assume that $N(K_X+\Delta)\sim 0$. This does not change $\scX_0$, hence does not affect the statement we want to prove. Since $\scX\setminus \scX_0 \cong X\times (\bA^1\setminus\{0\})$, on the total space $\scX$ we get a Weil divisor $L:=N(K_{\scX}+\Delta_{\scX})$ such that $L|_{\scX\setminus \scX_0} \sim 0$. Since $\scX_0$ is integral and $\scX_0\sim 0$, the natural restriction $\Cl(\scX)\to \Cl(\scX\setminus \scX_0)$ is an isomorphism. It follows that $L\sim 0$. In particular, $N(K_{\scX}+\Delta_{\scX})$ is Cartier at $x_0$.
\end{proof}




Geometrically, test configurations induced by prime divisors over a singularity have the following birational description. In the setting of Definition \ref{def:TC}, let $Y\to X$ be a proper birational map such that $Y$ is normal and $D$ appears as a divisor on $Y$. Let $\scY=Y\times \bA^1$, let $\varphi\colon \scY'\to \scY$ be the normalized blowup along $D\times \{0\}\subseteq \scY$, and let $\mathscr{G}$ be the (unique) exceptional divisor over the generic point of $D\times \{0\}$. Then



\begin{lem} \label{lem:tc bir construction}
The induced birational map $\scX\dashrightarrow \scY'$ is an isomorphism at the generic point of $\scX_0$, and $\mathscr{G}$ is the strict transform of $\scX_0$.
\end{lem}

\begin{proof}
This is well known to the experts. We include a proof for the reader's convenience. All we need to prove is that $\scX_0$ and $\scG$ give the same valuation of $\bC(X\times \bA^1)$. Since both divisors are invariant under the $\bG_m$-action, we only need to check that every $\bG_m$-equivariant function has the same vanishing order along $\scX_0$ and $\scG$. In other words, it suffices to show that $\ord_{\scX_0}(a\cdot t^m)=\ord_{\scG}(a\cdot t^m)$ for all $a\in R$ and $m\in\bN$. Let $l=\ord_D(a)$. Then by definition $\ord_{\scG}(a\cdot t^m) = l+m$. On the other hand, for any $s\in R[t]$, $\ord_{\scX_0}(s)$ is (by definition) the largest integer $r$ such that $t^{-r}s\in \cR$, thus from \eqref{eq:extended Rees} we get $\ord_{\scX_0}(a\cdot t^m) = l+m$ as well.
\end{proof}



Later we will need another geometric property of special test configurations, namely, the plt blowup extends to the test configuration, inducing a trivial degeneration of the Koll\'ar component.

\begin{lem} \label{lem:tc trivial on kollar comp}
Let $D$ be a Koll\'ar component over a klt singularity $x\in (X,\Delta)$ and let $\pi\colon Y\to X$ be its plt blowup. Let $\scX\to \bA^1$ be the corresponding test configuration (Definition \ref{def:TC}). Then there exists a projective birational morphism $\psi\colon \scY\to \scX$ with a unique exceptional divisor $\scD$ such that 
\begin{enumerate}
    \item $\psi$ is given by $\pi \times \mathrm{id}\colon Y\times (\bA^1\setminus\{0\})\to X \times (\bA^1\setminus\{0\})$ over $\bA^1\setminus\{0\}$, and
    \item if we write $(K_{\scY} + \scD + \psi_*^{-1}\Delta_{\scX})|_{\scD} = K_{\scD} + \Delta_{\scD}$ and $(K_Y + D + \pi_*^{-1}\Delta)|_D = K_D + \Delta_D$ by adjunction, then $(\scD, \Delta_{\scD})\cong (D, \Delta_D) \times \bA^1$.
\end{enumerate}
\end{lem}

\begin{proof}
This mostly follows from \cite{LX-Kol-comp-stab} or \cite[Proposition 4.6]{XZ-SDC}, but we also give a direct proof for the reader's convenience. As $-D$ is $\pi$-ample, we have $Y = \Proj_X \left( \bigoplus_{m\in \bN} \fa_m \right)$ where $\fa_m := \fa_m(\ord_D) = \pi_* \cO_Y(-mD)$ as before, and $$D\cong \Proj(\bigoplus_{m\in\bN} \fa_m/\fa_{m+1})=\Proj(\gr_D R)$$ by \cite[Proposition 2.10]{LZ-Tian-sharpness}. The flat extension $\widetilde{\fa}_m$ of $\fa_m$ to $\scX$ is given by
\[
\widetilde{\fa}_m = \cR\cap \fa_m[t,t^{-1}] = \bigoplus_{l\in\bZ} t^{-l} \fa_{\max\{l,m\}} \subseteq \cR.
\]
Finite generation of $\cR$ implies that $\bigoplus_{m\in \bN} \widetilde{\fa}_m\cong \bigoplus_{m\in\bN,\,l\le m} t^{-l}\fa_m$ is also finitely generated. Let $\scY:=\Proj_{\scX} \left( \bigoplus_{m\in \bN} \widetilde{\fa}_m \right)$ with induced map $\psi\colon \scY\to \scX$. Then (1) is satisfied by construction. Note that $\widetilde{\fa}_m$ is invariant under the $\bG_m$-action, hence the $\bG_m$-action lifts to $\scY$. We have
\[
\widetilde{\fa}_m/\widetilde{\fa}_{m+1} \cong \bigoplus_{l\in \bZ,\,l\le m} t^{-l}\cdot (\fa_m/\fa_{m+1}),
\]
thus the $\psi$-exceptional locus $\scD$ is 
$\Proj(\bigoplus_{m\in \bN} \widetilde{\fa}_m/\widetilde{\fa}_{m+1})\cong \Proj((\gr_D R)[t])\cong D\times \bA^1$. Under this isomorphism, the induced $\bG_m$-action on $\scD$ is trivial on the factor $D$. It is also not hard to see that $\Delta_{\scD}$ does not have any vertical component over $\bA^1$ (more generally, for any $\bG_m$-invariant Cartier divisor $H=(f=0)$ on $X$, where $0\neq f\in R$, the intersection $\psi_*^{-1}H_{\scX}\cap \scD$ does not have any vertical component; this can be checked using the explicit presentation of $\scY$ and $\scD$ above). Thus as  $\Delta_{\scD}$ is $\bG_m$-invariant, we see that $(\scD, \Delta_{\scD})$ is a trivial family over $\bA^1$. It follows that $(\scD, \Delta_{\scD})\cong (D, \Delta_D) \times \bA^1$ as its fiber over $0\neq t\in \bA^1$ is $(D,\Delta_D)$ by (1).
\end{proof}

To unwrap the higher rank case of Theorem \ref{thm-SDC}, we need some additional definitions.


\begin{defn}
Let $(Y,E)$ be a simple normal crossing (SNC) pair, let $E_1,\dots,E_r$ be irreducible components of $E$, and let $\eta$ be a generic point of $\cap_{i=1}^r E_i$. Then we have local coordinates $y_1,\dots,y_r\in \cO_{Y,\eta}$ such that $E_i=(y_i=0)$ around $\eta\in Y$. Any $u\in \cO_{Y,\eta}$ has a Taylor expansion
\[
u=\sum c_{\beta}y^{\beta} \in \widehat{\mathcal{O}}_{Y,\eta}\cong \bk(\eta)[\![y_1,\dots,y_r ]\!].
\]
For any $\alpha=(\alpha_1,\dots,\alpha_r)\in \bR_{\ge 0}^r\setminus\{0\}$, we can thus define a valuation $v_{\alpha}$ by setting
\[
v_{\alpha}(u)=\min \left\{\langle \alpha,\beta \rangle \,|\, c_{\beta}\neq 0\right\}.
\]
We denote the set of such valuations (for varying $\eta$ and $\alpha$) by $\QM(Y,E)$. 

A valuation $v\in \Val_X$ is called \emph{quasi-monomial} if there exists a log resolution $\pi\colon Y\to X$ and an SNC divisor $E\subseteq Y$ such that $v\in \QM(Y,E)$. If $(X,\Delta)$ is a klt pair, $\pi\colon (Y,E)\to (X,\Delta)$ is a log smooth model of the pair (i.e. $(Y,\Supp(E+\pi_*^{-1}\Delta))$ is SNC), and $v=v_\alpha$ as above, then the \emph{log discrepancy} of $v$ with respect to $(X,\Delta)$ is defined as (see \cite[Section 5]{JM-val-ideal-seq} and \cite[Section 3]{BdFFU-log-discrepancy} for more details) 
\[
A_{X,\Delta}(v) := \sum_{i=1}^r \alpha_i A_{X,\Delta}(E_i).
\]
\end{defn}

\begin{defn} \label{def:Reeb vector}
We say a torus $\bT=\bG_m^r$-action on a singularity $x\in (X=\Spec(R),\Delta)$ is {\it good} if it is effective and $x$ is in the closure of any $\bT$-orbit. A singularity with a good $\bG_m$-action is also called an \emph{orbifold cone} and the unique $\bG_m$-fixed point is called the \emph{vertex}.

Let $N:=N(\bT):=\Hom(\bG_m, \bT)$ be the co-weight lattice and $M:=M(\bT):=\Hom(\bT,\bG_m)$ the weight lattice. We have a weight decomposition 
\[
R=\bigoplus_{\alpha\in M} R_\alpha.
\]
The $\bT$-action gives rise to a natural homomorphism $N_\bR \subseteq \ft:=\Lie(\bT)\to H^0(X,T_X)$. A \emph{Reeb vector} on $X$ is an element $\xi\in N_\bR$ such that $\langle \xi, \alpha \rangle>0$ for all $0\neq \alpha\in M$ with $R_{\alpha}\neq 0$. The torus generated by $\xi$, denoted by $\langle\xi\rangle$, is the (unique) smallest torus of $\Aut(x\in (X,\Delta))$ such that $\xi\in N_\bR$. The set $\ft^+_{\bR}$ of Reeb vectors is called the \emph{Reeb cone}. 

Any Reeb vector $\xi\in \ft^+_{\bR}$ corresponds to a $\bT$-invariant valuation $\wt_\xi\in \Val_{X,x}$ defined by
\[
\wt_\xi (u):=\langle \xi, \alpha \rangle
\]
whenever $0\neq u\in R_\alpha$. A Reeb vector $\xi$ is said to be \emph{quasi-regular} if $\wt_\xi$ is a divisorial valuation; equivalently, there exists some $0\neq \lambda\in \bR$ such that $\lambda\cdot \xi\in N$. Finally, a \emph{log Fano cone singularity} is a klt singularity $x\in (X,\Delta)$ with a good torus action. It is also called a Fano cone (singularity) when $\Delta=0$.
\end{defn}

\begin{emp}[Rank one]\label{para-quasi-kollar}
Every quasi-regular Reeb vector $\xi$ generates a $\bG_m$-action, and every $\bG_m$-action on a log Fano cone singularity $x\in (X,\Delta)$ corresponds to a Koll\'ar component. Indeed, the induced morphism $\pi\colon X\setminus\{x\}\to (X\setminus\{x\})/\bG_m$ is a Seifert $\bG_m$-bundle in the sense of \cite{Kol-Seifert-bundle}. The divisorial valuation $\wt_\xi$ corresponds to the zero section $D$ of this Seifert bundle (in particular, we may and shall identify $D$ with $(X\setminus\{x\})/\bG_m$) and it follows from the arguments in \cite[Section 4]{Kol-Seifert-bundle} that this is a Koll\'ar component. Moreover, by {\it loc. cit.}, the log Fano pair $(D,\Delta_D)$ obtained from adjunction satisfies 
\begin{equation} \label{eq:orbifold cone crepant pullback}
\pi^*(K_D+\Delta_D) = (K_X+\Delta)|_{X\setminus \{x\}}.
\end{equation}
By \cite[Theorem 7]{Kol-Seifert-bundle} (see also \cite[Section 2.4]{LX-Kol-comp-stab} and \cite[Proposition 2.10]{LZ-Tian-sharpness}), there exists an ample $\bQ$-divisor $L\sim_\bQ -D|_D$ on $D$ such that
\begin{align*}
X\setminus\{x\} & \cong \Spec_D \left(\bigoplus_{m\in\bZ} \cO_D(mL)\right) \\
R & \cong \bigoplus_{m\in\bN} H^0(D,mL),
\end{align*}
where by convention $\cO_D(mL):=\cO_D(\lfloor mL\rfloor)$ and $H^0(D,mL):=H^0(D,\lfloor mL\rfloor)$. The plt blowup of $D$ is obtained by adding the zero section of the Seifert $\bG_m$-bundle:
\[
Y:= \Spec_D \left(\bigoplus_{m\in\bN} \cO_D(mL)\right) \to X = \Spec (R).
\]
\end{emp}

\begin{emp}[Higher rank minimizers]\label{say-higherrank}

The normalized volume minimizer $v$ is always quasi-monomial by \cite{Xu-quasimonomial}, but its rational rank $r$ could be larger than one. By \cite{XZ-SDC}, the associated graded ring $\gr_v R$ is finitely generated. Let $X_0 := \Spec(\gr_v R)$ be the degeneration of $X$, and let $\Delta_0$ be the degeneration of $\Delta$ (see the paragraph above \cite[Lemma 2.58]{LX-higher-rank} for the precise construction in this more general setting). The grading on $\gr_v R$ defines a good $\mathbb{T}\cong \mathbb{G}_m^r$-action with fixed point $x_0$ whose weight lattice is the value group $\Gamma_v := v(\bC(X)^*)$ of $v$. The natural inclusion $\Gamma_v\subseteq \bR$ gives an element $\xi_v$ of $N_\bR\cong M^*_\bR$. 

Theorem \ref{thm-SDC}(3) says that $x_0 \in (X_0, \Delta_0)$ is a log Fano cone singularity with Reeb vector $\xi_v$, and it is \emph{K-semistable} in the sense that $\wt_{\xi_v}$ minimizes the normalized volume on $x_0 \in (X_0, \Delta_0)$ (see \cite[Theorem 2.34]{LX-higher-rank} for the equivalence to the original definition using test configuration as in \cite{CS-cone}). We usually denote a K-semistable log Fano cone by $x\in (X,\Delta;\xi)$ to indicate the Reeb vector $\xi$ that minimizes the normalized volume. 

Since the minimizer is unique up to rescaling \cite{XZ-uniqueness}, we can also say a singularity $x\in (X,\Delta)$ is a \emph{K-semistable log Fano cone} if it is a log Fano cone (see Definition \ref{def:Reeb vector}), and the minimizer is given by a Reeb vector with respect to the given torus action.

\end{emp}

\begin{emp}

The K-semistability condition has a valuative description analogous to Definition \ref{def:Fano K-ss}. To state it, we first recall some definitions from \cite{XZ-uniqueness}.

Let $x\in (X=\Spec(R),\Delta)$ be a log Fano cone singularity of dimension $n$ and let $\bT$ be the torus acting on $X$. For any positive integer $m$, any Reeb vector $\xi\in \ft^+_\bR$, and any $\bT$-invariant valuation $v\in \Val_{X,x}$, we set (note that the associated graded ring $\gr_v R$ has a natural weight decomposition $\gr_v R = \bigoplus_{\alpha\in M,\lambda\in \bR} \gr_v^\lambda R_\alpha$)
\[
\tS_m(\xi;v):=\sum_{\lambda\in\bR,\,\alpha\in M,\langle\alpha,\xi\rangle< m} \lambda\cdot \dim \gr_v^\lambda R_\alpha.
\]
When $v=\wt_\eta$ for some Reeb vector $\eta$ on $X$, we will also denote the corresponding $\tS_m(\xi;v)$ by $\tS_m(\xi;\eta)$ (the same rules apply to other invariants of valuations such as log discrepancy). We then define ({\it cf.} \cite[Section 3.1]{XZ-uniqueness})
\begin{align*}
\tS(\xi;v)& : = \lim_{m\to \infty} \frac{\tS_m(\xi;v)}{m^{n+1}/(n+1)!}\\ 
S(\xi;v) & :=\frac{A_{X,\Delta}(\xi)}{\tS(\xi;\xi)}\cdot \tS(\xi;v),
\end{align*}
where the first limit exists and is positive by \cite[(3.4)]{XZ-uniqueness}. From the definition, we have 
\begin{equation} \label{eq:S tilde homogeneous}
\tS(\xi;\lambda v)=\lambda\cdot \tS(\xi;v)\mbox{ \ \ and \ \ }\tS(\lambda\xi; v)=\lambda^{-n-1}\cdot \tS(\xi;v)
\end{equation}
for any $\lambda>0$. It follows that 
\begin{equation} \label{eq:S scale inv}
S(\xi;v)=S(\lambda\xi;v)
\end{equation}
for any $\lambda>0$. We now have the following.

\begin{thm} \label{thm:minimizer is ss}
Let $x\in (X,\Delta;\xi)$ be a K-semistable log Fano cone singularity. Then $A_{X,\Delta}(v)\ge S(\xi;v)$ for any $\bT$-invariant quasi-monomial $v\in \Val_{X,x}$.
\end{thm}

\begin{proof}
This is \cite[Theorem 3.10]{XZ-uniqueness} (see also \cite[Theorem 1.1]{LX-higher-rank}).
\end{proof}

The converse is also true, but we will not need it in this paper.
\end{emp}

\subsection{Koll\'ar components near the minimizer}

Unlike the case of Koll\'ar components, higher rank minimizers do not come with a natural choice of test configurations as in Definition \ref{defn-testconfiguration}. In order to study the specialization of differential forms, we will use a sequence of Koll\'ar components to approximate the minimizer. Existence of such approximations follows easily from the higher rank finite generation, Theorem \ref{thm-SDC}(2). In this subsection, we show that the approximating Koll\'ar components are also close to being K-semistable. 

Before we state the precise result, observe that for any quasi-monomial valuation $v$ of rational rank $r$, we can always find a log smooth model $(Y,E)$ such that $E$ has exactly $r$ irreducible components and $v\in \QM(Y,E)$, see \cite[Lemma 3.6]{JM-val-ideal-seq}. Such a log smooth model is said to be \emph{adapted to} $v$\footnote{It is called a good pair adapted to $v$ in \cite{JM-val-ideal-seq}.}. In this case $v\in \QM(Y,E)^\circ\cong \bR_{> 0}^r$ and has $\bQ$-linearly independent coordinates. 

\begin{thm}\label{thm:delta->1}
Let $x\in (X=\Spec(R),\Delta)$ be a klt singularity and let $v$ be the normalized volume minimizer. Let $(Y,E)\to (X,\Delta)$ be a log smooth model adapted to $v$. Then for any $\varepsilon>0$, there exists an open neighborhood $U\subseteq \QM(Y,E)$ of $v$ such that every divisorial valuation in $U$ corresponds to a Koll\'ar component $D$ with $\delta(D,\Delta_D)\ge 1-\varepsilon$.
\end{thm}

Recall that the pair $(D,\Delta_D)$ is defined by adjunction on the plt blowup as in \eqref{eq:adjunction on kc}.

We will first prove Theorem \ref{thm:delta->1} in the case of K-semistable log Fano cones (see Lemma \ref{lem:delta->1 Fano cone}); the rest of the proof goes by reducing to this special case. We need the following criterion to estimate the stability thresholds of Koll\'ar components given by quasi-regular Reeb vectors.

\begin{lem} \label{lem:delta>? criterion}
Let $x\in (X=\Spec(R),\Delta)$ be a log Fano cone singularity and let $\bT$ be a torus acting on it. Let $\xi$ be a quasi-regular Reeb vector on $X$ and let $D$ be the corresponding Koll\'ar component. Let $\varepsilon\ge 0$. Then $\delta(D,\Delta_D)\ge 1-\varepsilon$ if and only if 
\begin{equation} \label{eq:A>=(1-epsilon)S}
A_{X,\Delta}(v)\ge (1-\varepsilon)S(\xi;v)  
\end{equation}
for all $\bT$-invariant quasi-monomial valuations $v\in \Val_{X,x}$.
\end{lem}

\begin{proof}
We largely follow the proof of \cite[Theorem 3.6]{XZ-uniqueness}. 
By the discussion in Paragraph \ref{para-quasi-kollar}, the Reeb vector $\xi$ generates a $\bG_m$-action on $(X,\Delta)$, the Seifert $\bG_m$-bundle $\varphi\colon X\setminus\{x\}\to D$ is given by some $\bQ$-divisor $L$ on $D$ and we have 
\[
R=\bigoplus_{m\in\bN} H^0(D,mL).
\]
Note that the torus $\bT$ also acts on the log Fano pair $(D,\Delta_D)$. Fix a positive integer $r$ such that $rL$ is Cartier. For any quasi-monomial valuation $w\in \Val_D$ and any $t\ge 0$, let $w_t\in \Val_{X,x}$ be the quasi-monomial valuation on $X$ defined by
\begin{eqnarray}\label{eq-wt}
w_t(s) = w(s)+ t\cdot \wt_\xi(s) 
\end{eqnarray}
for any $m\in\bN$ and any $0\neq s\in H^0(D,mL)$, where we set $w(s):=\frac{1}{r}w(s^r)$. Note that $w_t\in \Val_{X,x}$ if $t>0$ and it is $\bT$-invariant if and only if $w$ is $\bT$-invariant. We claim that 
\begin{equation}\label{eq-compare A}
A_{X,\Delta}(w_t) = A_{D,\Delta_D}(w)+t\cdot A_{X,\Delta}(\xi)\, .
\end{equation}
Indeed, if $L$ is Cartier, then $A_{X,\Delta}(w_0)=A_{D,\Delta_D}(w)$ as $(X,\Delta)$ is a cone over $(D,\Delta_D)$, and \eqref{eq-compare A} follows from the observation that $w_t$ is a monomial combination of $w_0$ and $\wt_\xi$ (i.e. on some log smooth model $(Y,E)$ of $(X,\Delta)$ the valuations $w_t$'s and $\wt_\xi$ belong to the same simplex of $\QM(Y,E)$ and the equality $w_t=w_0+t\cdot \wt_\xi$ holds in this simplex) and the fact that the log discrepancy function is linear on each simplex of $\QM(Y,E)$. In the general case, using the $\mu_r\subseteq \bG_m$ action on $x\in (X,\Delta)$ we let 
\[
\big(x'\in (X',\Delta')\big) = \big(x\in (X,\Delta)\big)/\mu_r,
\]
i.e. $X'=X/\mu_r$, and if $f\colon X\to X'$ is the induced map, then $x'=f(x)$ and 
\[
K_X+\Delta = f^*(K_{X'}+\Delta').
\]
Let $\xi'$ be the induced Reeb vector on $X'$ so that $\wt_{\xi'}\in \Val_{X'}$ is the restriction of $\wt_\xi$, and let $w'_t\in \Val_{X'}$ be the restriction of $w_t$. Note that $X'=\Spec(R')$ where $R':=\bigoplus_{m\in\bN} H^0(D,mrL)$, hence as $rL$ is Cartier and combined with \eqref{eq:orbifold cone crepant pullback} we know that $(X',\Delta')$ is a cone over $(D,\Delta_D)$. By construction, $w'_t(s) = w(s)+ t\cdot \wt_{\xi'}(s)$ for all $s\in H^0(D,mrL)$, thus from the special case of cones treated above we deduce that 
\[
A_{X',\Delta'}(w'_t) = A_{D,\Delta_D}(w)+t\cdot A_{X',\Delta'}(\xi').
\]
By \cite[Proof of Proposition 5.20(2)]{KM98}, we also have $A_{X,\Delta}(w_t)=A_{X',\Delta'}(w'_t)$ and $A_{X,\Delta}(\xi)=A_{X',\Delta'}(\xi')$. This gives \eqref{eq-compare A}.  



On the other hand, by \cite[(3.7)]{XZ-uniqueness}, we have $S(\xi;w_0) = S_{D,\Delta_D}(w)$. By \eqref{eq-wt}, we also see that $\tS(\xi;w_t) = \tS(\xi;w_0)+t\cdot \tS(\xi;\xi)$, hence 
\begin{equation}\label{eq-compare S}
S(\xi;w_t) = S_{D,\Delta_D}(w) + t\cdot A_{X,\Delta}(\xi)\, .
\end{equation}
Therefore, if $\delta(D,\Delta_D)\ge 1-\varepsilon$, then by combining  Equations \eqref{eq-compare A} and \eqref{eq-compare S}, we have
\begin{eqnarray*}
A_{X,\Delta}(w_t)&=&A_{D,\Delta_D}(w)+t\cdot A_{X,\Delta}(\xi) \\
                           &\ge& (1-\varepsilon)\cdot S_{D,\Delta_D}(w)+t\cdot A_{X,\Delta}(\xi) \\
                           &\ge &(1-\varepsilon)S(\xi;w_t)
\end{eqnarray*}
for any quasi-monomial valuation $w\in \Val_D$ and any $t>0$, where the first inequality is \cite[Theorem 4.8]{Xu-book}. Since $A_{X,\Delta}(\xi) = S(\xi;\xi)$ by definition, and every $\bT$-invariant quasi-monomial valuation in $\Val_{X}$ other than $\lambda\cdot \wt_\xi$ is of the form $w_t$, we conclude that \eqref{eq:A>=(1-epsilon)S} holds for all $\bT$-invariant quasi-monomial valuations $v\in \Val_{X,x}$.

Conversely, if $A_{X,\Delta}(w_t) \ge (1-\varepsilon)S(\xi;w_t)$ holds for all $\bT$-invariant quasi-monomial valuations $v\in \Val_{X,x}$, then letting $t\to 0^+$ we get (again by \eqref{eq-compare A} and \eqref{eq-compare S}) that
\[
A_{D,\Delta_D}(w)\ge (1-\varepsilon)\cdot S_{D,\Delta_D}(w)
\]
for all $\bT$-invariant quasi-monomial valuations $w\in \Val_D$. By \cite[Theorem 1.2]{Z-equivariant-K}, this implies $\delta(D,\Delta_D)\ge 1-\varepsilon$.
\end{proof}

The following lemma is the K-semistable log Fano cone case of Theorem \ref{thm:delta->1}.

\begin{lem} \label{lem:delta->1 Fano cone}
Let $x\in (X,\Delta;\xi)$ be a K-semistable log Fano cone singularity of dimension $n$. Then for any $\varepsilon>0$, there exists an open neighborhood $U\subseteq \ft_\bR^+$ of $\xi$ in the Reeb cone such that for any quasi-regular Reeb vector $\eta\in U$, the log Fano pair $(D,\Delta_D)$ on the corresponding Koll\'ar component satisfies $\delta(D,\Delta_D)\ge 1-\varepsilon$.
\end{lem}

\begin{proof}
Let $\bT$ be the torus acting on $x\in (X,\Delta)$. By Theorem \ref{thm:minimizer is ss}, we have $A_{X,\Delta}(v)\ge S(\xi;v)$ for any $\bT$-invariant quasi-monomial valuation $v\in \Val_{X,x}$. We claim that there exists an open neighborhood $U\subseteq \ft_\bR^+$ of $\xi$ such that 
\begin{equation} \label{eq:S perturbed}
S(\xi;v)\ge (1-\varepsilon)S(\eta;v)
\end{equation}
for any $\eta\in U$ and any $\bT$-invariant quasi-monomial valuation $v\in \Val_{X,x}$. The lemma follows immediately from Lemma \ref{lem:delta>? criterion} and this claim. 

Since the $\bT$-action is good, its weight cone $\sigma\subseteq M_\bR$ (generated by those $\alpha\in M(\bT)$ with $R_\alpha\neq 0$) is strongly convex, thus we can pick some $0<t\ll 1$ and some open neighborhood $U\subseteq \ft_\bR^+$ of $\xi$ such that $(1-t)\langle \alpha, \xi\rangle\le \langle \alpha, \eta\rangle$ for any $\alpha\in \sigma$ and $\eta\in U$. From the definition of the $\tS_m$ and $\tS$-invariant, this gives 
\[
\tS_m ((1-t)\xi;v)\ge \tS_m(\eta;v) \, .
\] 
Hence, for any $\bT$-invariant quasi-monomial valuation $v\in \Val_{X,x}$, we get
\[
\tS(\xi;v) = (1-t)^{n+1}\tS((1-t)\xi;v)\ge (1-t)^{n+1}\tS(\eta;v)
\]
by \eqref{eq:S tilde homogeneous}. By the first displayed formula on \cite[pp. 160]{XZ-uniqueness}, we have $\tS(\eta;\eta) = n\cdot \vol(\wt_\eta)$ for any $\eta\in \ft^+_\bR$ (see e.g. \cite[page 411]{ELS03} for the definition of the volume of a valuation). Since both the log discrepancy $\eta\mapsto A_{X,\Delta}(\eta)$ and the volume function $\eta\mapsto \vol(\eta)$ are continuous on the Reeb cone by \cite[Theorem 2.15(3) and Proposition 2.39]{LX-higher-rank}, after possibly shrinking $U$, we may assume that 
\[
A_{X,\Delta}(\xi)\ge (1-t)A_{X,\Delta}(\eta)\quad \mathrm{and}\quad \tS(\xi;\xi)\le (1-t)^{-1} \tS(\eta;\eta)
\]
for any $\eta\in U$. Putting these estimates together we obtain
\[
S(\xi;v) \ge (1-t)^{n+3} S(\eta;v),
\]
thus for sufficiently small $t$ this yields \eqref{eq:S perturbed}.
\end{proof}

We now prove the general case of Theorem \ref{thm:delta->1}.

\begin{proof}[Proof of Theorem \ref{thm:delta->1}]
Let $x_0\in (X_0,\Delta_0;\xi_v)$ be the K-semistable log Fano cone degeneration of $x\in (X,\Delta)$ given by Theorem \ref{thm-SDC}. Since $\gr_{v} R$ is finitely generated, by \cite[Lemma 2.10]{LX-higher-rank} we know that there exists an open neighborhood $U\subseteq \QM(Y,E)$ of $v$ such that $\gr_w R\cong \gr_{v} R$ for every quasi-monomial valuation $w\in U$. The grading on $\gr_w R$ induces a Reeb vector $\xi_w$ on $X_0=\Spec(\gr_v R)$ through this isomorphism. Moreover, $w\in U$ is a divisorial valuation if and only if $\xi_w$ is quasi-regular. 

The assignment $w\mapsto \xi_w$ then identifies $U$ with an open neighborhood of $\xi_v$ in the Reeb cone $\ft_\bR^+$ of $X_0$. By Lemma \ref{lem:delta->1 Fano cone}, the theorem holds for some open neighborhood $U_0$ of $\xi_v$ in the Reeb cone. Shrinking $U$ if necessary, we may assume that  $\xi_w\in U_0$ for all $w\in U$. We may further assume by \cite[Proposition 2.59]{LX-higher-rank} that every divisorial valuation $w\in U$ is a Koll\'ar component and the corresponding test configuration degenerates $\Delta$ to $\Delta_0$. The theorem now follows as the log Fano pairs on the Koll\'ar components corresponding to $w$ and $\xi_w$ are isomorphic by Lemma \ref{lem:tc trivial on kollar comp}.
\end{proof}

\section{Specializations of differential forms}\label{s-differential forms}

In this section, we define the specialization of differential forms along test configurations and establish a criterion for when the specialization of a non-degenerate differential form remains non-degenerate.

\begin{defn}\label{defn:p-form}
Let $X$ be a normal variety of dimension $n$, and let $i\colon X^{\mathrm{sm}} \hookrightarrow X$ be the inclusion of the smooth locus. We denote by $\Omega^{[p]}_X = i_*\Omega^p_{X^{\mathrm{sm}}}$ the sheaf of reflexive differential $p$-forms. A section $\sigma \in H^0(X, \Omega^{[p]}_X)$ is called a (reflexive) $p$-form on $X$. Moreover, $\sigma$ is called \emph{non-degenerate} if $p \mid n$ and $\sigma^{\frac{n}{p}} \in H^0(X, \Omega^{[n]}_X)$ is nowhere vanishing on $X^{\mathrm{sm}}$.

More generally, let $f\colon X \to S$ be a flat morphism of normal varieties and let $U \subseteq X$ be the smooth locus of $f$. Assume that $X\setminus U$ has codimension at least $2$ in $X$ (e.g. if $f$ is flat with normal fibers). Then we denote by $\Omega^{[p]}_{X/S} = i_*\Omega^p_{U/S}$ the sheaf of relative reflexive differential $p$-forms, where $i\colon U \hookrightarrow X$ is the inclusion.

Suppose further that $S$ is smooth. For any closed point $s \in S$ with normal fiber $X_s$, we have $U \cap X_s = X_s^{\mathrm{sm}}$. The natural morphism $\Omega^p_{U/S} \to \Omega^p_{X_s^{\mathrm{sm}}}$ then induces, by pushing forward, a restriction morphism 
\begin{equation}\label{eq:forms restriction}
\Omega^{[p]}_{X/S} \to \Omega^{[p]}_{X_s}\,.
\end{equation}

The definition naturally extends to the logarithmic setting. For any reduced divisor $D$ on a normal variety $X$, we set $\Omega^{[p]}_X (\log D) := i_*(\Omega^p_{U}(\log D))$ where $i\colon U\hookrightarrow X$ is the inclusion of the SNC locus of $(X,D)$. Moreover, if $f\colon X\to S$ is a morphism such that $(X,D)$ is SNC and log smooth over $S$ on some big open set $i\colon U \hookrightarrow X$, then we set $\Omega^{[p]}_{X/S}(\log D):=i_*(\Omega^p_{U/S}(\log D))$.
\end{defn}

\begin{rem}\label{rem:trivial-canonical}
If $X$ admits a non-degenerate $p$-form $\sigma$, then $\omega_X\cong \cO_X$. Indeed, the nowhere vanishing section $\sigma^{\frac{n}{p}}$ yields an isomorphism $\cO_{X^{\mathrm{sm}}}\cong \omega_{X^{\mathrm{sm}}}$, and pushing forward to $X$ gives
\[
\omega_X = i_*(\omega_{X^{\mathrm{sm}}}) \cong i_*\cO_{X^{\mathrm{sm}}} \cong \cO_X \, .
\]
In particular, such $X$ has rational singularities if and only if it is canonical.
\end{rem}

Next, we define the specialization of differential forms along test configurations with normal fibers.

\begin{defn}\label{def:degeneration of form}
Let $\scX\to \bA^1$ be a test configuration of a normal variety $X$ such that the central fiber $\scX_0$ is also normal. Let $p$ be a positive integer and $0\neq \sigma\in H^0(X,\Omega^{[p]}_X)$. Pulling back, we get a section $\mathrm{pr}_X^*\sigma$ of $\Omega^{[p]}_{X\times \bA^1 /\bA^1}$, and through the isomorphism $\scX\setminus \scX_0 \cong X\times (\bA^1\setminus \{0\})$, this gives a rational section $\sigma'$ of $\Omega^{[p]}_{\scX/\bA^1}$. Set $\sigma_{\scX}:=t^m \sigma'$, where $m$ is the smallest integer such that $t^m \sigma'$ extends to a section of $\Omega^{[p]}_{\scX/\bA^1}$. This is equivalent to saying that $t^m \sigma'$ is a non-vanishing section of $\Omega^{p}_{\scX/\bA^1,\eta}$ at the generic point $\eta$ of $\scX_0$.

Since $\scX_0$ is normal, we have a restriction map $\Omega^{[p]}_{\scX/\bA^1}\to \Omega^{[p]}_{\scX_0}$ as in \eqref{eq:forms restriction}. We call the (nonzero) $p$-form 
\[
\sigma_0:=\sigma_{\scX}|_{\scX_0}\in H^0(\scX_0,\Omega^{[p]}_{\scX_0})
\]
on $\scX_0$ the \emph{specialization} of $\sigma$ along the test configuration $\scX$.
\end{defn}

\begin{rem}
In general, for a morphism $X\to Y$ between varieties with rational singularities, one can define a functorial  pullback morphism
$\Omega^{[p]}_Y \to \Omega^{[p]}_X$ as in \cite{KS-extend-forms}. We will not need this generality. 
\end{rem}

\begin{prop}\label{prop:closed+non-deg}
Let $\scX\to \bA^1$ be a test configuration of a normal variety $X$ of dimension $n$ such that $\scX_0$ is also normal, and let $\sigma_0$ be the specialization of a $p$-form $\sigma$ on $X$ as in Definition \ref{def:degeneration of form}. Then:
\begin{enumerate}
\item If $\sigma$ is closed, then so is $\sigma_0$.
\item If $p\mid n$ and $\sigma$ is non-degenerate, then $\sigma_0$ is non-degenerate if and only if $\sigma_0^{n/p} \neq 0$, or equivalently, $\sigma_{\scX}^{n/p}$ does not vanish at the generic point of $\scX_0$.
\end{enumerate}
\end{prop}

\begin{proof}
We follow the notation of Definition \ref{def:degeneration of form}. 
For (1), note that if $\rd\sigma=0$ then $\rd\sigma'=0$ (it suffices to check this over $\scX\setminus \scX_0\cong X\times (\bA^1\setminus \{0\})$, since the sheaf in question is torsion-free). Thus $\rd(t^m\sigma') = mt^{m-1}\rd t\wedge\sigma'=0$ in $\Omega_{\scX/\bA^1}^{[p+1]}$ (again it is enough to verify this equality over $\scX\setminus \scX_0$). Restricting to the special fiber gives $\rd \sigma_0 = 0$.

For (2), the forward direction is clear. Suppose that $\sigma_0^{n/p}\neq 0$. Then since $\sigma$ is non-degenerate, the zero locus of $\sigma_{\scX}^{n/p} \in H^0(\scX,\Omega_{\scX/\bA^1}^{[n]})$ has codimension at least two in $\scX$. Since $\Omega_{\scX/\bA^1}^{[n]}$ is reflexive of rank one, this implies that $\Omega_{\scX/\bA^1}^{[n]}\cong \cO_{\scX}$ and $\sigma_{\scX}^{n/p}$ extends to a nowhere vanishing section of $\Omega_{\scX/\bA^1}^{[n]}$. In particular, its restriction yields a nowhere vanishing $n$-form on $\scX_0^{\mathrm{sm}}$, i.e.\ $\sigma_{0}$ is non-degenerate.
\end{proof}

To state our criterion for non-degeneracy of differential forms under specialization, we need to introduce a generalization of log discrepancy for differential forms. To this end, let $\cF$ be a locally free sheaf on a variety $X$, and let $s$ be a rational section of $\cF$. For any valuation $v\in \Val_X$, locally around its center $\eta$ on $X$ we can write $s=\sum_i a_is_i$ for an $\cO_{X,\eta}$-basis $\{s_i\}$ of $\cF_{\eta}$ with $a_i\in \bC(X)$. We define $v(s)=\min_i v(a_i)$, and it is clear that the definition does not depend on the choice of basis.

\begin{defn}\label{defn-log discrepancy with respect to a form}
Let $\sigma$ be a (not necessarily non-degenerate) $p$-form on a normal variety $X$. For any prime divisor $E$ over $X$ realized by a proper birational morphism $\pi\colon Y\to X$, we define the \emph{log discrepancy of the divisor $E$ with respect to $\sigma$} as
\[
A_{\sigma}(E):=\ord_E(\pi^*\sigma),
\]
where we view $\pi^*\sigma$ as a rational section of $\Omega^p_Y(\log E)$ (which is locally free at the generic point of $E$). Note that $A_\sigma(E)$ is always an integer.

More generally, for any quasi-monomial valuation $v\in \Val_X$ and any log smooth model $\pi\colon (Y,E)\to X$ such that $v\in \QM(Y,E)$, we define the \emph{log discrepancy of the valuation $v$ with respect to $\sigma$} as
\[
A_{\sigma}(v) := v(\pi^*\sigma),
\]
where $\pi^*\sigma$ is again viewed as a rational section of $\Omega^p_Y(\log E)$.
\end{defn}

\begin{lem}
The log discrepancy $A_{\sigma}(v)$ is well-defined, i.e. it does not depend on the choice of the log smooth model $(Y,E)$.
\end{lem}

\begin{proof}
Let $(Y',E')$ and $(Y,E)$ be two log smooth models such that
\[
v\in \QM(Y,E) \cap \QM(Y',E').
\]
By taking a common log resolution, we may assume there is a morphism $\varphi\colon Y'\to Y$ and $\Supp(\varphi^*E+\Ex(\varphi))\subseteq E'$. In particular, there is a natural map 
\begin{equation} \label{eq:pullback log differential}
\varphi^*(\Omega_Y(\log E))\to \Omega_{Y'}(\log E').    
\end{equation}
Suppose the center $c_{Y'}(v)$ of $v$ on $Y'$ is contained in the intersection of components $E'_1,\ldots, E'_r\subseteq E'$, and let $a_i:=v(E'_i)>0$ be the coordinate of $v$ along $E'_i$ for $i=1,\ldots,r$. Since $A_{Y,E}(E'_i)\ge 0$ and
\[
0=A_{Y,E}(v)=\sum_{i=1}^r a_i A_{Y,E}(E'_i)
\]
by definition, we have $A_{Y,E}(E'_i)=0$ for all $i$, hence as $\Supp(\varphi^*E+\Ex(\varphi))\subseteq E'$ we get $\varphi^*(K_Y+E)=K_{Y'}+E'$ near $c_{Y'}(v)$. This implies that \eqref{eq:pullback log differential} is an isomorphism near $c_{Y'}(v)$. Therefore, $v(\pi^*\sigma) = v(\varphi^*\pi^*\sigma)$ and the lemma follows.
\end{proof}

\begin{expl}
Let $X=\bA^n$ and $\sigma = \rd x_1 \wedge\dots\wedge \rd x_p$. Let $\alpha\in \bR_{\ge 0}^n$ and let $v_\alpha$ be the corresponding monomial valuation with $v_\alpha(x_i)=\alpha_i$. Then 
\[
A_{\sigma}(v_\alpha) = \alpha_1+\dots+\alpha_p \mbox{ \ \ and \ \ }A_X(v_{\alpha})=\alpha_1+\dots+\alpha_n \, .
\]
\end{expl}

\begin{lem}\label{lem:A linear}
Let $\cF$ be a vector bundle on a normal variety $X$, let $s$ be a rational section of $\cF$, and let $v\in \Val_X$ be a quasi-monomial valuation of rational rank $r$. Then there exists a log smooth model $(Y,E=E_1+\cdots+E_r)$ adapted to $v$ such that
\[
v(s) = \sum_{i=1}^r \alpha_i \ord_{E_i}(s)\, ,
\]
where $(\alpha_1,\ldots,\alpha_r)$ are the coordinates of $v$ in $\QM(Y,E)$. In particular, for any differential form $\sigma$ on $X$ and any log smooth model $(Y,E)$ adapted to $v$, the log discrepancy function $A_\sigma(\cdot)$ is linear in some neighborhood of $v\in \QM(Y,E)$.
\end{lem}

\begin{proof}
Since the question is local, we may assume that $X$ is affine and $\cF$ is free with basis $\{s_i\}$. Since $s$ is a rational section, there exists some $a\in H^0(\cO_X)$ such that $a s\in H^0(\cF)$. Write $as=\sum_i a_i s_i$ and let $\cI\subseteq \cO_X$ be the ideal generated by the $a_i$'s so that $Z=V(\cI)\subseteq X$ is the zero scheme of $(as)$. Then by definition $v(s) = v(\cI)-v(a)$, hence it is linear in some neighborhood of $v\in \QM(Y,E)$ if $(Y,E)$ is adapted to $v$. Moreover, $v\mapsto v(s)$ is linear on any log smooth model $\pi\colon(Y,E)\to X$ adapted to $v$ such that $E+\pi^{-1}(Z)+\{\pi^*a=0\}$ is a divisor with SNC support.
\end{proof}





\begin{thm} \label{thm:log disc>=0}
Let $X$ be a variety with klt singularities and let $\sigma$ be a reflexive differential form on $X$. Then $A_{\sigma}(v)\ge 0$ for any quasi-monomial valuation $v\in \Val_X$.
\end{thm}

\begin{proof}
This follows from \cite[Theorem 1.4]{GKKP-extend-forms} or \cite[Corollary 1.7]{KS-extend-forms}.
\end{proof}

We now formulate the criterion for non-degeneracy of differential forms under special test configurations of klt singularities.

\begin{defn} \label{def:sigma-admissible}
Let $D$ be a Koll\'ar component over a klt singularity $x\in X$ of dimension $n$, and $\pi\colon Y\to X$ the corresponding plt blowup.  Let $\sigma$ be a non-degenerate $p$-form on $X$. We define the \emph{strict transform} of $\sigma$ on the plt blowup as the induced section
\[
\tilde{\sigma}\in H^0(Y,\Omega^{[p]}_{Y}(\log D)(-rD)),
\]
where $r=A_{\sigma}(D)$. 

We call $D$ a \emph{$\sigma$-admissible} Koll\'ar component  if the specialization of $\sigma$ along the test configuration induced by $D$ is still non-degenerate.
\end{defn}

\begin{prop} \label{prop:non-degenerate criterion}
Notation as above. Then the following are equivalent:
\begin{enumerate}
    \item $D$ is $\sigma$-admissible.
    \item $\frac{A_X(D)}{n} = \frac{A_{\sigma}(D)}{p} $.
    \item $\tilde{\sigma}^{n/p} \in H^0(Y,\Omega^{[n]}_{Y}(\log D)(-\frac{nr}{p}D))$ does not vanish at the generic point of $D$. 
\end{enumerate}
\end{prop}

\begin{proof}
By definition, 
\[
A_X(D)=\ord_D(\pi^*\sigma^{n/p})\ge \frac{n}{p}\cdot \ord_D(\pi^*\sigma) = \frac{n}{p}\cdot A_\sigma (D) = \frac{nr}{p}\, ,
\]
with equality if and only if $\tilde{\sigma}^{n/p} \in H^0(Y,\Omega^{[n]}_{Y}(\log D)(-\frac{nr}{p}D))$ does not vanish at the generic point of $D$. This proves that $(2)\Leftrightarrow (3)$.

To  prove $(1)\Leftrightarrow (3)$, we track how the differential form specializes using Lemma \ref{lem:tc bir construction}. We follow the notation thereof: let $\scY=Y\times \bA^1$, $\varphi\colon \scY'\to \scY$ be the normalized blowup along $D\times \{0\}\subseteq \scY$, and $\mathscr{G}$ be the exceptional divisor over the generic point $\eta$ of $D\times \{0\}$. Let $\scD=D\times \bA^1\subseteq \scY$. Then $\mathrm{pr}_Y^*\tilde{\sigma}$ is a section of 
\[
\mathrm{pr}_Y^*\Omega^{[p]}_{Y}(\log D)(-rD) = \Omega^{[p]}_{\scY/\bA^1}(\log \scD)(-r\scD)
\]
that does not vanish at $\eta$, thus
\[ 
\sigma_{\scY'} = \varphi^*\mathrm{pr}_Y^*\tilde{\sigma} \in \varphi^*\Omega^{[p]}_{\scY/\bA^1}(\log \scD)(-r\scD) 
\]
does not vanish at the generic point $\zeta$ of $\mathscr{G}$ by construction. By Proposition \ref{prop:closed+non-deg}(2), in order to check whether the specialization $\sigma_0$ of $\sigma$ along the test configuration $\scX\to \bA^1$ is non-degenerate, it suffices to show that $\sigma_{\scX}^{n/p}$ is non-vanishing at the generic point of $\scX_0$. By Lemma \ref{lem:tc bir construction}, $\scX_0$ is the strict transform of $\mathscr{G}$. By a local computation, we have 
\[
\varphi^*\Omega^{[p]}_{\scY/\bA^1}(\log \scD)(-r\scD) = \Omega^{[p]}_{\scY'/\bA^1}(\log \mathscr{G})(-r\mathscr{G}) \cong  \Omega^{[p]}_{\scY'/\bA^1}(-r\mathscr{G})
\]
in a neighborhood of $\zeta$, where the second isomorphism holds because $\scG=\{t=0\}$ at $\zeta$ (as before $t$ is the coordinate on $\bA^1$). Thus by definition, the integer $m$ in Definition \ref{def:degeneration of form} is equal to $-r$, and $\sigma_{\scX} = \sigma_{\scY'}$ at the generic point $\zeta$ of $\mathscr{G}$. It now suffices to check that $\sigma_{\scY'}^{n/p}$ is non-vanishing at $\zeta$. But
\[
\sigma_{\scY'}^{n/p} = \varphi^*\mathrm{pr}_Y^*(\tilde{\sigma}^{n/p}),
\]
and the right hand side does not vanish at $\zeta$ if and only if $\mathrm{pr}_Y^*(\tilde{\sigma}^{n/p})$ does not vanish at $\eta = \varphi(\zeta)$. This is the case if and only if $\tilde{\sigma}^{n/p}$ does not vanish at the generic point of $D$. This proves that $(1)\Leftrightarrow (3)$. 
\end{proof}



\section{Deformation of symplectic singularities}\label{s-deformation}

In this section, we present a proof of Theorem \ref{thm:symplectic rigid} on the rigidity of symplectic singularities with respect to special test configurations. This was proved in \cite{NO-symp}, using results in \cite{Namikawa-deformation-terminal,Namikawa-deformation,Namikawa-notes-on-deformation}.
Our argument also follows essentially from a combination of the tools developed in these papers. Since the relevant
ingredients are dispersed in the literature, we include a direct proof for
the reader's convenience. In particular, we do not explicitly need the theory of Poisson deformations.


Let $(x\in X=\Spec(R),\sigma)$ be a symplectic singularity of dimension $n$, i.e. $x\in X$ is a klt singularity with a closed non-degenerate $2$-form $\sigma$. Let $D$ be a $\sigma$-admissible Koll\'ar component over $x\in X$ (Definition \ref{def:sigma-admissible}) and let $\fa_m:=\fa_m(\ord_D)$ be the valuation ideals. Let
\[
f\colon \scX:=\Spec(\cR)\to \bA^1_t
\]
be the test configuration corresponding to $D$ (Definition \ref{def:TC}), where
\begin{equation} \label{eq:extended Rees repeated}
\cR:=\bigoplus_{m\in\bZ} t^{-m} \fa_m \subseteq R[t,t^{-1}]
\end{equation} 
is the extended Rees algebra. Let $R_0:=\gr_D R = \bigoplus_{m\in\bN} \fa_m/\fa_{m+1}$ so that $\scX_0\cong \Spec(R_0)$. By the definition of $\sigma$-admissibility, there exists some $\bG_m$-equivariant relative $2$-form $\sigma_{\scX}\in H^0(\scX,\Omega^{[2]}_{\scX/\bA^1})$ such that $\sigma_t:=\sigma_{\scX}|_{\scX_t}$ is a symplectic form for every $t\in \bA^1$ and $\sigma_1=\sigma$. Let $(x_0\in X_0,\sigma_0)$ be the central fiber of the test configuration. Our goal is to show:

\begin{thm}[\cite{NO-symp}] \label{thm:symplectic rigid}
Under the above assumptions, we have a formal isomorphism 
\[
(x\in X,\sigma)^{\wedge}\cong (x_0\in X_0,\sigma_0)^{\wedge}
\]
of symplectic singularities.
\end{thm}

As a preliminary step, we show that the symplectic form $\sigma_0$ on the central fiber has positive weight with respect to the induced $\bG_m$-action. This is a consequence of the following more general observation.

\begin{lem} \label{lem:wt=log discrep}
Let $x\in X$ be a normal singularity with a good torus $\bT$-action, and let $\sigma$ be a (nonzero) $\bT$-equivariant $p$-form on $X$ with weight $\alpha\in M(\bT)$. Then for any Reeb vector $\xi\in \ft_\bR^+$, we have $A_\sigma(\xi) = \langle\alpha,\xi\rangle$. In particular, a good $\bG_m$-action on a symplectic singularity is dilating if the symplectic form is $\bG_m$-equivariant.
\end{lem}

Recall that a $\bG_m$-action on a symplectic singularity is called dilating if the symplectic form is $\bG_m$-equivariant of positive weight (\cite[Definition 1.7]{Kaledin-survey}).

\begin{proof}
By Lemma \ref{lem:A linear}, it suffices to prove this for quasi-regular Reeb vectors, thus we may assume that $\bT=\bG_m$, the $\bG_m$-action on $\sigma$ has weight $\ell\in\bZ$, and we need to show that $A_\sigma (D) = \ell$, where $D$ is the prime divisor over $X$ corresponding to this $\bG_m$-action.  Recall that $g\colon X\setminus \{x\}\to Z:=(X\setminus \{x\})/\bG_m$ has a Seifert $\bG_m$-bundle structure and $D$ can be identified with its zero section. Let $Y$ be the variety obtained by adding the zero section to $X\setminus \{x\}$; in particular $X\setminus\{x\} \cong Y\setminus D$ and we get a morphism $Y\to X$ that contracts $D$ to $x$. The $\bG_m$-invariant differential forms on $Y\setminus D$ are generated by linear combinations of wedge products of $\frac{\rd t}{t}$ and $g^*\sigma_0$, where $t$ is a $\bG_m$-equivariant defining equation of $D$ and $\sigma_0$ varies over the $1$-forms on $Z$. These are also the local free generators of $\Omega_Y^{[p]}(\log D)$ at the generic point of $D$. Since $t$ has weight $1$, we see that $t^{-\ell}\sigma$ is $\bG_m$-invariant, thus by the above discussion it defines a local section of $\Omega_Y^{[p]}(\log D)$ that does not vanish at the generic point of $D$. It follows that $A_\sigma(D)=\ell$ by the definition of log discrepancy. 

For the last statement, note that if $\sigma$ is a symplectic form, then $\sigma^{n/2}$ is a free generator of $\omega_X$, hence $A_{\sigma^{n/2}}(D)=A_X(D)>0$ by the definition of log discrepancy and the fact that $X$ is klt. It follows that the $\bG_m$-action has positive weight on $\sigma^{n/2}$ and hence on $\sigma$ as well.
\end{proof}


The crucial step in the proof of Theorem \ref{thm:symplectic rigid} is the following equivariant lifting statement. Let $\partial_t$ be the vector field $\frac{\partial}{\partial t}$ on $\bA^1_t$.

\begin{lem} \label{lem:lift vector field}
Under the assumptions of Theorem \ref{thm:symplectic rigid}, there exists some $\bG_m$-equivariant lifting $\xi\in H^0(T_{\scX})$ of the vector field $\partial_t$ to $\scX$ such that $\cL_\xi (\sigma_{\scX})=0$.
\end{lem}


\begin{proof}
To make sense of the statement, we first observe that for any vector field $\eta$ on $\bA^1$ and any lift $\xi$ of $\eta$ to some open subset $\mathscr{U}$ of $\scX$, we have $\cL_{\xi}(f^*\rd t) = f^*\cL_\eta (\rd t)\in f^*\Omega_{\bA^1}^1$, thus the Lie derivative $\cL_\xi$ on $\Omega_{\mathscr{U}/\bA^1}^{[\bullet]}$ is well defined.

By construction, the symplectic form $\sigma_0$ is $\bG_m$-equivariant. By Lemma \ref{lem:wt=log discrep}, the $\bG_m$-action has positive weight on $\sigma_0$, hence also on $\sigma_{\scX}$. From \eqref{eq:extended Rees repeated}, we also see that the $\bG_m$-action on $t$ (resp. $\partial_t$) has weight $-1$ (resp. $1$). 

Let $f\colon \scY\to \scX$ be a terminal modification \cite[Theorem 1.33]{Kol13}. Then $A_{\scX}(\scE)=1$ for every exceptional divisor $\scE$ of $f$. Since $\scX$ has a klt fiber over $0\in \bA^1$, the pair $(\scX,\scX_0)$ is plt by inversion of adjunction \cite[Theorem 5.50]{KM98}, thus $A_{\scX}(\scE) = A_{\scX,\scX_0}(\scE)+\ord_{\scE}(\scX_0)>\ord_{\scE}(\scX_0)$. It follows that $\ord_{\scE}(\scX_0) = 0$ and hence every exceptional divisor of $f$ dominates $\mathbb{A}^1$. Moreover, we get $f^*(K_{\scX}+\scX_0)=K_{\scY}+\scY_0$, which implies that $(\scY,\scY_0)$ is also plt, hence $\scY_0$ is normal by \cite[Proposition 5.51]{KM98}.

Let $\mathscr{U}\subseteq \scY$ be the ($\bG_m$-invariant) locus where the induced morphism $\scY\to \bA^1$ is smooth. Then its complement $\scY\setminus \mathscr{U}$ has codimension at least $3$ because $\scY$ is terminal and $\scY_0$ is normal. Because $f\colon \scY\to \scX$ is crepant and fiberwise birational, the pullback $\sigma_{\scY}$ of $\sigma_{\scX}$ remains symplectic on each fiber $\scY_t$. Since both $\scX$ and $\scY$ have rational singularities (see \cite[Theorem 5.22]{KM98}) and $\scX$ is affine, we have $H^i(\scY,\cO_{\scY})=H^i(\scX,\cO_{\scX})=0$ for $i=1,2$. Since $\scY$ is Cohen-Macaulay by \cite[Theorem 5.10]{KM98}, from the long exact sequence of local cohomology we also have $H^i(\mathscr{U},\cO_{\mathscr{U}})=0$ for $i=1,2$.

Consider the truncated algebraic de Rham complex $\Omega_{\mathscr{U}/\bA^1}^{\bullet\ge 1}$ on $\mathscr{U}$ (by convention the complex starts in degree $1$):
\[
\Omega_{\mathscr{U}/\bA^1}^1 \xrightarrow{\rd} \Omega_{\mathscr{U}/\bA^1}^2 \xrightarrow{\rd} \dots \xrightarrow{\rd} \Omega_{\mathscr{U}/\bA^1}^n . 
\]
It is well known that this complex controls the deformation theory of smooth symplectic varieties. The $\bG_m$-action on $\scX$ lifts to a $\bG_m$-action on $\mathscr{U}$, hence $\bG_m$ also naturally acts on $\Omega_{\mathscr{U}/\bA^1}^{\bullet\ge 1}$ and its hypercohomology. 

We show that the obstruction to the equivariant lifting problem of $\partial_t$ in the lemma's statement is a class of positive weight in $\mathbb{H}^2(\mathscr{U},\Omega_{\mathscr{U}/\bA^1}^{\bullet\ge 1})$. To see this, let $(\mathscr{U}_i)_{i\in I}$ be a finite open covering of $\mathscr{U}$ by $\bG_m$-invariant affine subsets (this is possible by \cite{Sumihiro}). Since each term in the de Rham complex is quasi-coherent, the hypercohomology groups $\mathbb{H}^\bullet(\mathscr{U},\Omega_{\mathscr{U}/\bA^1}^{\bullet\ge 1})$ can be computed by the total \v{C}ech complex $(\mathscr{C}^\bullet,\delta)$ \cite[\href{https://stacks.math.columbia.edu/tag/01FP}{Tag 01FP}]{stacks-project} (see \cite[Proof of Proposition 8]{Namikawa-deformation-terminal}), whose degree $m$ term is
\[
\mathscr{C}^m = \bigoplus_{p+q=m,\,q\ge 1} \prod_{i_0<\dots<i_p} H^0(\mathscr{U}_{i_0\dots i_p},\Omega_{\mathscr{U}/\bA^1}^q). 
\]
Since $\mathscr{U}\to \bA^1$ is smooth, on each $\bG_m$-invariant affine open $\mathscr{U}_i$ we can lift $\partial_t$ to some $\bG_m$-equivariant vector fields $\xi_i$ of weight $1$ (i.e. same weight as $\partial_t$). On each $\mathscr{U}_{ij}$ we then get a vertical vector field $\xi_{ij}=\xi_i - \xi_j \in T_{\mathscr{U}/\bA^1}$. Since contraction $\iota_{\bullet}(\sigma_{\mathscr{U}})$ with the relative symplectic form gives an isomorphism $T_{\mathscr{U}/\bA^1}\cong \Omega^1_{\mathscr{U}/\bA^1}$, and since $\sigma_{\mathscr{U}}$ has positive weight under the $\bG_m$-action, the collection
\[
\theta:=\big( \iota_{\xi_{ij}}(\sigma_{\mathscr{U}}), \cL_{\xi_i} \sigma_{\mathscr{U}}\big) \in \mathscr{C}^2
\]
gives a $\bG_m$-equivariant element with positive weight. 

We show that $\delta(\theta)=0$. This amounts to three equalities:
\begin{align*}
\rd (\cL_{\xi_i} \sigma_{\mathscr{U}}) = \cL_{\xi_i} (\rd \sigma_{\mathscr{U}}) & = 0, \\
\cL_{\xi_{ij}} \sigma_{\mathscr{U}} - \rd\, \iota_{\xi_{ij}}(\sigma_{\mathscr{U}}) = \iota_{\xi_{ij}}(\rd  \sigma_{\mathscr{U}}) & = 0, \\
\iota_{\xi_{ij}}(\sigma_{\mathscr{U}}) +  \iota_{\xi_{jk}}(\sigma_{\mathscr{U}}) + \iota_{\xi_{ki}}(\sigma_{\mathscr{U}}) & = 0 \quad (\mbox{on }\mathscr{U}_{ijk}).
\end{align*}
The first equality holds because $\rd \sigma_{\scX}=0$, the second follows from Cartan's formula $\cL_\xi = \rd\, \iota_\xi+\iota_\xi\, \rd$, while the third holds as $\xi_{ij}+\xi_{jk}+\xi_{ki}=0$ by construction. Thus we get a positive-weighted cohomology class $[\theta]\in H^2(\mathscr{C}^\bullet)\cong \mathbb{H}^2(\mathscr{U},\Omega_{\mathscr{U}/\bA^1}^{\bullet\ge 1})$. Different choices of lifting have the form $\xi'_i=\xi_i+\eta_i$ for some $\eta_i\in H^0(T_{\mathscr{U}_i/\bA^1})$, and it is straightforward to check using Cartan's formula that the corresponding element $\theta'\in \mathscr{C}^2$ satisfies $\theta'-\theta = \delta(\tau)$ where $\tau:=\big(\iota_{\eta_i}(\sigma_{\mathscr{U}})\big)\in \mathscr{C}^1$. From these we conclude that the cohomology class $[\theta]$ is independent of the choice of the lifting $\xi_i$. 

Next we show that $[\theta]=0$ if and only if $\partial_t$ equivariantly lifts to some globally defined vector field $\xi$  such that $\cL_{\xi} \sigma_{\mathscr{U}}=0$. Indeed, $[\theta]=0$ implies that we can find lifts $\xi_i$ such that $\iota_{\xi_{ij}}(\sigma_{\mathscr{U}}) =0$ and $\cL_{\xi_i} \sigma_{\mathscr{U}}=0$. Since contraction $\iota_{\bullet}(\sigma_{\mathscr{U}})$ with the relative symplectic form gives an isomorphism $T_{\mathscr{U}/\bA^1}\cong \Omega^1_{\mathscr{U}/\bA^1}$, this implies $\xi_{ij}=0$. Therefore, the $\xi_i$'s glue to a vector field $\xi$ and we have $\cL_{\xi} \sigma_{\mathscr{U}}=0$. The reverse direction is clear. 

It remains to show that $[\theta]=0$. In fact, we shall prove the stronger statement that the weight-positive part of $\mathbb{H}^2(\mathscr{U},\Omega_{\mathscr{U}/\bA^1}^{\bullet\ge 1})$ is trivial, {\it cf.} \cite[Lemma 21]{Namikawa-deformation-terminal}. For this we follow the argument in {\it ibid}. Since $H^i(\mathscr{U},\cO_{\mathscr{U}})=0$ for $i=1,2$, the long exact sequence associated to the distinguished triangle $\Omega_{\mathscr{U}/\bA^1}^{\bullet\ge 1} \to \Omega_{\mathscr{U}/\bA^1}^{\bullet}\to \cO_{\mathscr{U}} \to \Omega_{\mathscr{U}/\bA^1}^{\bullet\ge 1}[1]$ gives
\[
\mathbb{H}^2(\mathscr{U},\Omega_{\mathscr{U}/\bA^1}^{\bullet\ge 1}) \cong \mathbb{H}^2(\mathscr{U},\Omega_{\mathscr{U}/\bA^1}^{\bullet}).
\]
But the de Rham complex $\Omega_{\mathscr{U}/\bA^1}^{\bullet}$ is $\bG_m$-equivariantly quasi-isomorphic to $f^{-1}\cO_{\bA^1}$, and the weight positive part of the latter is trivial (because the $\bG_m$-action on $\bA^1$ has negative weight). Therefore, the weight-positive part of $\mathbb{H}^2(\mathscr{U},\Omega_{\mathscr{U}/\bA^1}^{\bullet\ge 1})$ is also trivial. This concludes the proof.
\end{proof}

\begin{proof}[Proof of Theorem \ref{thm:symplectic rigid}]
By Lemma \ref{lem:lift vector field}, we can fix a $\bG_m$-equivariant lifting $\xi$ of $\partial_t$ with $\cL_\xi (\sigma_{\scX})=0$. This corresponds to a derivation $\partial_\xi\colon \cR\to \cR$ such that $\partial_\xi(\cR_m)\subseteq \cR_{m+1}$ for all $m\in\bZ$ and $\partial_\xi(t) = 1$. The $\bG_m$-action on $\scX$ also generates a vector field $\xi'$ whose image in $\bA^1$ is $-t\frac{\partial}{\partial t}$; as a derivation, it acts by $\partial_{\xi'}=m\cdot \mathrm{id}$ on $\cR_m$. Then $\eta = t\xi+\xi'$ satisfies 
\[
\partial_\eta(t)= t\cdot \partial_\xi(t)-t=0 \, ,
\]
hence it is a $\bG_m$-invariant vertical vector field (i.e. $\eta\in H^0(T_{\scX/\bA^1})$). In particular, by restricting to the general fiber it defines a derivation $\partial_\eta$ of $R$, which can be identified with the derivation $t\partial_\xi$ on $\cR_0 = R$. The condition $\partial_\xi(\cR_m)\subseteq \cR_{m+1}$ is then equivalent to $(\partial_\eta - m)\fa_m\subseteq \fa_{m+1}$. In particular, $\partial_\eta(\fa_m)\subseteq \fa_m$. Note that every $\fa_m$ is $\fm_x$-primary and $\fa_1=\fm_x$ since the Koll\'ar component $D$ is centered at $x$. Hence the vector field $\eta$ fixes the canonical section $\overline{\{x\}\times (\bA^1\setminus\{0\})}$ of the test configuration and  $\partial_\eta$ extends to a derivation of the formal completion $\widehat{R}$. By induction, we see that
\begin{equation} \label{eq:derivation}
\frac{(\partial_\eta-m)\cdots(\partial_\eta-m-i+1)}{i!}\cdot \fa_m\subseteq \fa_{m+i}    
\end{equation}
for all $i\in\bN$ (by convention, the operator on the left hand side is the identity when $i=0$). Let 
\[
\widehat{R}_m:=\{s\in \widehat{R}\,|\,\partial_\eta(s) = ms\}.
\]
Then $\widehat{R}_m\cdot \widehat{R}_\ell \subseteq \widehat{R}_{m+\ell}$, as for any $s\in \widehat{R}_m $ and $s'\in  \widehat{R}_\ell$, we have 
\[
\partial_\eta(s\cdot s')=\partial_\eta s \cdot s'+s\cdot \partial_\eta s' =ms\cdot s'+\ell s\cdot s'=(m+\ell) s\cdot s'\, .
\]   

We claim that $\widehat{R}_m\subseteq \widehat{\fa}_m:=\fa_m \widehat{R}$ and the induced map 
\[
\varphi_m\colon \widehat{R}_m\to \widehat{\fa}_m/\widehat{\fa}_{m+1}\cong \fa_m/\fa_{m+1} = \gr^m_D R
\]
is an isomorphism. Essentially this holds because $\partial_\eta-m$ is the zero map on $\gr_D^m R$, and is invertible on $\gr^{\ell}_D R$ whenever $\ell\neq m$. More concretely, for any $s\in \fa_m$, let $\bar{s}\in \fa_m/\fa_{m+1}$ be its reduction, and set
\[
\tilde{s}:=\sum_{i=0}^\infty \frac{(-1)^i}{i!} (\partial_\eta-m)\cdots(\partial_\eta-m-i+1)\cdot s \in \widehat{\fa}_m \, ,
\]
which is well-defined by \eqref{eq:derivation}. Then $\tilde{s}-s\in \widehat{\fa}_{m+1}$ and
\begin{align*}
(\partial_\eta-m)(\tilde{s}) & = \sum_{i=0}^\infty \frac{(-1)^i}{i!} (\partial_\eta-m)\cdots(\partial_\eta-m-i+1) \Big((\partial_\eta-m-i) +i\Big)\cdot s \\
& = \sum_{i=0}^\infty \frac{(-1)^i}{i!} \prod_{j=0}^i (\partial_\eta-m-j) \cdot s + \sum_{i=0}^\infty \frac{(-1)^{i-1}}{i!} \prod_{j=0}^i (\partial_\eta-m-j)\cdot s = 0.
\end{align*}
Therefore, $\tilde{s} \in \widehat{R}_m$ and $\varphi_m(\tilde{s})=\bar{s}\in \fa_m/\fa_{m+1}$. This proves that $\varphi_m$ is surjective. On the other hand, if $0\neq s\in \widehat{R}_m$ then there is some integer $\ell\in \bN$ (in fact $\ell = \ord_D(s)$) such that $s\in \widehat{\fa}_\ell \setminus \widehat{\fa}_{\ell+1}$. In particular, its reduction $\bar{s}\in \widehat{\fa}_\ell/\widehat{\fa}_{\ell+1}\cong \gr^\ell_D R$ is nonzero and we have $\partial_\eta (\bar{s})=\ell \bar{s}$. But we also have $\partial_\eta(s)=ms$, hence $\ell = m$. This gives $\widehat{R}_m\subseteq \widehat{\fa}_m$ as well as the injectivity of $\varphi_m$. The claim is thus proved. 

Using the isomorphism $\varphi_m$ we may identify $R_0=\bigoplus_{m\in\bN} \gr_D^m R$ with the subring $\bigoplus_{m\in\bN} \widehat{R}_m$ of $\widehat{R}$ whose completion is also $\widehat{R}$. Moreover, $\widehat{\fa}_m = \widehat{\bigoplus}_{\ell\ge m} \widehat{R}_\ell$. As a result, we get a formal $\bG_m$-action on $\widehat{R}$ with corresponding Koll\'ar component $D$ and $\widehat{R}\cong \widehat{R_0}$. It remains to show that the induced symplectic form $\widehat{\sigma}$ on $\widehat{X}_x = \Spec(\widehat{R})$ is equivariant under the above $\bG_m$-action. Equivalently, we need to show that $\cL_\eta(\sigma_{\scX})=\ell \cdot \sigma_\scX$ for some integer $\ell$. This follows because by our choice of the vector fields $\xi,\xi'$, we have $\cL_{\xi'}(\sigma_{\scX})=\ell \cdot \sigma_\scX$ where $\ell = \wt(\sigma_{\scX})$, while $\cL_{t\xi}(\sigma_{\scX}) = t\cL_{\xi}(\sigma_{\scX}) = 0 \in H^0(\Omega^{[2]}_{\scX/\bA^1})$ by Cartan's formula. This completes the proof.
\end{proof}

\section{Stability and non-degeneracy}\label{s-canonical}

In this section, we give a lower bound on the log discrepancy of a Koll\'ar component with respect to some differential form, in terms of the stability threshold of the corresponding log Fano pair. We also prove that stable degeneration preserves non-degenerate forms. One of the key observations is that if $Y\to X$ is the plt blowup of the Koll\'ar component $D$, then the sheaf $\Omega_Y^{[1]}(\log D)|_D$ (interpreted in the appropriate orbifold or stacky sense), which already appears in Proposition \ref{prop:non-degenerate criterion}, is the canonical extension of $\cO_D$ by the orbifold cotangent sheaf of the induced log Fano pair $(D,\Delta_D)$, and the maximal slope of this sheaf is closely related, on one hand, to the stability threshold of $(D,\Delta_D)$ and, on the other hand, to the vanishing order of differential forms along $D$.




\subsection{Orbifold (co)tangent sheaf and stability of canonical extension}

We start by recalling the definition of the orbifold (co)tangent sheaf and its canonical extension, and prove some general properties of these sheaves. The main reference is \cite{CKT-hyperbolicity} (see also \cites{GT-orbifold-stab,Li-extension-stab,Dai-MY,Dai-stability}).

\begin{defn}[{\cite[Definition 2.37]{CKT-hyperbolicity}}]
Let $(X,\Delta)$ be a pair. A quasi-finite dominant morphism $f\colon Y \to X$ from a normal variety is called \emph{adapted} to $(X,\Delta)$ if $f^*\Delta$ has integral coefficients.
\end{defn}

By \cite[Proposition 2.38]{CKT-hyperbolicity}, every pair $(X,\Delta)$ has an adapted (abelian) Galois cover. 


\begin{defn}[{\cite[Definition 3.11]{CKT-hyperbolicity}}]
Let $(X,\Delta)$ be a pair and let $f\colon Y \to X$ be a quasi-finite cover adapted to $(X,\Delta)$. 
We define the \emph{sheaf of adapted reflexive differentials} (or \emph{adapted cotangent sheaf}) $\Omega^{[1]}_{(X,\Delta,f)}$ on $Y$ as the reflexive hull of 
\[
\big(f^{[*]}\Omega_X^{1}(\log \lfloor \Delta\rfloor)\big)(f^*\{\Delta\}) \cap \Omega^{[1]}_Y(\log \Delta_f) \subseteq \Omega^{[1]}_Y(\log \Delta_f)(f^*\{\Delta\}),
\]
where $\Delta_f:=(f^*\lfloor\Delta\rfloor)_{\mathrm{red}}$ and $f^{[*]}$ is the reflexive pullback. The adapted tangent sheaf is defined to be 
\[
T_{(X,\Delta,f)}:=\left(\Omega^{[1]}_{(X,\Delta,f)}\right)^* .
\]

More concretely, there is a big open subset $U\subseteq X$ over which $\Supp(\Delta)$ is smooth and at any $y\in f^{-1}(U)$ the map $f$ takes the form
\[
(y_1,\dots,y_n)\mapsto (x_1,\dots,x_n)=(y_1^m,y_2,\dots,y_n)
\]
in local coordinates for some positive integer $m$; moreover, if $x=f(y)\in \Supp(\Delta)$ then $\Supp(\Delta) = \{x_1=0\}$. Then $\Omega^{[1]}_{(X,\Delta,f)}$ is the unique reflexive subsheaf of $\Omega^{[1]}_Y(\log \Delta_f)(f^*\{\Delta\})$ that is locally generated by
\begin{equation} \label{eq:local generator}
f^*\left(\frac{\rd x_1}{x_1^{\lambda_1}}\right), f^*\rd x_2,\dots,f^*\rd x_n
\end{equation}
over $f^{-1}(U)$, where $\lambda_1$ is the coefficient of $\Delta$ in this local chart. One can check that the subsheaf generated by \eqref{eq:local generator} is independent of the choice of local coordinates (e.g. if $u$ is a unit, then replacing $x_1$ by $ux_1$ does not change the subsheaf). Similarly, over $f^{-1}(U)$ the adapted tangent sheaf $T_{(X,\Delta,f)}$ is locally generated (as a subsheaf of $f^*T_X$) by
\[
f^*\left(x_1^{\lambda_1}\frac{\partial}{\partial x_1}\right), f^*\left(\frac{\partial}{\partial x_2}\right),\dots,f^*\left(\frac{\partial}{\partial x_n}\right).
\]
\end{defn}
 
We next define the canonical extensions of the adapted (co)tangent sheaf ({\it cf.} \cite[Section 2.3]{Li-extension-stab} or \cite[Section 2.2.2]{Dai-MY}).

\begin{defn}
Let $(X,\Delta)$ be a pair and let $f\colon Y \to X$ be a quasi-finite cover adapted to $(X,\Delta)$. Note that $f^{[*]}\Omega^1_X \subseteq \Omega^{[1]}_{(X,\Delta,f)}$. Consider the composition
\[
\Pic(X)=H^1(X,\cO_X^*)\xrightarrow{c_1} H^1(X,\Omega_X^{[1]})\xrightarrow{f^*} H^1(Y,\Omega^{[1]}_{(X,\Delta,f)}) \, ,
\]
where the first map is induced by the composition of $\rd \log \colon \cO_X^*\to \Omega_X^1$ and $\Omega^1_X\to \Omega_X^{[1]}$. For any $\bR$-Cartier divisor $L$ on $X$, we denote by $E_{(X,\Delta,f),L}$ the extension of $\cO_Y$ by $\Omega^{[1]}_{(X,\Delta,f)}$ with extension class $f^*c_1(L) \in H^1(Y,\Omega^{[1]}_{(X,\Delta,f)})$. We note that its isomorphism class only depends on $(\bR\setminus\{0\})\cdot [L]\subseteq \Pic(X)_\bR$. We also set $E_{X,L}:=E_{(X,0,\mathrm{id}),L}$. If $K_X+\Delta$ is $\bQ$-Cartier, the \emph{canonical extension} of $\cO_Y$ by $\Omega^{[1]}_{(X,\Delta,f)}$ is defined to be $E_{(X,\Delta,f)}:=E_{(X,\Delta,f),K_X+\Delta}$.
\end{defn}

\begin{rem}
The definitions above are functorial in the sense that if $f\colon Y\to X$ is a quasi-finite cover adapted to $(X,\Delta)$ and $g\colon Z\to Y$ is another quasi-finite cover, then the natural maps 
\[
g^{[*]}\Omega^{[1]}_{(X,\Delta,f)} \to \Omega^{[1]}_{(X,\Delta,g\circ f)} \quad \mathrm{and} \quad g^{[*]}E_{(X,\Delta,f),L} \to E_{(X,\Delta,g\circ f),L}
\]
are isomorphisms. As such, we can also define adapted (co)tangent sheaves $\Omega^{[1]}_{(\cX,\Delta,f)}$, $T_{(\cX,\Delta,f)}$ and the extensions $E_{(\cX,\Delta,f),L}$ on any adapted quasi-finite cover $f\colon \cY\to \cX$ of Deligne-Mumford stacks. In particular, if $\Delta=\sum_{i=1}^m \frac{a_i}{b_i} \Delta_i$ is a decomposition into irreducible components (where $a_i,b_i\in \bN$, $\gcd(a_i,b_i)=1$) and each component $\Delta_i$ is Cartier (e.g. if $(X,\Supp(\Delta))$ is SNC), then we may form the fiber product 
\[
\cX:=\sqrt[b_1]{\Delta_1}\times_X \dots \times_X \sqrt[b_m]{\Delta_m}
\]
of the corresponding root stacks (see \cite[Appendix B]{AGV-GW-DM} and \cite[Section 2]{Cad-root-stack}). In this case a quasi-finite cover $f\colon Y\to X$ is adapted to $(X,\Delta)$ if and only if it factors through $\cX$, and the sheaves $\Omega^{[1]}_{(X,\Delta,f)}$, $T_{(X,\Delta,f)}$ and $E_{(X,\Delta,f),L}$ on $Y$ are the reflexive pullbacks of the corresponding sheaves on the normalization of $\cX$. The existence of such a stack $\cX$ will streamline many of our constructions by just considering the geometric objects on $\cX$.
\end{rem}

The following geometric description of the extension sheaf will be useful later.

\begin{lem} \label{lem:log cotangent restrict to can ext}
Let $\cX$ be a smooth separated Deligne-Mumford stack and $\cD$ a smooth divisor on $\cX$. Then
\[
\Omega_{\cX}^1(\log \cD)|_{\cD} \cong E_{\cD,N_{\cD/\cX}}.
\]
\end{lem}

\begin{proof}
This should be well-known to experts. The exact sequence 
\[
0\to \Omega^1_{\cX}\to \Omega^1_{\cX}(\log \cD)\xrightarrow{\mathrm{res}}\cO_{\cD}\to 0
\]
on $\cX$ restricts to an exact sequence
\[
\Omega_{\cX}^1|_{\cD}\to \Omega^1_{\cX}(\log \cD)|_{\cD}\to \cO_{\cD}\to 0
\]
on $\cD$. The first map factors through $\Omega_{\cD}^1$, and we get an induced exact sequence (exactness can be checked in local coordinates)
\[
0\to \Omega_{\cD}^1\to \Omega^1_{\cX}(\log \cD)|_{\cD}\to \cO_{\cD}\to 0.
\]
A direct computation using \v{C}ech cohomology then shows that the corresponding extension class is $c_1(N_{{\cD}/\cX})\in H^1(\cD,\Omega_{\cD}^1)$. (Here we note that as we work over $\bC$, any DM stack is tame, see e.g. \cite[Section 3]{AOV-tame-stack}, thus the cohomology of a quasi-coherent sheaf can be computed on the coarse moduli space and therefore coincides with \v{C}ech cohomology.) 
\end{proof}

\begin{lem} \label{lem:orbifold cotangent}
Let $(X,D)$ be a plt pair whose boundary divisor $D$ is reduced, and let $\Delta_D$ be the $\bQ$-divisor on $D$ obtained by adjunction: $(K_X+\Delta)|_D=K_D+\Delta_D$. Assume that $D$ is $\bQ$-Cartier. Let $f\colon \cX\to X$ be the index one covering stack with respect to $D$, and let $\cD:=f^*D$. Then 
\begin{enumerate}
    \item both $D$ and $\cD$ are normal,
    \item the induced cover $g=f|_{\cD}\colon \cD\to D$ is adapted to $(D,\Delta_D)$, 
    \item we have 
    \[
    \Omega^{[1]}_{(D,\Delta_D,g)} \cong \Omega^{[1]}_{\cD} \quad \mathrm{and} \quad E_{(D,\Delta_D,g),N_{D/X}}\cong E_{\cD,N_{\cD/\cX}}.
    \] 
\end{enumerate}
\end{lem}

\begin{proof}
By construction, $\cD$ is a Cartier divisor on $\cX$. Since $f$ is quasi-\'etale, we have $f^*(K_X+D)=K_{\cX}+\cD$. As $(X,D)$ is plt, we deduce that $(\cX,\cD)$ is also plt \cite[Proposition 5.20]{KM98}, hence (1) holds by \cite[Proposition 5.51]{KM98}. By adjunction we get $K_{\cD} = g^*(K_D+\Delta_D)$, thus $g^*\Delta_D$ is an integral divisor on $\cD$ which gives (2).

By definition, both $\Omega^{[1]}_{(D,\Delta_D,g)}$ and $\Omega^{[1]}_{\cD}$ are reflexive subsheaves of $\Omega^{[1]}_{\cD}(f^*\Delta_D)$. To prove the first isomorphism in (3), it suffices to check that they agree over every codimension one point (call it $\xi$) of $D$. Since $(X,D)$ is plt, \'etale locally at $\xi$ it is a cyclic quotient $\big((\bA^2,\{x=0\})/\mu_r\big)\times \bA^{n-2}$ with $n=\dim X$ (see e.g. \cite[Section 3.3]{Kol13}), thus (ignoring the smooth factor $\bA^{n-2}$) we reduce to the case $\cD=[\bA^1/\mu_r]$, with coarse moduli space $D=\bA^1/\mu_r\cong \bA^1$, and $\Delta_D = (1-\frac{1}{r})[0]$ (see \cite[Theorem 3.36]{Kol13}), where the equality $\Omega^{[1]}_{(D,\Delta_D,g)}=\Omega^{[1]}_{\cD}$ holds by a direct local computation. This gives the first part of (3).

By definition, $E_{(D,\Delta_D,g),N_{D/X}}$ is the extension of $\cO_{\cD}$ by $\Omega^{[1]}_{(D,\Delta_D,g)}$ with extension class $g^*c_1(N_{D/X})$. As $\Omega^{[1]}_{(D,\Delta_D,g)} \cong \Omega^{[1]}_{\cD}$ and $f^*D = \cD$, we see that $E_{(D,\Delta_D,g),N_{D/X}}$ is also the extension of $\cO_{\cD}$ by $\Omega^{[1]}_{\cD}$ with extension class $c_1(N_{\cD/\cX})$, thus the second part of (3) holds as well.
\end{proof}

\begin{defn}[{\emph{cf.} \cite[Definition 2.23]{GT-orbifold-stab}, \cite[Example 2.21]{Dai-MY}}] \label{def:orbifold slope}
The orbifold cotangent (resp. tangent) sheaf $\Omega^{[1]}_X(\log \Delta)$ (resp. $T_X(-\log \Delta)$) of a pair $(X,\Delta)$ is defined to be the collection $\{\Omega^{[1]}_{(X,\Delta,f)}\}$ (resp. $T_{(X,\Delta,f)}$) as $f$ varies among the finite covers of $X$ that are adapted to $(X,\Delta)$. Similarly, when $K_X+\Delta$ is $\bQ$-Cartier, we define the canonical extension $E_{X,\Delta}$ of $\cO_X$ by $\Omega^{[1]}_X(\log \Delta)$ to be the collection $\{E_{(X,\Delta,f)}\}$. 

Let $H$ be an ample $\bQ$-Cartier divisor on $X$ and $\cE$ be one of $\Omega^{[1]}_X(\log \Delta)$, $T_X(-\log \Delta)$, or $E_{X,\Delta}$. While $\cE$ is not a sheaf on $X$, its (reflexive) pullback $f^{[*]}\cE$ along a finite cover $f\colon Y\to X$ adapted to $(X,\Delta)$ is a well-defined (reflexive) sheaf on $Y$. We say that $\cE$ is slope semistable with respect to $H$ if $f^*\cE$ is slope semistable with respect to $f^*H$. By \cite[Lemma 3.2.2]{HL-book-sheaf}, this definition is independent of the cover $f$. For the same reason, we can define the slope and maximal slope of $\cE$ with respect to $H$ as
\[
\mu_{H}(\cE):=\frac{\mu_{f^*H}(f^{[*]}\cE)}{\deg f}\quad \mathrm{and}\quad\mu_{\max,H}(\cE):=\frac{\mu_{\max,f^*H}(f^{[*]}\cE)}{\deg f},
\]
where again the right hand sides are independent of $f$.
\end{defn}

Next, we need a logarithmic analogue of \cite{DGP-Q-Fano-decomp} (see also \cite{Li-extension-stab,Dai-stability}). We shall use it to control the maximal slope of the adapted canonical extension for log Fano pairs that are not necessarily K-semistable. This is done by adding a general boundary to make the pair K-semistable and comparing the corresponding adapted canonical extensions. 

\begin{thm} \label{thm:can ext ss}
Let $(X,\Delta)$ be a K-semistable log Fano pair. Then the canonical extension $E_{X,\Delta}$ of $\cO_X$ by the orbifold cotangent sheaf $\Omega^{[1]}_X(\log \Delta)$ is slope semistable with respect to $-(K_X+\Delta)$.
\end{thm}

\begin{proof}
See Theorem \ref{thm:can ext appendix}.
\end{proof}

\begin{lem} \label{lem:mu_max bound by delta}
Let $(X,\Delta)$ be a log Fano pair of dimension $n$ and let $H=-(K_X+\Delta)$. Then $\mu_H(E_{X,\Delta})=-\frac{1}{n+1}(H^n)$ and
\begin{equation} \label{eq:mu_max bound by delta}
\mu_{\max,H}(E_{X,\Delta})\le -\frac{\min\{1,\delta(X,\Delta)\}\cdot (H^n)}{n+1}.
\end{equation}
\end{lem}

\begin{proof}
The equality $\mu_H(E_{X,\Delta})=-\frac{1}{n+1}H^n$ is a direct consequence of \cite[Corollary 3.9]{CKT-hyperbolicity}. If $(X,\Delta)$ is K-semistable, then $E_{X,\Delta}$ is slope semistable with respect to $H$ by Theorem \ref{thm:can ext ss}, thus 
\[
\mu_{\max,H}(E_{X,\Delta}) = \mu_{H}(E_{X,\Delta}) = -\frac{1}{n+1}(H^n)
\]
and \eqref{eq:mu_max bound by delta} obviously holds. 

Assume that $(X,\Delta)$ is not K-semistable and let $\delta=\delta(X,\Delta)$. Then $0<\delta<1$ and $\delta\in \bQ$ by \cite[Theorem 5.1]{LXZ-HRFG}. By \cite[Theorem 1.8]{LXZ-HRFG}, there exists some general $\bQ$-divisor $0\le G\sim_{\bQ} H$ such that the log Fano pair $(X,\Gamma:=\Delta+(1-\delta)G)$ is K-semistable. 

Let $f\colon Y\to X$ be a finite cover adapted to $(X,\Gamma)$. Since $G$ is general, $f$ is also adapted to $(X,\Delta)$. By construction, $\Omega^{[1]}_{(X,\Delta,f)}$ is a subsheaf of $\Omega^{[1]}_{(X,\Gamma,f)}$. Since $K_X+\Gamma\sim_{\bQ} -\delta H$ is proportional to $K_X+\Delta$, we deduce that the canonical extension $E_{(X,\Gamma,f)}$ can be obtained as the cokernel of 
\[
\Omega^{[1]}_{(X,\Delta,f)} \to E_{(X,\Delta,f)} \oplus \Omega^{[1]}_{(X,\Gamma,f)},
\]
and we have a commutative diagram
\[
\begin{tikzcd}
0 \arrow[r] & \Omega^{[1]}_{(X,\Delta,f)} \arrow[r] \arrow[d, hook] & E_{(X,\Delta,f)} \arrow[r] \arrow[d, hook] & \cO_Y \arrow[r] \arrow[d, equal] & 0 \\
0 \arrow[r] & \Omega^{[1]}_{(X,\Gamma,f)} \arrow[r] & E_{(X,\Gamma,f)} \arrow[r] & \cO_Y \arrow[r] & 0
\end{tikzcd}
\]
whose middle column is injective by the five lemma. By Theorem \ref{thm:can ext ss}, we know that $E_{X,\Gamma}$ is slope semistable with respect to $-(K_X+\Gamma)$, hence also slope semistable with respect to $H$.  It follows that
\[
\mu_{\max,H}(E_{X,\Delta})\le \mu_{\max,H}(E_{X,\Gamma}) = \mu_H (E_{X,\Gamma}) = \frac{1}{n+1}(K_X+\Gamma\cdot H^{n-1}) = -\frac{\delta}{n+1}(H^n).
\]
This proves \eqref{eq:mu_max bound by delta}.
\end{proof}

\subsection{Log discrepancy estimate and non-degeneracy}

In this subsection, we prove that stable degeneration preserves non-degenerate forms. The last missing ingredient is the following log discrepancy estimate that refines Theorem \ref{thm:log disc>=0}.

\begin{thm}\label{thm:estimate dis}
Let $\sigma$ be a non-zero (but not necessarily non-degenerate) reflexive $p$-form on a klt singularity $x\in X$ of dimension $n$.
\begin{enumerate}
\item Let $D$ be a Koll\'ar component over $x$. Then
\[
\min\{\delta(D,\Delta_D),1\} \cdot \frac{A_X(D)}{n}\le \frac{A_{\sigma}(D)}{p} \, .
\]
\item Let $v$ be the minimizer of the normalized volume function. Then 
\[
\frac{A_X(v)}{n}\le \frac{A_{\sigma}(v)}{p} \, .
\] 
\end{enumerate}
\end{thm}

\begin{proof}
We first prove (1). Let $\delta := \min\{\delta(D,\Delta_D),1\}$ and $r := A_\sigma(D)$. Let $Y\to X$ be the plt blowup of $D$. Consider the index one covering stack $\varphi\colon \cY\to Y$ with respect to $D$ and let $\cD=\varphi^*D$. By pulling back through $\varphi$, we shall view the strict transform $\tilde{\sigma}$ of $\sigma$ (Definition \ref{def:sigma-admissible}) as a section of $\Omega_{\cY}^{[p]}(\log \cD)(-r\cD)$ that does not vanish at the generic point of $\cD$. By restriction, we get a non-zero section $\tilde{\sigma}|_{\cD}$ of $\Omega_{\cY}^{[p]}(\log \cD)(-r\cD)|_{\cD}$. In particular,
\[
H^0(\cD,\Omega_{\cY}^{[p]}(\log \cD)(-r\cD)|_{\cD}) \neq 0
\]
and hence
\[
\mu_{\max}(\Omega^{[p]}_{\cY}(\log \cD)(-r\cD)|_{\cD}) \ge 0,
\]
where the slope is computed with respect to $-(K_D+\Delta_D)$ and defined in a way analogous to Definition \ref{def:orbifold slope} by considering finite covers of $D$ that factor through $\cD$. Since
\[
K_D+\Delta_D = (K_Y+D)|_D\sim_\bQ (\pi^*K_X + A_{X}(D)\cdot D)|_D\sim_\bQ A_{X}(D)\cdot D|_D,
\]
we also have
\begin{align*}
\mu_{\max}(\Omega^{[p]}_{\cY}(\log \cD)(-r\cD)|_{\cD}) & \le p\cdot \mu_{\max}(\Omega^{[1]}_{\cY}(\log \cD)|_{\cD})+ (-rD|_D\cdot (-K_D-\Delta_D)^{n-2}) \\
& = p\cdot \mu_{\max}(\Omega^{[1]}_{\cY}(\log \cD)|_{\cD})+\frac{r}{A_{X}(D)}(-K_D-\Delta_D)^{n-1}\,.
\end{align*}
By Lemmas \ref{lem:log cotangent restrict to can ext} and \ref{lem:orbifold cotangent}, we have 
\[
\Omega_{\cY}^{[1]}(\log \cD)|_{\cD} \cong E_{\cD,N_{\cD/\cY}}= E_{(D,\Delta_D,\varphi|_{\cD}),N_{D/Y}}
\]
(since the sheaves in question are reflexive, it suffices to show that they agree on the big open set where both $\cY$ and $\cD$ are smooth, thus Lemma \ref{lem:log cotangent restrict to can ext} applies). Since $K_D+\Delta_D$ is proportional to $D|_D$, we see that $E_{(D,\Delta_D,\varphi|_{\cD}),N_{D/Y}}$ is isomorphic to the canonical extension $E_{D,\Delta_D,\varphi}$, thus
 \[
 \mu_{\max}(\Omega^{[1]}_{\cY}(\log \cD)|_{\cD})=\mu_{\max} (E_{D,\Delta_D}) \, .
 \]
Therefore, combining the inequalities above we conclude 
 \[
A_{\sigma}(D)= r\ge \frac{-p\cdot \mu_{\max}(E_{D,\Delta_D}) \cdot A_{X}(D)}{(-K_D-\Delta_D)^{n-1}} \, .
\]
By Lemma \ref{lem:mu_max bound by delta}, we also have
\[
\mu_{\max}(E_{D,\Delta_D})\le -\frac{\delta}n(-K_D-\Delta_D)^{n-1},
\]
hence $A_{\sigma}(D)\ge \frac{\delta \cdot p}{n} A_X(D)$ which gives (1). 

For (2), let $(Y,E)\to X$ be a log smooth model adapted to $v$. By Theorem \ref{thm:delta->1}, for any $\varepsilon > 0$ there exists a divisorial valuation $\lambda \cdot \ord_D \in \QM(Y,E)$ near $v$ such that $D$ is a Koll\'ar component and $\delta(D,\Delta_D) \ge 1 - \varepsilon$. By part (1),
\[
(1-\varepsilon) \cdot \frac{A_X(\lambda\cdot \ord_D)}{n} \le \frac{A_\sigma(\lambda\cdot \ord_D)}{p}\, .
\]
By Lemma \ref{lem:A linear}, we may assume that both $A_X(\cdot)$ and $A_\sigma(\cdot)$ are linear around  $v\in \QM(Y,E)$. Letting $\varepsilon \to 0$, we obtain $\frac{A_X(v)}{n}\le \frac{A_{\sigma}(v)}{p}$.
\end{proof}

\begin{thm} \label{thm:sigma-admissible nbhd}
Let $x\in X$ be a klt singularity and let $v\in \Val_{X,x}$ be the normalized
volume minimizer. Let $(Y,E)\to X$ be a log smooth model adapted to $v$ and let $\sigma\in H^0(X,\Omega_X^{[p]})$ be a non-degenerate reflexive
differential form. 

Then there exists an open neighborhood $U\subseteq \QM(Y,E)$ of $v$ such that every divisorial valuation in $U$ is given by a Koll\'ar component, and the corresponding test configuration specializes $\sigma$ to a non-degenerate $p$-form $\sigma_0$ on the stable degeneration $x_0\in X_0$ of $x\in X$ (see Theorem \ref{thm-SDC}).
\end{thm}

\begin{proof}
Let $n=\dim X$. We first show that
\begin{equation}\label{eq:n/p log discrep}
\frac{A_X(v)}{n}= \frac{A_{\sigma}(v)}{p}\, .
\end{equation}
Indeed, by Theorem \ref{thm:estimate dis}(2), we have $\frac{A_X(v)}{n}\le \frac{A_{\sigma}(v)}{p}$. On the other hand, since $\sigma$ is non-degenerate, $\sigma^{\frac{n}{p}}$ is a free generator of $\omega_X$ (see Remark \ref{rem:trivial-canonical}), hence by definition $A_X(w) = A_{\sigma^{n/p}}(w) \ge \frac{n}{p} A_\sigma(w)$ for any quasi-monomial valuation $w\in \Val_{X}$. In particular, $A_X(v) \ge \frac{n}{p} A_\sigma(v)$. Combining the two inequalities we obtain \eqref{eq:n/p log discrep}.

By Lemma \ref{lem:A linear}, both $A_\sigma(\cdot)$ and $A_X(\cdot)$ are linear in some neighborhood of $v$ in $\QM(Y,E)$. Since $\frac{A_X(w)}{n}\ge \frac{A_{\sigma}(w)}{p}$ for any quasi-monomial valuation $w\in \Val_X$ and equality holds for $v$, we must have
\begin{equation} \label{eq:n/p log discrep in nbhd}
\frac{A_X(w)}{n} = \frac{A_{\sigma}(w)}{p}
\end{equation}
in some neighborhood $U\subseteq \QM(Y,E)$ of $v$. Shrinking $U$ if necessary, we may assume by Theorem \ref{thm:delta->1} that every divisorial valuation in $U$ corresponds to a Koll\'ar component. By \cite[Lemma 2.10]{LX-higher-rank}, we may also assume (possibly after shrinking $U$) that the corresponding test configuration degenerates $x\in X$ to $x_0\in X_0$. By \eqref{eq:n/p log discrep in nbhd} and Proposition \ref{prop:non-degenerate criterion}, these Koll\'ar components are also $\sigma$-admissible; in other words, the specialization $\sigma_0$ of $\sigma$ is also non-degenerate.
\end{proof}

\section{Kaledin's conjecture} \label{s-Kaledin}

In this section, we prove a strong form Theorem \ref{thm:Kaledin enhanced} of Kaledin's conjecture that symplectic singularities are formally conical. It also answers a question of Namikawa \cite[Page 3]{Namikawa-torichyperkahler}  (also see \cite{Namikawa-torichyperkahlerII}) in greater generality.




\begin{defn} \label{def:dilating auto}
Let $x\in X$ be a singularity and let $\sigma$ be a reflexive differential form on $X$. A $\sigma$-dilating (or simply dilating if $\sigma$ is clear from the context) automorphism of $x\in X$ is defined to be an automorphism $g$ of $x\in X$ such that $g^*\sigma = \lambda\sigma$ for some $\lambda\in \bC^*$. The group of dilating automorphisms is denoted by $\Aut(x\in X,[\sigma])$. In the case of a symplectic singularity $(x\in X,\sigma)$, a $\sigma$-dilating automorphism is also called a dilating symplectic automorphism. A symplectic singularity is said to be conical if it has a dilating symplectic good torus action. Recall that by Lemma \ref{lem:wt=log discrep}, the symplectic form has positive weight with respect to any dilating good $\bG_m$-action.
\end{defn}

\begin{thm} \label{thm:Kaledin enhanced}
Let $(x\in X,\sigma)$ be a symplectic singularity. Then there exists a K-semistable Fano cone singularity $x_0\in (X_0;\xi_0)$ and a $\bT_0:=\langle\xi_0\rangle$-equivariant symplectic form $\sigma_0$ on $X_0$ such that $(x\in X,\sigma)$ is formally isomorphic to $(x_0\in X_0,\sigma_0)$. 

If moreover $(x\in X,\sigma)$ has a dilating symplectic good torus $\bT_X$-action, then there exists a $\bT$-equivariant algebraic isomorphism 
\[
(x\in X,\sigma)\cong (x_0\in X_0,\sigma_0)
\]
for some torus $\bT$ extending the dilating symplectic actions by $\bT_X$ and $\bT_0$ respectively.
\end{thm}
We have the following corollary, whose first part is Kaledin's conjecture (see \cite{Kaledin-sym, Kaledin-survey}). 
\begin{cor} \label{cor-Kaledin}
Every symplectic singularity is formally isomorphic to a conical symplectic singularity, and every conical symplectic singularity is a K-semistable Fano cone. \qed
\end{cor}


We prove Theorem \ref{thm:Kaledin enhanced} in several steps. Given the results from the previous sections, what we really need to prove is that formal isomorphisms of conical symplectic singularities are algebraizable, and that the Reeb vector dilates the symplectic form (a priori, since the construction involves approximating the normalized volume minimizer by Koll\'ar components, the specialization of the symplectic form depends on the choice of the approximation and is only preserved by some rational approximation of the Reeb vector). For this we follow some ideas from \cite[8.4]{Kol-cone-auto}. Roughly speaking, we shall prove that conical symplectic singularities with a given formal isomorphism type are in one-to-one correspondence with the maximal tori in the formal dilating symplectic automorphism group. Moreover, the maximal torus is unique up to conjugation, hence the algebraic isomorphism class of the conical symplectic singularity is determined by its completion. We then analyze the stable degeneration process to show that the maximal torus yields a Reeb vector giving the normalized volume minimizer.

As a prototype of the argument, we first show how to turn a formal isomorphism into an algebraic isomorphism when the singularities have good torus actions.

\begin{lem} \label{lem:cone formal isom = alg isom}
Let $x\in X$ and $y\in Y$ be two singularities with good actions by tori $\mathbb T_X$ and $\mathbb T_Y$. Assume that they are formally isomorphic, i.e. $\widehat{\cO_{X,x}}\cong \widehat{\cO_{Y,y}}$. Then there is an algebraic isomorphism $(x\in X)\cong (y\in Y)$ as $\mathbb{T}$-singularities for some good actions by a torus $\mathbb T$ extending the actions by $\mathbb T_X$ and $\mathbb T_Y$ respectively.
\end{lem}

We emphasize that we do not assume that the tori acting on $X$ and $Y$ have the same dimension, and even if they do, the formal isomorphism does not need to be compatible with the torus actions.

\begin{proof}
This should be well-known to the experts, see e.g. \cite[Section 8]{Kol-cone-auto}. We sketch a proof for the reader's convenience. Let $\widehat{R}=\widehat{\cO_{X,x}}\cong \widehat{\cO_{Y,y}}$. By \cite[8.4]{Kol-cone-auto}, the maximal torus in the formal automorphism group $\Aut(\widehat{R})$ is well-defined and unique up to conjugation. 

If $\bT_X$ is the torus acting on $x\in X$, we may extend it to a maximal torus $\bT\subseteq \Aut(\widehat{R})$. The subalgebra $R_X:=H^0(X,\cO_X)\subseteq \widehat{R}$ is a direct sum of $\bT_X$-invariant subspaces. Because the $\bT_X$-action is good, every $\bT_X$-invariant subspace of a given weight is finite dimensional, and further decomposes into a direct sum of $\bT$-invariant subspaces. It follows that $R_X$ can be recovered from $\widehat{R}$ as the direct sum of $\bT$-invariant subspaces. Since $\bT$ is unique up to conjugation, this implies that the algebra $R_X$ is determined (up to isomorphism) by $\widehat{R}$, i.e. we have
\[
\begin{tikzcd}
 (R_X, \mathbb T_X\subset \mathbb T) \arrow[r, hook] \arrow[d,dotted] & (\widehat{R},\mathbb{T})\arrow{d}{\cong}  \\
(R_Y, \mathbb T_Y\subset \mathbb T) \arrow[r,hook] & (\widehat{R},\mathbb{T})
\end{tikzcd}
\]
In particular, there exists a $\mathbb T$-equivariant algebraic isomorphism between $X$ and $Y$.  
\end{proof}

To prove Theorem \ref{thm:Kaledin enhanced} we need the following generalization of Lemma \ref{lem:cone formal isom = alg isom} to the setting of reflexive differentials, which in particular answers the algebraicity question raised in \cite[Page 3]{Namikawa-torichyperkahler}.

\begin{prop} \label{prop:cone+form formal isom = alg isom}
Let $x\in X$ (resp. $y\in Y$) be  singularities with reflexive $p$-forms $\sigma_X$ (resp. $\sigma_Y$) and dilating good actions by tori $\mathbb T_X$ (resp. $\mathbb T_Y$). Assume that we have a formal isomorphism $(x\in X,\sigma_X)^\wedge \cong (y\in Y,\sigma_Y)^\wedge$. Then there exist some dilating good actions on $X$ (resp. $Y$) by a torus $\mathbb T$ extending the actions by $\mathbb T_X$ (resp. $\mathbb T_Y$), and a $\bT$-equivariant algebraic isomorphism $(x\in X,\sigma_X)\cong (y\in Y,\sigma_Y)$.
\end{prop}

We start with some analysis of the formal automorphism group. We use the following notation until we finish the proof of this proposition. Let $R:=R_X:=H^0(\cO_X)$ and $\fm:=\fm_x$. Let $\widehat{R} = \widehat{\cO_{X,x}}$, let $\widehat{\Omega}_{\widehat{R}}^{[p]}$ (resp. $\widehat{\Omega}_{\widehat{R}}^p$) be the completion of $\Omega_X^{[p]}$ (resp. $\Omega_{X}^p$) at $x$, and $\widehat{\sigma}\in \widehat{\Omega}_{\widehat{R}}^{[p]}$ the image of $\sigma:=\sigma_X$. We use similar notation on $Y$ by changing the subscript. For each $k\in\bN$, let $R_k:=R/\fm^k$, $\Omega_k :=\Omega_{X}^{[p]}\otimes R_k$ and let $\sigma_k$ be the image of $\sigma$ in $\Omega_k$. These only depend on the formal completion $(x\in X,\sigma_X)^\wedge$. By analogy with \cite[8.4]{Kol-cone-auto}, we want to express $\Aut(\widehat{R},[\widehat{\sigma}])$ as an inverse limit of $\Aut(R_l,[\sigma_k])$. In order to make sense of this, we need to show that each $\Omega_k$ depends only on $R_l$ (rather than $\widehat{R}$) for some sufficiently large $l$, so that the group $\Aut(R_l,[\sigma_k])$ is well-defined.

\begin{lem} \label{lem:reflex diff finite level}
For each $k\in\bN$, there exists some integer $l\ge k$ such that the $R_k$-module $\Omega_k$ is determined by $R_l$. In particular, any $\varphi\in \Aut(R_l)$ induces an isomorphism $\varphi^*\Omega_k \cong \Omega_k$. 
\end{lem}

\begin{proof}
For any finitely generated $\widehat{R}$-modules $M,N$ and any $l\in \bN$, let $V_l$ be the image of $$\varphi_l\colon \Hom(M,N)\to \Hom(M_l,N_l)$$ (where for ease of notation we denote $M\otimes R_l$ by $M_l$, etc., and Hom is taken in the category of $\widehat{R}$-modules). Since $\Hom(M_l,N_l)$ is of finite dimension and $\Hom(M,N)$ is an inverse limit of $\Hom(M_l,N_l)$, we know that $V_l$ is also the image of $\Hom(M_{l'},N_{l'})$ for some sufficiently large $l'$. We claim that for any $k\in\bN$, there exists some $l\in\bN$, $l\ge k$ such that the induced map 
\[
\varphi_{k,l}\colon\Hom(M,N)\otimes R_k \to V_l\otimes R_k
\]
is an isomorphism. Indeed, let $L\to M$ be a surjection from a finite free $\widehat{R}$-module, then $\Hom(M,N)$ is a submodule of $\Hom(L,N)$ and $s\in \Hom(M,N)$ is in the kernel of $\varphi_{l}$ if and only if $s\in \fm^l \cdot\Hom(L,N)$. By the Artin-Rees lemma, the latter condition implies $s\in \fm^k\cdot\Hom(M,N)$ if $l\gg 0$. This translates to the injectivity of $\varphi_{k,l}$, whereas its surjectivity is clear. It then follows that $\Hom(M,N)\otimes R_k$ is determined by $M_l$ and $N_l$ for some sufficiently large $l\in\bN$ (depending on $k$).

By definition, $$\widehat{\Omega}_{X}^{[p]} = \Hom(\Hom(\widehat{\Omega}_{X}^p,\widehat{R}),\widehat{R})\,.$$ Applying the above observation twice, we see that $\Omega_k$ only depends on $\Omega_{X}^p\otimes R_l$ and $R_l$ for some sufficiently large $l$. Note that $\Omega_{X}^1\otimes R_l\cong \Omega_{R_{l+1}/\bC}\otimes R_l$ only depends on $R_{l+1}$. Therefore, $\Omega_k$ only depends on $R_l$ for sufficiently large $l$.
\end{proof}

\begin{lem} \label{lem:max torus in formal auto}
A maximal torus of $\Aut(\widehat{R},[\widehat{\sigma}])$ exists and is unique up to conjugation. 
\end{lem}

\begin{proof}
By Lemma \ref{lem:reflex diff finite level}, we can choose an increasing sequence of positive integers $l_k\ge k$ ($k=1,2,\dots$) such that $\Omega_k$ only depends on $R_{l_k}$ for each $k$. Then we can define
\[
\Aut(R_{l_k},[\sigma_k]) := \{\varphi\in \Aut(R_{l_k})\,|\,\varphi^*\sigma_k = \lambda \sigma_k \mbox{ for some } \lambda\in \bC^*\},
\]
which is an algebraic subgroup of $\Aut(R_{l_k})$. Recall that $\Omega_k$ and $R_k$ only depend on the completion $\widehat{R}$, and we have $\Aut(\widehat{R},[\widehat{\sigma}]) = \varprojlim \Aut(R_{l_k},[\sigma_k]) $. The kernel of the natural homomorphism
\[
\Aut(R_{l_{k+1}},[\sigma_{k+1}])\to \Aut(R_{l_k},[\sigma_k])
\]
is unipotent for all $k\ge 2$, as it is contained in the kernel of $\Aut(R_{l_{k+1}})\to \Aut(R_{l_k})$, and the latter is unipotent as in \cite[8.4]{Kol-cone-auto}. By the same argument as in {\it loc. cit.}, we deduce that a maximal torus of $\Aut(\widehat{R},[\widehat{\sigma}])$ exists and is unique up to conjugation. 
\end{proof}

\begin{proof}[Proof of Proposition \ref{prop:cone+form formal isom = alg isom}]
We apply the same argument as in the proof of Lemma \ref{lem:cone formal isom = alg isom}. Extend $\bT_X$ to a maximal torus $\bT$ of $\Aut(\widehat{R},[\widehat{\sigma}])$. As in the proof of Lemma \ref{lem:cone formal isom = alg isom}, we know that $R_X$ is the direct sum of $\bT$-invariant subspaces of $\widehat{R}$. Once the inclusion $R_X\subseteq \widehat{R}$ and hence $\Omega_X^{[p]}\subseteq \widehat{\Omega}_{\widehat{R}}^{[p]}$ is fixed, the reflexive $p$-form $\sigma_X$ is also uniquely determined by its image $\widehat{\sigma}$ in $\widehat{\Omega}_{\widehat{R}}^{[p]}$. Since the maximal torus of $\Aut(\widehat{R},[\widehat{\sigma}])$ is unique up to conjugation by Lemma \ref{lem:max torus in formal auto}, the rest of the proof works as in Lemma \ref{lem:cone formal isom = alg isom}.
\end{proof}

We can now prove Kaledin's conjecture.

\begin{proof}[Proof of Theorem \ref{thm:Kaledin enhanced}]
By Theorem \ref{thm:sigma-admissible nbhd} and Proposition \ref{prop:closed+non-deg}(1), we can find some Koll\'ar component over $x\in X$ whose corresponding test configuration degenerates $(x\in X,\sigma)$ to some conical symplectic singularity $(x_0\in X_0,\sigma_0)$ such that $x_0\in X_0$ is a K-semistable Fano cone. By Theorem \ref{thm:symplectic rigid}, there is a formal isomorphism $(x\in X,\sigma)^{\wedge}\cong (x_0\in X_0,\sigma_0)^{\wedge}$.

We next show that the Reeb vector $\xi_0$ on $X_0$ corresponding to the normalized volume minimizer can be chosen so that the action of $\bT_0=\langle\xi_0\rangle$ is $\sigma_0$-dilating. To this end, let $\bT$ be a maximal torus of $\Aut(x_0\in X_0,[\sigma_0])$ that contains the $\sigma_0$-dilating $\bG_m$ given by its conical structure. In particular, the $\bT$-action on $x_0\in X_0$ is good. We aim to show that we can choose $\mathbb T\supseteq \mathbb T_0$. 
If the minimizer $v_0\in \Val_{X_0,x_0}$ is not of the form $\wt_\xi$ for any Reeb vector $\xi\in N(\bT)_\bR$, then there exists a $\bT$-equivariant log smooth model $(Y_0,E_0)\to X_0$ adapted to $v_0$ and an open neighborhood $U\subseteq \QM(Y_0,E_0)$ of $v_0$ such that none of the valuations in $U$ is of the form $\wt_\xi$. By Theorem \ref{thm:sigma-admissible nbhd} and Proposition \ref{prop:closed+non-deg}(1), we can find some $\bT$-invariant Koll\'ar component over $x_0\in X_0$, given by a divisorial valuation in $U$ sufficiently close to $v_0$, whose corresponding test configuration induces a trivial degeneration of $x_0\in X_0$ (since it is already a K-semistable Fano cone) and specializes $\sigma_0$ to some other symplectic form $\sigma'_0$ on $X_0$. Since the test configuration is $\bT$-equivariant by construction and the Koll\'ar component is not given by any one parameter subgroup of $\bT$ (by our choice of $U$), we get an effective dilating symplectic good $\bT\times \bG_m$-action on $(x_0\in X_0,\sigma'_0)$. By Theorem \ref{thm:symplectic rigid}, we have a formal isomorphism 
\[
(x_0\in X_0,\sigma_0)^\wedge \cong (x_0\in X_0,\sigma'_0)^\wedge;
\]
since both symplectic singularities are conical, we may upgrade it into an algebraic isomorphism 
\[
(x_0\in X_0,\sigma_0) \cong (x_0\in X_0,\sigma'_0)
\]
by Proposition \ref{prop:cone+form formal isom = alg isom}. But then we also get an effective dilating symplectic $\bT\times \bG_m$-action on $(x_0\in X_0,\sigma_0)$, contradicting our assumption that $\bT$ is a maximal torus of $\Aut(x_0\in X_0,[\sigma_0])$. Therefore, $v_0=\wt_{\xi_0}$ for some $\xi_0\in N(\bT)_\bR$ and in particular the action of $\bT_0=\langle\xi_0\rangle\subseteq \bT$ is dilating since $\bT$ is so. 

The remaining part of the theorem follows directly from Proposition \ref{prop:cone+form formal isom = alg isom}.
\end{proof}



\section{Examples}\label{s-open}

In this section, we apply our general results in Section \ref{s-Kaledin} to two classes of symplectic singularities, namely (normalized) nilpotent orbit closures and hypertoric singularities. We analyze the torus actions on these singularities, and identify the minimizers of their normalized volume functions with the help of their symmetry.

\subsection{Torus in the center}

We first explain our general strategy for finding the minimizers. For any group $G$, we denote its center by $Z(G)$ and denote by $C_G(H)$ the centralizer of a subgroup $H$, i.e. 
\[
C_G(H)=\{g\in G\mid gh=hg, \forall h\in H\}\,.
\]
If $G$ acts on a singularity $x\in (X,\Delta)$, we also denote by $\Aut^G(x\in (X,\Delta))$ the group of $G$-equivariant automorphisms, which is also the centralizer of $G$ in $\Aut(x\in (X,\Delta))$. The identity component of an algebraic group $G$ is denoted by $G^\circ$. The starting point is the following  observation. 

\begin{lem}\label{l-center}
Let $x\in (X=\Spec(R),\Delta)$ be a log Fano cone singularity and let $\xi$ be a Reeb vector. Let $\bT = \langle\xi\rangle$ be the torus generated by $\xi$ and let $G\subseteq {\rm Aut}(x\in (X,\Delta))$ be a reductive subgroup containing $\bT$ that preserves the valuation $\wt_\xi$. Then $\mathbb T \subseteq Z(G)$.
\end{lem}

We caution that the lemma fails without the reductivity assumption. Already for $0\in \bA^2$ there exist automorphisms (e.g. $(x,y)\mapsto (x+ty^2,y)$, $t\in \bC$) that preserve the valuation $\mult_0$ but do not commute with the $\bG_m$-scaling action. The reason behind this is that while the minimizer $\wt_\xi$ is invariant under automorphisms, the Reeb vector $\xi$ is not, as different Reeb vectors (for different torus actions) can yield the same valuation. 

We will apply Lemma \ref{l-center} to K-semistable Fano cone singularities. In this setting, the reductivity assumption is almost necessary: if $\xi$ is quasi-regular (i.e. $\bT=\bG_m$) and $x\in (X,\Delta;\xi)$ is \emph{K-polystable}, then the centralizer of $\bT$ in $\Aut(x\in (X,\Delta))$ is a direct product of $\bT$ with the automorphism group of the K-polystable log Fano pair $(V,\Delta_V):=\left((X,\Delta)\setminus\{x\}\right)/\bG_m$, and $\Aut(V,\Delta_V)$ is reductive by \cite{ABHLX-reductivity}. 

\begin{proof}
By \cite[8.4]{Kol-cone-auto}, the natural map $G\to \Aut(R/\fm_x^2)$ (where the right hand side is the automorphism group of the finite dimensional algebra $R/\fm_x^2$) is an embedding, since it has a unipotent kernel and $G$ is reductive. We may thus view $G$ and $\bT$ as subgroups of ${\rm GL}(V)$, where $V=R/\fm_x^2$. We have a weight decomposition $V = \bigoplus_{\alpha\in M(\bT)} V_\alpha$ and the Reeb vector $\xi$ induces a filtration of $V$ by 
\[
\cF_\xi^\lambda V := \bigoplus_{\alpha\in m(\bT),\,\langle\alpha,\xi\rangle\ge \lambda} V_\alpha
\]
which coincides with the filtration induced by $\wt_\xi$, i.e. $\bar{s}\in \cF_\xi^\lambda V$ if and only if there exists some $s\in R$ with reduction $\bar{s}$ such that $\wt_\xi (s)\ge \lambda$. 

Since $\wt_\xi$ is $G$-invariant by assumption, the filtration $\cF_\xi^\bullet V$ is also $G$-invariant, hence $G$ is contained in the parabolic subgroup $P(\xi)\subseteq {\rm GL}(V)$ that fixes the filtration $\cF_\xi^\bullet V$. In effect, after a change of basis and with $m=\dim V$ we may assume that $\bT$ is contained in the diagonal subgroup of ${\rm GL}_m$, the Reeb vector $\xi$ may then be identified with an element $(\xi_1,\dots,\xi_m)$ of $\bR^m$ which without loss of generality can be arranged so that $\xi_1\ge \dots\ge\xi_m$. Then
\[
P(\xi) = \{A=(a_{ij})_{1\le i,j\le m}\,|\,a_{ij}=0 \mbox{ whenever } \xi_i>\xi_j\}\subseteq {\rm GL}_m
\]
is the subgroup of block upper triangular matrices. From this explicit description, it is straightforward to see that $\xi$ is in the Lie algebra of the center of the Levi subgroup
\[
L(\xi) = \{(a_{ij})\in {\rm GL}_m\,|\,a_{ij}=0 \mbox{ whenever } \xi_i\neq \xi_j\} \subseteq P(\xi)
\]
of block diagonal matrices. In particular, the torus $\bT$ generated by $\xi$ is (by definition) contained in the center of $L(\xi)$. Since $G$ is reductive, the induced map $G\subseteq P(\xi)\to L(\xi)$ is injective (note that since we are in characteristic $0$, the unipotent kernel of $P(\xi)\to L(\xi)$ can not have any finite subgroup). Therefore, as $\bT$ commutes with $L(\xi)$, it also commutes with $G$.
\end{proof}


As a corollary, we have the following criterion for the quasi-regularity of the minimizer.

\begin{lem} \label{lem:quasi-regular}
Let $x\in (X,\Delta;\xi)$ be a K-semistable log Fano cone singularity with the action of a reductive group $G$. Assume that
\begin{enumerate}
    \item $Z(G)^\circ\cong \bG_m$,
    \item the $\bT_0:=Z(G)^\circ$-action on $x\in X$ is good, and
    \item $\bT_0$ is a maximal torus of $\Aut^G(x\in (X,\Delta))$. 
\end{enumerate}
Then $\xi$ is quasi-regular, and $\langle\xi\rangle$ is conjugate to $\bT_0$ in $\Aut(x\in (X,\Delta))$.
\end{lem}


\begin{proof}
Since the $\bT_0$-action on $x\in X$ is good, $\Aut^{\bT_0}(x\in (X,\Delta))$ is the direct product of $\mathbb T_0$ and the automorphism group of the log Fano pair 
\[
((X,\Delta)\setminus\{x\})/\bT_0\, ,
\]
in particular it is an algebraic group. Note that $G\subseteq \Aut^{\bT_0}(x\in (X,\Delta))$. Let $H$ be a maximal reductive subgroup of $\Aut^{\bT_0}(x\in (X,\Delta))$ containing $G$ (in particular $H^\circ$ is a Levi subgroup of $\Aut^{\bT_0}(x\in (X,\Delta))$), and let $\bT$ be a maximal torus of $H$ containing $\bT_0$. Then $\bT$ is a maximal torus of $\Aut^{\bT_0}(x\in (X,\Delta))$ and hence also of $\Aut(x\in (X,\Delta))$ as $\Aut^{\bT_0}(x\in (X,\Delta))$ is the centralizer of the torus $\bT_0$. 

Since all the maximal tori of $\Aut(x\in (X,\Delta))$ are conjugate to each other by Lemma \ref{lem:cone formal isom = alg isom}, we may replace $\xi$ by a conjugate and assume that $\langle \xi \rangle\subseteq \bT$. In particular, $\langle \xi \rangle\subseteq H$. Since the normalized volume minimizer $\wt_\xi$ is unique up to scaling \cite{XZ-uniqueness}, it is invariant under the action of ${\rm Aut}(x\in (X,\Delta))$.  We can thus apply Lemma \ref{l-center} and conclude that
\[
\langle \xi \rangle \subseteq Z(H)\subseteq \Aut^H(x\in (X,\Delta))\subseteq \Aut^G(x\in (X,\Delta))\,.
\]
By assumption, $\bT_0$ is a maximal torus of $\Aut^G(x\in (X,\Delta))$, hence by \cite[Section 21.3, Corollary A]{Humphreys-alg-gp}, 
the torus $\langle \xi \rangle$ is conjugate to a subtorus of $\bT_0$. But $\bT_0\cong \bG_m$, thus $\langle\xi\rangle \cong \bG_m$ and it is conjugate to $\bT_0$. 
\end{proof}

Recall that $\Aut(x\in X,[\sigma])$ is the dilating symplectic automorphism group (see Definition \ref{def:dilating auto}). By the same argument as above, we have  

\begin{lem} \label{lem:quasi-regular symplectic}
Let $(x\in X, \sigma)$ be a conical symplectic singularity and let $G\subseteq \Aut(x\in X,[\sigma])$ be a reductive group. Assume that
\begin{enumerate}
    \item $Z(G)^\circ\cong \bG_m$,
    \item the $\bT_0:=Z(G)^{\circ}$-action on $x\in X$ is good, and
    \item $\bT_0$ is a maximal torus of $\Aut^G(x\in X,[\sigma])$. 
\end{enumerate}
Then the Koll\'ar component corresponding to the $\bT_0$-action is the minimizer of the normalized volume.
\end{lem}

\begin{proof}
By Theorem \ref{thm:Kaledin enhanced}, there exists some Reeb vector $\xi$ on $X$ such that $\wt_\xi$ is the normalized volume minimizer and $\langle\xi\rangle\subseteq \Aut(x\in X,[\sigma])$. The rest of the proof then works as in Lemma \ref{lem:quasi-regular}, replacing $\Aut(x\in (X,\Delta))$ by $\Aut(x\in X,[\sigma])$, and the use of Lemma \ref{lem:cone formal isom = alg isom} by Theorem \ref{thm:Kaledin enhanced}.
\end{proof}

\subsection{Normalized nilpotent closures}


We apply the above criterion to nilpotent orbit closures and address part of \cite[Problem 23]{XZ-open}. Let $G$ be a simply connected semisimple algebraic group and let $\fg$ be its Lie algebra.  The  group $G$ acts on $\fg$ by the adjoint representation.  An element $u\in \fg$ is called \emph{nilpotent} if ${\rm ad}(u):\fg\to \fg,\,v\mapsto [u,v]$ is a nilpotent endomorphism.  The \emph{nilpotent cone} of $\fg$ is
\[
  \cN=\cN(\fg):=\{e\in \fg\mid e \text{ is nilpotent}\}.
\]
Then $\cN$ is invariant under the adjoint action ${\rm Ad}\colon G\times \fg\to \fg$.

For any nonzero element $e\in \cN$, let $O_{e}\subset \cN$ be its orbit under the adjoint action of $G$ and let $\overline{O}_e $ be its closure in $\cN$. We also denote by $\widetilde{O}_e$ the normalization of  $\overline{O}_e $. It is well-known that there are only finitely many nilpotent orbits in $\cN$ (see e.g. \cite[Section 3]{CM-nilpotent-orbit}), so $\overline{O}_e$ and $\widetilde{O}_e $ consist of finitely many $G$-orbits. 

It is also known that $\overline{O}_e$ and $\widetilde{O}_e$ are preserved by the natural scaling $\bG_m$-action on $\fg$. Denote by $\overline{o}\in \overline{O}_e$ and $\tilde{o}\in \widetilde{O}_e $ the corresponding vertices, which are the preimages of $0\in \fg$. The group $G$ acts on $O_e$, $\overline{O}_e$ and therefore $\widetilde{O}_e$, and has $\tilde{o}$ as a fixed point. The normalization $\widetilde{O}_e$ is a conical symplectic singularity with the $G$-invariant Kirillov-Kostant-Souriau form $\sigma$ and the induced $\bG_m$-action on $\widetilde{O}_e $ is dilating of weight one (see \cite{Pan-nil-orbit-symp} and \cite[Section 1]{BK-nilpotent}). 

\begin{lem}\label{l-one dim}
The identity component of $\Aut^G(\tilde{o} \in \widetilde{O}_e,[\sigma])$ is a semi-direct product $\mathbb{G}_m \ltimes U$ for some unipotent group $U$. 
\end{lem}

\begin{proof}
Let $G=\prod_{i=1}^d G_i$, $\fg=\bigoplus_{i=1}^d \fg_i$ be the decomposition into simple factors and write $e=(e_1,\dots,e_d)$. By discarding factors $G_i$ with $e_i=0$ (which does not affect the isomorphism class of $(\tilde{o} \in \widetilde{O}_e,\sigma)$), we may assume that every $e_i$ is nonzero. We have $O_e = \prod_{i=1}^d O_{e_i}$, $\widetilde{O}_e = \prod_{i=1}^d \widetilde{O}_{e_i}$ and $\sigma = \sum_{i=1}^d \mathrm{pr}_i^*\sigma_i$ where $\sigma_i$ is the Kirillov-Kostant-Souriau form on the normalized orbit closure $\widetilde{O}_{e_i}$ of $e_i\in \cN(\fg_i)$.

Let $H=\Aut^G(\tilde{o} \in \widetilde{O}_e,[\sigma])^\circ$. As $H$ commutes with the $G$-action, it preserves each $G$-orbit and in particular $H\subseteq  \Aut^G(O_e,[\sigma])$.
Since $O_e=G/G_e = \prod_{i=1}^d G_i/G_{i,e_i}$, where $G_e$ (resp. $G_{i,e_i}$) is the stabilizer of $e$ (resp. $e_i$) in $G$ (resp. $G_i$), we know that
\[
\Aut^G(O_e) \cong N(G_e)/G_e \cong \prod_{i=1}^d \left(N(G_{i,e_i})/G_{i,e_i}\right)\cong \prod_{i=1}^d \Aut^G(O_{e_i})\,,
\]
where $N(G_e)$ is the normalizer of $G_e$ in $G$ and similarly for $N(G_{i,e_i})$. Taking the symplectic form $\sigma$ into account, the above product decomposition gives
\[
\Aut^G(O_e,[\sigma]) \subseteq \prod_{i=1}^d \Aut^{G_i}(O_{e_i},[\sigma_i]) \subseteq \prod_{i=1}^d \left(N(G_{i,e_i})/G_{i,e_i}\right)\,. 
\]
By \cite[Theorem 7.1]{BK-nilpotent}, we have 
\[
\left(N(G_{i,e_i})/G_{i,e_i}\right)^\circ \cong \bG_m \ltimes U_i
\]
for some unipotent group $U_i$. From the discussion before the lemma, we also know that $\Aut^G(O_{e_i},[\sigma_i])$ contains at least one $\bG_m$, thus 
\[
\Aut^{G_i}(O_{e_i},[\sigma_i])^\circ = \bG_m \ltimes U'_i
\]
for some unipotent group $U'_i$. Note that $U'_i$ acts trivially on $\sigma_i$ (as a unipotent group does not have any non-trivial one-dimensional representation), and the $\bG_m$ factor is dilating on $\sigma_i$ with weight one. Now 
\[
\prod_{i=1}^d \Aut^{G_i}(O_{e_i},[\sigma_i])^\circ \cong \bG_m^d \ltimes U
\]
where $U=\prod_{i=1}^d U'_i$ acts trivially on $\sigma$, while the only subgroup of $\bG_m^d$ that dilates $\sigma$ is the diagonal subgroup. Hence 
\[
\Aut^G(O_e,[\sigma])^\circ \cong \bG_m\ltimes U
\]
as desired.
\end{proof}

\begin{thm}\label{t-nilpotent}
For the normalization of a nilpotent orbit closure, the $\bG_m$ scaling action gives the minimizer of the normalized volume. As a consequence, the quotient $\left(\widetilde{O}_e \setminus \{\tilde{o}\}\right)/\bG_m$ is a K-semistable Fano variety.
\end{thm}

\begin{proof}
Let $H=G\times \bG_m\subseteq \Aut(\tilde{o} \in \widetilde{O}_e,[\sigma])$. Then $Z(H)=\{1\}\times \bG_m$, since $G$ is semisimple. The $Z(H)$-action on $\tilde{o} \in \widetilde{O}_e$ is induced by the scaling action on $\fg$, which is good, and 
\[
\Aut^H(\tilde{o} \in \widetilde{O}_e,[\sigma])^\circ\subseteq \Aut^G(\tilde{o} \in \widetilde{O}_e,[\sigma])^\circ \cong \bG_m\ltimes U
\]
for some unipotent group $U$ by Lemma \ref{l-one dim}, hence $Z(H)\cong \bG_m$ is a maximal torus of $\Aut^H(\tilde{o} \in \widetilde{O}_e,[\sigma])$. This shows that $H$ satisfies all the conditions of Lemma \ref{lem:quasi-regular symplectic}. Therefore, we conclude by Lemma \ref{lem:quasi-regular symplectic} that the normalized volume minimizer is given by the Koll\'ar component that corresponds to the $Z(H)$-action, i.e. the $\bG_m$ scaling action on $\tilde{o} \in \widetilde{O}_e$.



Since  $\mathbb{C}[\overline{O}_e]$ is generated by its degree one elements  (which are naturally identified with $\fg$), we know the $\mathbb{G}_m$-action on $\overline{O}_e\setminus \{\overline{o}\}$ is free. As $\widetilde{O}_e\to \overline{O}_e$ is isomorphic in codimension one (see \cite[Proposition 1.2]{BK-nilpotent}), the $\mathbb{G}_m$-action on $\widetilde{O}_e\setminus \{\tilde{o}\}$ is free in codimension one. Thus the quotient is a variety (instead of a pair). Therefore, it is a K-semistable Fano variety by \cite{LX-Kol-comp-stab}.
\end{proof}


\begin{rem}\label{rem-Kronheimer}
It is expected that the base $\left(\widetilde{O}_e \setminus \{\tilde{o}\}\right)/\bG_m$ of the normalized nilpotent orbit closure, and more generally, the natural bases (obtained through the Kazhdan $\bG_m$-action) of the Slodowy slices of nilpotent orbit closures are K-polystable. Indeed, Kronheimer \cite[Theorem 1 and Remark (ii) thereafter]{Kronheimer-nilpotent} has shown that the smooth locus of any Slodowy slice admits a hyperk\"ahler metric. While Kronheimer's original construction is somewhat indirect, in the case of nilpotent orbit closures in classical simple groups, this hyperk\"ahler metric can also be realized via hyperk\"ahler reduction  from a flat quaternionic vector space $\mathbb H^n$ \cite{KS-classical-nilpotent} (see also \cite[Section 2]{KS-hyperkahler-potential}). In this case, it follows from the explicit description of the corresponding hyperk\"ahler potential (we only need its asymptotic behavior near the boundary $\widetilde{O}_e\setminus O_e$) that the hyperk\"ahler metric extends to a K\"ahler cone metric on the entire normalized nilpotent orbit closure; in particular, the base is a K-polystable Fano variety.  
\end{rem}

\subsection{Hypertoric singularities}\label{ss-hypertoric}
Hypertoric singularities are the hyperk\"ahler analogue of affine toric varieties. They are defined as the Hamiltonian reduction of $(\mathbb{C}^{2N},\sigma_{\rm st})$ with respect to a subtorus of $\mathbb{G}_m^N$. In particular, they admit an effective Hamiltonian action by a half-dimensional torus. See \cite{BD-toric,Hausel-Sturmfels, Proudfoot-survey, Namikawa-torichyperkahler} for more background.  
\begin{defn}[Hypertoric singularity]
Let \(N\) and \(n\) be positive integers with \(n \leq N\). 
Let
\[
        B : \mathbb Z^n \longrightarrow \mathbb Z^N
\]
be an injective homomorphism, represented by an integer-valued
\(N \times n\)-matrix, such that each row vector of \(B\) is primitive and
\(\operatorname{Coker}(B)\) is torsion-free. Choose an integer-valued
\((N-n)\times N\)-matrix \(A\) so that there is an exact sequence
\[
        0 \longrightarrow \mathbb Z^n   \xrightarrow{\;B\;}  \mathbb Z^N     \xrightarrow{\;A\;}    \mathbb Z^{N-n}  \longrightarrow 0 .
\]
This gives an exact sequence of algebraic tori
\[
        1 \longrightarrow \mathbb T_A      \longrightarrow \mathbb T        \longrightarrow  \mathbb T_Y
        \longrightarrow 1 \,,
\]
where $\mathbb T_A\cong \bG_m^{N-n}$, $\mathbb T\cong \bG_m^{N}$ and $\mathbb T_Y\cong \bG_m^{n}$. 
Let \(\mathbb C^{2N}\) have coordinates
\[
        z_1,\ldots,z_N,w_1,\ldots,w_N
\]
and the standard symplectic form
\[
        \sigma_{\mathrm{st}} =  \sum_{i=1}^N \rd z_i \wedge \rd w_i .
\]
The torus $\mathbb T $ acts on $(\mathbb C^{2N},   \sigma_{\mathrm{st}})$ by
\[
        z_i \mapsto t_i z_i,
        \qquad
        w_i \mapsto t_i^{-1} w_i,
        \qquad 1 \leq i \leq N .
\]
Restricting this action to the subtorus $ \mathbb T _A\subset  \mathbb T $, we obtain a Hamiltonian $ \mathbb T _A$-action. Let \(a_{ij}\) be the \((i,j)\)-entry of \(A\), and let $ \mu \colon \mathbb C^{2N} \longrightarrow \mathbb C^{N-n}$ be defined by
\[
        \mu(\mathbf{z},\mathbf{w})
        =
        \left(
        \sum_{j=1}^N a_{1j} z_j w_j,
        \ldots,
        \sum_{j=1}^N a_{N-n,j} z_j w_j
        \right) \, ,
\]
which is the corresponding moment map, normalized by \(\mu(0)=0\).  Then
\begin{eqnarray*}
        Y(A,0)
       & := &
        \mu^{-1}(0) /\!\!/   \mathbb T_A
         \end{eqnarray*}     
is called the \emph{hypertoric singularity} associated to \(A\).

The 2-form $\sigma_{\rm st}|_{\mu^{-1}(0)}$ descends to a symplectic form $\sigma_Y$ on $Y(A,0)$. One can see that $(Y(A,0),\sigma_Y)$ is a conical symplectic singularity with a good $\bG_m$-action induced by the standard dilation on $\mathbb C^{2N}$, and $\mathbb{T}_Y$ acts on $(Y(A,0),\sigma_Y)$ by symplectomorphisms.
\end{defn}

Recall that a symplectic $G$-action on a symplectic variety $(Y,\sigma)$ is said to be Hamiltonian if there is a map $\fg\to \bC[Y]$, $\xi\mapsto H_\xi$ such that $\rd H_\xi =\iota_{\xi_Y} \sigma$ for any $\xi\in \fg$, where $\xi_Y$ is the vector field on $Y$ induced by the $G$-action. 

\begin{lem}\label{l-Sym=Hamilton}
Any symplectic torus $\bT$-action on a conical symplectic singularity $(y\in Y,\sigma)$ that commutes with the conical $\bG_m$-action is Hamiltonian. 
\end{lem}
\begin{proof}
Assume the $\mathbb{G}_m$-action on $\sigma$ has weight $\ell>0$.
It is enough to construct the  Hamiltonian function for each
\(\xi\in\mathfrak t=\operatorname{Lie}(\mathbb T)\). Let $\xi_Y$ denote the corresponding vector field on $Y$. Since the $\mathbb T$-action is symplectic, it preserves $\sigma$ and we have $\mathcal L_{\xi_Y}\sigma=0$.
Hence, by Cartan's formula,
\[
\rd(\iota_{\xi_Y}\sigma)=\mathcal L_{\xi_Y}\sigma-\iota_{\xi_Y}\rd\sigma=0.
\]
Thus $\alpha_\xi:=\iota_{\xi_Y}\sigma$ is a closed regular \(1\)-form.

Let $\eta$ be the Euler vector field corresponding to the conical $\bG_m$-action. Then $\cL_{\eta}\sigma = \ell\sigma$. Because the $\mathbb T$-action commutes with the conical
\(\mathbb G_m\)-action, we have $ [\eta,\xi_Y]=0$. Therefore
\[
\mathcal L_\eta\alpha_\xi = \mathcal L_\eta(\iota_{\xi_Y}\sigma) = \iota_{[\eta,\xi_Y]}\sigma+\iota_{\xi_Y}(\mathcal L_\eta\sigma) = \ell\,\iota_{\xi_Y}\sigma =\ell\alpha_\xi .
\]
On the other hand, applying Cartan's formula to the closed \(1\)-form
\(\alpha_\xi\), we get
\[
\mathcal L_\eta\alpha_\xi
=
\rd(\iota_\eta \alpha_\xi)+\iota_\eta(\rd\alpha_\xi)
=
\rd(\iota_\eta\alpha_\xi).
\]
Since $\ell>0$, combining the two identities gives 
\[
\alpha_\xi
=
\rd\left(\frac{1}{\ell}\iota_\eta\alpha_\xi\right)
=
\rd\left(\frac{1}{\ell}\iota_\eta\iota_{\xi_Y}\sigma\right).
\]
Thus the Hamiltonian function for \(\xi\) is
\[
H_\xi
=
\frac{1}{\ell}\iota_\eta\iota_{\xi_Y}\sigma
=
\frac{1}{\ell}\sigma(\xi_Y,\eta),
\]
up to the usual sign convention for contractions. 
In particular,
\[
\rd H_\xi=\iota_{\xi_Y}\sigma \, ,
\]
so the action is Hamiltonian.
\end{proof}

We also need the following standard lemma. 
\begin{lem}\label{l-dimHamilton}
Let $\mathbb T$ be a torus with an effective Hamiltonian action on a symplectic variety $(Y,\sigma)$. Then $\dim (\mathbb{T})\le \frac{1}{2}\dim (Y)$.
\end{lem}

\begin{proof}
For $\xi\in \operatorname{Lie}(\mathbb T)$, let \(\xi_Y\) denote the corresponding vector field on $Y$. Since the torus action under consideration is Hamiltonian, there is a Hamiltonian function $H_\xi$ such that $\rd H_\xi=\iota_{\xi_Y}\sigma$.
For two elements \(\xi,\eta\in \operatorname{Lie}(\mathbb T)\), as the torus $\mathbb T$ is abelian, the Hamiltonian functions Poisson commute:
\[
        \{H_\xi,H_\eta\}=0 .
\]
But $  \{H_\xi,H_\eta\}=        \sigma(\xi_Y,\eta_Y)$.
Therefore, $      \sigma(\xi_Y,\eta_Y)=0$
for all \(\xi,\eta\in \operatorname{Lie}(\mathbb T)\). Hence the tangent space to the $\mathbb T$-orbit at a general smooth point $y\in Y$ is an isotropic subspace of the symplectic vector space
$ T_yY$. Therefore, $\dim (\mathbb T)\le \frac{1}{2}\dim Y$.
\end{proof}

\begin{thm}\label{t-hypertoric}
The standard diagonal $\bG_m$-action on $\mathbb C^{2N}$ descends to a good \(\mathbb G_m\)-action on \(Y(A,0)\), which yields the minimizer of $Y(A,0)$.
\end{thm}

\begin{proof}
It is easy to see that the diagonal $\mathbb{G}_m$-action descends to an action $\lambda$ on $Y(A,0)$ and preserves the symplectic form \(\sigma_Y\) (with weight \(2\)). 

Since $\dim (Y(A,0))=2n$, by Lemmas \ref{l-Sym=Hamilton} and \ref{l-dimHamilton}, the maximal torus which acts on $(Y(A,0),\sigma_Y)$ by symplectomorphisms and commutes with $\lambda$ is of dimension at most $n$. So $\mathbb{T}_Y$ is such a maximal torus. Therefore, the maximal torus which dilates $(Y(A,0),\sigma_Y)$ is of dimension $n+1$, and $\mathbb{T}'=\mathbb{G}_m\times \mathbb T_Y$ yields such a maximal torus, where the action of the first factor is $\lambda$.

By Theorem \ref{thm:Kaledin enhanced}, we know that the Reeb cone of $\mathbb{T}'$ contains a minimizer of the normalized volume function for $0\in Y(A,0)$. 
There is an involution
\[
        \iota:\mathbb C^{2N}\longrightarrow \mathbb C^{2N},
        \qquad
        \iota(z_i,w_i)=(w_i,z_i),
\]
which preserves $\mu^{-1}(0)$ and $\sigma_{\mathrm{st}}$, hence descends to a symplectic automorphism of $Y(A,0)$. Note that $\iota$ commutes with $\lambda$ and  $\iota^{-1}\bT_Y\iota=\bT_Y$, and the conjugation action by $\iota$ on $\bT_Y$ swaps the weights on $z_i$ and $w_i$. Let $G\cong \bZ/2\bZ\ltimes \bT'$ be the subgroup of $\Aut(Y(A,0),[\sigma_Y])$ generated by $\iota$ and $\bT'$. Then $Z(G)=\bG_m$, generated by the projection of the standard dilation on $\mathbb{C}^{2N}$. By Lemma \ref{l-center}, we see that it gives the minimizer of $Y(A,0)$. 
\end{proof}

\begin{rem}
By taking symplectic quotients and products, we get more examples of symplectic Fano cone singularities with explicit descriptions of the normalized volume minimizer. For quotients, this follows from the uniqueness of the minimizer \cite{XZ-uniqueness}. For products, the relevant more general fact is that if $x_i\in (X_i;\xi_i)$ are K-semistable Fano cone singularities, and we normalize the Reeb vectors $\xi_i$ so that $A_{X_i}(\xi_i)=n_i:=\dim X_i$, then the product $(x_1,x_2)\in (X_1\times X_2;\xi_1+\xi_2)$ is a K-semistable Fano cone. Since we are unable to find a good reference for this, we sketch an analytic proof as follows (we skip some details since this fact is not used elsewhere in this paper; it should also be possible to prove it by purely algebraic arguments along the lines of \cite{Z-product-K,XZ-uniqueness}). By taking K-polystable degenerations of $X_i$ (see e.g. \cite{LWX-tangent-cone}) it is not hard to reduce this to the case when $x_i\in (X_i;\xi_i)$ are K-polystable Fano cones. By the Yau-Tian-Donaldson correspondence for Fano cone singularities (see \cite[Theorem 2.9]{Li-Fano-cone-YTD} or \cite[Theorem 1.0.8]{Huang-thesis}), each $x_i\in (X_i;\xi_i)$ admits a Ricci-flat K\"ahler cone metric $\omega_i$, hence the product also admits a Ricci-flat K\"ahler cone metric $\omega:=\mathrm{pr}_1^*\omega_1+\mathrm{pr}_2^*\omega_2$ (the normalization $A_{X_i}(\xi_i)=n_i$ ensures that $\omega$ is homothetic with respect to the vector field $\xi_1+\xi_2$), hence another application of the Yau-Tian-Donaldson correspondence yields that $(x_1,x_2)\in (X_1\times X_2;\xi_1+\xi_2)$ is a K-polystable Fano cone.

\end{rem}

\appendix

\section{Darboux theorem for symplectic singularities} \label{s:Darboux}

In this appendix, we prove an analog of Darboux's theorem for symplectic singularities, in the sense that the symplectic structure is unique up to analytic isomorphism on any singularity germ. This might be well-known to the experts but we are unable to find a good reference (see e.g. \cite[Lemma 1.3]{Namikawa-deformation} and \cite{Namikawa-equivalence-symp} for some special cases). Throughout, let $\Aut(X^{\an},x)$ be the automorphism group of an analytic singularity germ $x\in X$.

\begin{thm} \label{thm:Darboux}
Let $x\in X$ be a rational singularity germ, and let $\sigma_0,\sigma_1$ be two symplectic forms on $X$. Then there exists some $\varphi\in \Aut(X^\an,x)$ such that $\varphi^*\sigma_0 = \sigma_1$.
\end{thm}

Before proving this, we recall an auxiliary result on local group actions on normal singularity germs induced by vector fields. Here we say that a subvariety (or subscheme) is preserved by a vector field $\xi$ on (the smooth locus of) a normal variety if its ideal sheaf is preserved by the corresponding derivation $\partial_\xi$.

\begin{lem} \label{lem:vector field->auto}
Let $x\in X$ be a normal singularity germ and let $\xi_t\in H^0(X,T_X)$ be a family of vector fields on $X$ that fixes $x$, parametrized by an analytic convex open neighborhood $U_0$ of $0\in\bA^1$. Then for any compact subset $V_0\subseteq U_0$, there exists a one-parameter family $\varphi_t\in \Aut(X^{\an},x)$, $t\in V_0$, such that $\varphi_0=\mathrm{id}$ and $\frac{\rd \varphi_t}{\rd t} = \xi_t$.
\end{lem}

\begin{proof}
We may assume that $V_0$ is also convex. Consider the vector field $\widetilde{\xi}$ on $X\times U_0$ given by $\widetilde{\xi}(\mathbf{x},t) = (\xi_t,\partial_t)$. By \cite[Satz 3]{Kau-group-action}, it is induced by a local group action on $X\times U_0$. In other words, there exists an open neighborhood $U$ of $X\times U_0\times \{0\}$ in $X\times U_0\times U_0$ and a morphism $U\to X\times U_0$ of the form $$(\mathbf{x},s,t)\mapsto f_t(\mathbf{x},s) =  (\widetilde{\varphi}_t(\mathbf{x},s),s+t)$$ such that $f_0=\mathrm{id}$ and $f_t\circ f_{t'} = f_{t+t'}$ whenever both sides are well-defined. Because the vector field $\xi_t$ fixes $x$, we have $f_t(x,s)=(x,s+t)$ for any $s\in U_0$ and $0<|t|\ll 1$ (depending on $s$). By the compactness of $V_0$, we can then find some $0<\varepsilon\ll 1$ and an analytic open neighborhood $U_x$ of $x\in X$ such that $f_t(\mathbf{x},s)$ is defined for all $\mathbf{x}\in U_x$, $s\in V_0$ and $|t|<\varepsilon$. This in turn implies that the composition 
\[ f_{t/k}\circ \dots \circ f_{t/k} \mbox {\ \ \ ($k$ times)}
\] is well defined on some open neighborhood $U'$ of $(x,0)\in X\times U_0$ for some sufficiently large integer $k$ and for all $t\in V_0$. It is an analytic continuation of $f$ to $U'\times V_0$ since $f$ defines a group action. Therefore, we may assume from the beginning that $U'\times V_0\subseteq U$. Then 
\[
\varphi_t (\mathbf{x}) := \widetilde{\varphi}_t (\mathbf{x},0), \quad t\in V_0
\]
defines an automorphism of the analytic germ $x\in X$ and satisfies the conditions of the lemma by construction.
\end{proof}

Recall that any symplectic form $\sigma$ on a normal singularity $x\in X$ induces an isomorphism $\Omega_X^{[2]}\cong (\wedge^{2}T_X)^{**}$, which turns $\sigma$ into a bivector $\Theta$ and defines a Poisson structure on $X$ by $\{f,g\}:=\Theta(\rd f\wedge \rd g)$, see \cite{Kaledin-sym}. Any $f\in \cO_X$ thus gives a derivation $\{f,-\}$ on $\cO_X$, which can be identified with the Hamiltonian vector field $H_f$ on $X$ defined by $\iota_{H_f}\sigma = \rd f$. 

\begin{lem} \label{lem:Poicare}
Let $x\in X$ be a symplectic singularity, and let $\xi$ be a vector field on $X$. Then there exists some $f\in \cO_{X,x}$ such that $\xi-H_f$ fixes $x$.
\end{lem}

\begin{proof}
Define a stratification $$X_0\supseteq X_1\supseteq \dots\supseteq X_m$$ of $X$ inductively by $X_0:=X$ and $X_i:=\mathrm{Sing}(X_{i-1})$ until the last stratum becomes smooth. This stratification is preserved by any automorphism of $X$ and hence also by all vector fields. In particular, the ideal sheaf of $X_i$ is preserved by the derivation $\partial_\xi$ corresponding to $\xi$, thus $\partial_\xi$ also defines a derivation on $\cO_{X_i}$ and hence a vector field $\xi_i$ on each $X_i$. By \cite[Theorem 2.3]{Kaledin-sym} (and the proof of \cite[Proposition 3.1]{Kaledin-sym} which gives the precise form of the stratification), every $X_i$ is a Poisson subscheme of $X$ whose induced Poisson structure is given by a symplectic form $\sigma_i$. By construction, the smallest stratum $X_m$ is smooth. 

If $x_1,\dots,x_r \in \cO_{X_m,x}$ are local coordinates at $x\in X_m$, and $\eta_i:=H_{x_i}$ ($i=1,\dots,r$) are the corresponding Hamiltonian vector fields on $X_m$, then $\eta_1,\dots,\eta_r$ span $T_x X_m$. In particular, we can choose some linear combination $\bar{f}=\sum_{i=1}^r a_i x_i \in \cO_{X_m,x}$ ($a_i\in \bC$) of the $x_i$ such that the vector field
\[
\xi_m-H_{\bar{f}} = \xi_m - \sum_{i=1}^r a_i\eta_i
\]
on $X_m$ vanishes at $x$, which implies 
\[
\partial_{\xi_m}(\bar{g})-\{\bar{f},\bar{g}\}\in \fm_x \cO_{X_m,x}\mbox{\ \ for any $\bar{g}\in \fm_x$}\, . 
\]
Therefore, if $f\in\cO_{X,x}$ is any lift of $\bar{f}$, then the above condition yields 
\[
\partial_\xi(g) - \{f,g\}\in \fm_x
\]
for any $g\in \fm_x$. In other words, $\xi-H_f$ fixes $x\in X$.
\end{proof}

We now prove the Darboux theorem for symplectic singularities.

\begin{proof}[Proof of Theorem \ref{thm:Darboux}]
The basic ingredients are Moser's trick \cite{Moser-trick} and the exactness of the symplectic forms proved by Kaledin \cite{Kaledin-sym}, though the singularity introduces some mild subtleties. 

For $\tau\in\bC$, let $\sigma_\tau := (1-\tau)\sigma_0 + \tau \sigma_1$. Note that $\sigma_\tau$ is a symplectic form if and only if the induced section $\sigma_\tau^{n/2}\in H^0(X,\omega_X)$ is non-vanishing at $x$, which is the case when $\tau=0,1$. Thus, there are only finitely many values of $\tau\in\bC$ for which $\sigma_\tau$ is not symplectic. Hence to prove the theorem, we may assume that $\sigma_t$ is symplectic for all $t\in [0,1]$; this is because for general $\tau\in \bC$ and $\sigma'=\sigma_\tau$, any convex combination of $\sigma_j$ ($j=0,1$) and $\sigma'$ is symplectic, and it is enough to find an automorphism that takes $\sigma'$ to $\sigma_i$ for each $i$.

By \cite[Corollary 2.8]{Kaledin-sym}, possibly after shrinking $X$ around $x$, there exists a reflexive $1$-form $\alpha$ on $X$ such that 
\begin{equation} \label{eq:symp form exact}
\sigma_1 - \sigma_0 = \rd \alpha.    
\end{equation}
Consider the vector field $\xi'_t$ on $X$ defined by $\alpha = \iota_{\xi'_t}(\sigma_t)$. Replacing $\alpha$ by $\alpha-\rd f$ for $f\in \cO_{X,x}$ preserves the condition \eqref{eq:symp form exact}, while the corresponding vector field changes to $\xi'_t-H_f$, where $H_f$ is the Hamiltonian vector field with respect to $\sigma_t$. Therefore, by Lemma \ref{lem:Poicare}, we can choose an algebraic family of reflexive $1$-forms $\alpha_t$ on $X$ such that
\[
\rd \alpha_t = \sigma_1-\sigma_0 = \frac{\rd \sigma_{t}}{\rd t},
\]
and the vector field $\xi_t$ defined by $\alpha_t = \iota_{\xi_t}(\sigma_t)$ fixes $x$ for all $t$. 

As $x$ is fixed by all vector fields $\xi_t$, we know from Lemma \ref{lem:vector field->auto} that there exists a one-parameter family of automorphisms $\varphi_t\in \Aut(X^{\an},x)$, $t\in [0,1]$, such that $\varphi_0=\mathrm{id}$ and $\frac{\rd \varphi_t}{\rd t} = \xi_t$. We then have 
\[
\frac{\rd}{\rd t} \varphi_t^*\sigma_0 = \cL_{\xi_t} (\varphi_t^*\sigma_0) = \rd\big(\iota_{\xi_t}(\varphi_t^*\sigma_0)\big),
\]
where the second equality is by Cartan's formula. Since we also have $\frac{\rd \sigma_t}{\rd t} = \rd \big(\iota_{\xi_t}(\sigma_t)\big)$ and $\varphi^*_0 \sigma_0 = \sigma_0$, we must have $\varphi^*_t \sigma_0 = \sigma_t$ for all $t\in [0,1]$. In particular, $\varphi_1^*\sigma_0 = \sigma_1$.
\end{proof}

\markboth{HENRI GUENANCIA}{STABLE DEGENERATION AND KALEDIN'S CONJECTURE}
\clearpage
\appendix
\setcounter{section}{1}

\section{Slope stability of the canonical extension of log Fano pairs} \label{s:can ext ss}

\centerline{\small{\textsc{Henri Guenancia}}}

\medskip

The aim of this appendix is to show the following result:

\begin{thm}
\label{thm:can ext appendix}
Let $(X,\Delta)$ be a log Fano pair and let $E_{X,\Delta}$ be the canonical extension of $\cO_X$ by $\Omega_X^{[1]}(\log \Delta)$.
\begin{enumerate}
\item If $(X,\Delta)$ is K-polystable, then $E_{X,\Delta}$ is slope polystable with respect to $-c_1(K_X+\Delta)$.
\item If $(X,\Delta)$ is K-semistable, then $E_{X,\Delta}$ is slope semistable with respect to $-c_1(K_X+\Delta)$.
\end{enumerate}
\end{thm}

This is a generalization of previous results obtained in \cite{Tian-extension} (when $X$ is smooth and $\Delta=0$), \cite{Li-extension-stab} (when $(X,\Delta)$ is log smooth), \cite{DGP-Q-Fano-decomp} (when $\Delta=0$) and \cite{Dai-stability} (when $\Delta$ has standard coefficients). We will mostly follow the approach of \cite{DGP-Q-Fano-decomp} but a few additional technical inputs will be needed along the way to deal with codimension one singularities, e.g. Lemma~\ref{lem1}, the slope polystability statement and the K-semistable case (where the existence of approximate K\"ahler-Einstein metrics as in \cite{Li-valuative-criterion} does not seem to have appeared in the literature yet).

\subsection{Setup}

Let $X$ be a normal projective complex variety of dimension $n$ and let $\Delta=\sum_{i\in I} d_i\Delta_i$ be an effective $\mathbb Q$-divisor. Write $d_i=\frac{b_i}{a_i}$ with $\mathrm{gcd}(a_i,b_i)=1$ and set $N:=\mathrm{lcm}\{a_i, i\in I\}$. We assume that $(X,\Delta)$ is klt and that $-(K_X+\Delta)$ is an ample $\Q$-divisor so that $(X,\Delta)$ is a log Fano pair. Set $\alpha:=c_1(-(K_X+\Delta))\in H^2(X,\R)$. We denote by $(X,\Delta)_{\rm reg}$ the big Zariski open subset of $X$ where $X$ is smooth and $\Delta$ has simple normal crossings. \\

\subsection{Varieties and morphisms}
We fix a log resolution $\pi:\tX\to X$ of the pair $(X,\Delta)$. That is, if $\tDe$ denotes the strict transform of $\Delta$, then $(\tX, \tDe)$ is log smooth, $\pi$ is an isomorphism over $(X,\Delta)_{\rm reg}$ and the exceptional locus of $\pi$ is a divisor $F=\sum_{j\in J} F_j$ such that the support of $\tDe+F$ has simple normal crossings. We write 
\begin{equation}
\label{disc}
K_{\tX}+\tDe=\pi^*(K_X+\Delta)+\sum_{j\in J} c_j F_j
\end{equation}
where $c_j>-1$ for all $j\in J$. 

Choose a sufficiently ample divisor $A$ on $\tX$ and set $B:=\sum_{i\in I} \tDe_i+A$. Then there exists a finite surjective morphism $f:\tY\to \tX$ satisfying the following properties. 
\begin{enumerate}
\item  The pair $(\tY, f^{-1}(B))$ is log smooth.  
\item $f$ is étale over the complement of $B$.
\item Near any point $\tilde y \in f^{-1}(B)$,  there exists a system of coordinates $(w_k)$ (resp. $(z_k)$) centered at $\tilde y$ (resp. $f(\tilde y)$) and an integer $p=p(\tilde y)$ such that with respect to these coordinates, the map $f$ can be expressed as
\[f(w_1, \ldots, w_n)=(w_1^N, \ldots, w_p^N, w_{p+1}, \ldots, w_n)\]
where for each $1\le k \le p$, $(z_k=0)$ is a local equation near $f(\tilde y)$ of one of the irreducible components of $B$. 
\end{enumerate}
We denote by $\tY\overset{\mu}{\to} Y\overset{g}{\to} X$ the Stein factorization of $\pi \circ f$ so that we have a commutative diagram
\[
\begin{tikzcd}
\tY \arrow[r, "f"] \arrow[d, "\mu"] & \tX \arrow[d, "\pi"]\\
Y \arrow[r, "g"] & X
\end{tikzcd}
\] 
Clearly, $g$ is adapted for $(X,\Delta)$ in the sense that $g$ is a finite surjective morphism and $g^*\Delta$ is an integral Weil divisor.

\subsection{Vector bundles}
By definition, the vector bundle $\Omega^1_{(\tX,\tDe, f)}$ on $\tY$ is squeezed between $f^*\Omega_{\tX}^1$ and  $\Omega_{\tY}^1$ as below
\begin{equation}
\label{sandwich}
 f^*\Omega_{\tX}^1 \hookrightarrow \Omega^1_{(\tX,\tDe, f)} \hookrightarrow \Omega_{\tY}^1
 \end{equation}
where both maps are isomorphisms away from $f^{-1}(B)$. We denote by $\tE:=E_{(\tX, \tDe, f)}$ the vector bundle on $\tY$ which is the extension of $\cO_{\tY}$ by $\Omega^1_{(\tX,\tDe, f)}$ whose extension class is the image of $\pi^*\alpha$ under the natural map 
\begin{equation}
\label{pullback}
r:H^1(\tX,\Omega_{\tX}^1)\to H^1(\tY, \Omega^1_{(\tX,\tDe, f)})
\end{equation} induced by the inclusion map $f^*\Omega_{\tX}^1 \to \Omega^1_{(\tX,\tDe, f)}$. In particular, we have an exact sequence of vector bundles on $\tY$
\[0\longrightarrow  \Omega^1_{(\tX,\tDe, f)} \longrightarrow \tE \longrightarrow \cO_{\tY} \longrightarrow 0\]
which splits $\tau_i:\tE|_{V_i}\overset{\simeq}{\longrightarrow} \cO_{\tY}|_{V_i} \oplus  \Omega^1_{(\tX,\tDe, f)}|_{V_i}$ in restriction to each element of a covering family $(V_i)$ of open subsets of $\tY$ arising as the inverse image $V_i=f^{-1}(U_i)$ of some open subsets $U_i\subset \tX$ . Moreover, the transition maps on the overlaps $V_{ij}$ are of the form
\begin{equation}
\label{transition}
\Phi_{ij}:=
\begin{pmatrix}
\mathrm{Id}_{\cO_{\tY}|_{V_i}} &0 \\
a_{ij}& \mathrm{Id}_{\Omega^1_{(\tX,\tDe, f)}|_{V_i}}
\end{pmatrix}
\end{equation}
where $a_{ij}\in H^0(V_{ij},\Omega^1_{(\tX,\tDe, f)})$ represents $r(\pi^*\alpha)$. 

We will later need to work with a deformation of $\tE$ that we now describe. Up to shrinking $U_i$, one can assume that for each $\sigma \in J$, $F_\sigma|_{U_i}=(f_{\sigma i}=0)$ where $f_{\sigma i}\in \cO_{\tY}(U_i)$. The transition functions of $\cO_{\tX}(F_\sigma)$ are $g_{\sigma ij}:=\frac{f_{\sigma j}}{f_{\sigma i}}$ and $c_1(F_\sigma)$ is represented by $\frac{dg_{\sigma ij}}{g_{\sigma ij}}$. Given $z:=(z_\sigma)_{\sigma \in J}\in \mathbb C^J$, consider the extension $\tE_z$ of $\cO_{\tY}$ by $\Omega^1_{(\tX,\tDe, f)}$ which is split over $V_{i}$ and has transition maps
\[\Phi_{ij}^z:=
\begin{pmatrix}
\mathrm{Id}_{\cO_{\tY}|_{V_i}} &0 \\
a_{ij}+f^*(\sum_{\sigma \in J} z_\sigma \frac{dg_{\sigma ij}}{g_{\sigma ij}})& \mathrm{Id}_{\Omega^1_{(\tX,\tDe, f)}|_{V_i}}
\end{pmatrix}
\]
on $V_{ij}$. The extension class of $\tE_z$ is $r(\pi^*\alpha+\sum_{\sigma \in J}z_\sigma c_1(F_\sigma))$. Moreover, the injective maps $\Psi_i^z:\tE|_{V_i}\to \tE_z|_{V_i}\otimes f^*\cO_{\tX}(F)$ given via the splitting $\tau_i$ by
\[\Psi_i^z:=
\begin{pmatrix}
\mathrm{Id}_{\cO_{\tY}|_{V_i}} &0 \\
f^*(\sum_{\sigma \in J} z_\sigma \frac{df_{\sigma i}}{f_{\sigma i}})& \mathrm{Id}_{\Omega^1_{(\tX,\tDe, f)}|_{V_i}}
\end{pmatrix}
\]
satisfy $\Phi_{ij}^z \circ \Psi_i^z=\Psi_j^z \circ \Phi_{ij}$ on $V_{ij}$ thanks to the identity  $\frac{dg_{\sigma ij}}{g_{\sigma ij}}=\frac{df_{\sigma j}}{f_{\sigma j}}-\frac{df_{\sigma i}}{f_{\sigma i}}$. That is, the maps $\Psi_i^z$ glue to a global injective morphism 
\begin{equation}
\label{inj}
\Psi^z:\tE\longhookrightarrow \tE_z\otimes f^*\cO_{\tX}(F)
\end{equation} over $\tY$.

Similarly, one can define a reflexive sheaf $E:=E_{(X,\Delta,g)}$ as the extension of $\cO_Y$ by $\Omega^{[1]}_{(X,\Delta,g)}$ with extension class the image of $-(K_X+\Delta)$ under the map 
\begin{equation}
\label{ext map}
\mathrm{Pic}(X)_{\Q}\to H^1(X,\Omega_X^1) \overset{g^{[*]}}{\to}H^1(Y,\Omega_{(X,\Delta, g)}^{[1]})
\end{equation} 
where the first map is the composition of $d\log:\cO_X^*\to \Omega_X^1$ while the second one is induced by the composition $f^*\Omega_X^1\to f^{[*]}\Omega_X^1$ and the injection $f^{[*]}\Omega_X^1 \subset \Omega_{(X,\Delta, g)}^{[1]}$. 

We say that $E_{X,\Delta}$ is semistable (resp. polystable) with respect to $\alpha$ if $E=E_{(X,\Delta,g)}$ is semistable (resp. polystable) with respect to $g^*\alpha$. This notion is independent of the particular choice of the adapted cover as one sees e.g. using \cite[Lemmas~3.2.2 and 3.2.3]{HL-book-sheaf}. The local nature of the formation of the canonical extension (see \cite[Remark~7]{DGP-Q-Fano-decomp}) shows that over the big Zariski open set $U:=g^{-1}((X,\Delta)_{\rm reg})$, $\mu$ is an isomorphism and induces an identification $\mu^*E|_U\simeq \tE|_{\mu^{-1}(U)}$. Now, if $G\subset E$ is a coherent subsheaf, then one can extend $\mu^*G|_U$ to a coherent sheaf $\tilde G \subset \tE$. By construction, one has that $c_1(\tilde G)-\mu^*c_1(G)$ is a linear combination of $\mu$-exceptional divisors hence $\mu_{f^*\pi^*\alpha}(\tilde G)=\mu_{g^*\alpha}(G)$. In particular
\begin{equation}
\label{ss}
\tE \,\, \mbox{is} \,\, f^*\pi^*\alpha \mbox{-semistable} \quad \Longrightarrow \quad E\,\, \mbox{is} \,\, g^*\alpha \mbox{-semistable}.
\end{equation}

\subsection{Metrics} 
\label{metrics}
We fix a smooth K\"ahler metric $\omega_X\in \alpha=-c_1(K_X+\Delta)$. 
We assume from now on that there exists a K\"ahler-Einstein metric $\omega$ on $(X,\Delta)$, i.e. there exists a positive current $\omega=\omega_X+dd^c \varphi$ for some $\varphi \in \mathrm{PSH}(X,\omega_X) \cap L^{\infty}(X)$ solving
\begin{equation}
\label{KE}
\mathrm{Ric}(\omega)= \omega+[\Delta]
\end{equation}
where $[\Delta]$ is the current of integration along the $\Q$-divisor $\Delta$. In other words, the bounded $\omega_X$-psh function $\varphi$ is a solution of the complex Monge-Ampère equation
\begin{equation}
\label{MA}
(\omega_X+dd^c \varphi)^n=e^{-\varphi} \mu_{(X,\Delta,h)} 
\end{equation}
where $h$ is a smooth hermitian metric on the $\Q$-line bundle $K_X+\Delta$ with Chern curvature $\Theta(K_X+\Delta,h)=-\omega_X$ and $\mu_{(X,\Delta, h)}$ is the adapted measure on $(X,\Delta)$ relative to $h$. That is, if $\sigma$ is a local trivialization of $m(K_X+\Delta)$, then $\mu_{(X,\Delta,h)}$ is locally given by $i^{n^2}\frac{\sigma^{1/m} \wedge\, \overline{\sigma^{1/m}} }{|\sigma|^{2/m}_{h^{ m}}}$, where $\sigma$ is viewed as a local meromorphic section of $K_{X_{\rm reg}}$ with poles of order $md_i$ along $\Delta_i\cap X_{\rm reg}$. Here, $m$ is a sufficiently large and divisible integer and $\mu_{(X,\Delta,h)}$ is well-defined globally on $X_{\rm reg}\setminus \Supp(\Delta)$. The klt condition is equivalent to 
$\mu_{(X,\Delta,h)}$ having finite mass on $X_{\rm reg}\setminus \Supp(\Delta)$; we extend the measure trivially to the whole $X$. Thanks to \cite[Theorem~2]{GP-conic}, the restriction $\omega|_{(X,\Delta)_{\rm reg}}$ is a K\"ahler metric with cone singularities along $\Delta$. \\

Fix sections $t_j\in H^0(\tX, \mathcal O_{\tX}(F_j))$ such that $(t_j=0)=F_j$ and pick smooth hermitian metrics $h_{F_j}$ on $\mathcal O_{\tX}(F_j)$. Through \eqref{disc}, we get a smooth hermitian metric $\tilde h$ on $K_{\tX}+\tDe$ such that the pullback of \eqref{MA} to $\tX$ becomes
\begin{equation}
\label{MA2}
(\pi^*\omega_X+dd^c \pi^*\varphi)^n=e^{-\pi^*\varphi}\prod_{j\in J} |t_j|^{2c_j} \mu_{(\tX,\tDe,\tilde h)}.
\end{equation}
where $|t_j|^{2}$ has to be understood as $|t_j|^{2}_{h_{F_j}}$. There exist positive numbers $(\ep_j)_{j\in J}$ such that there exists a K\"ahler metric
 \[\omega_{\tX}\in \pi^*\alpha-\sum_{j\in J} \ep_j c_1(F_j).\] 
 Thanks to Demailly's regularization theorem there exists a family of smooth $C\omega_X$-psh functions $\psi_{\ep}$ such that $\psi_{\ep}$ decreases pointwise to $\pi^*\varphi$. Moreover, $\psi_\ep$ converges locally smoothly to $\pi^*\varphi$ on $\pi^{-1}(X_{\rm reg}\setminus \Supp(\Delta))$. For any $t,\ep>0$, consider the bounded $(\pi^*\omega_X+t\omega_{\tX})$-psh function $\varphi_{t,\ep}$ solution of
\begin{equation}
\label{MA3}
(\pi^*\omega_X+t\omega_{\tX}+dd^c \varphi_{t,\ep})^n=e^{-\psi_\ep+c_{t,\ep}}\prod_{j\in J} (|t_j|^{2}+\ep^2)^{c_j} \mu_{(\tX,\tDe,\tilde h)}.
\end{equation}
where $c_{t,\ep}\in \R$ is a harmless normalization constant and set 
\[\omega_{t,\ep}:=\pi^*\omega_X+t\omega_{\tX}+dd^c \varphi_{t,\ep}, \quad \alpha_t:=\pi^*\alpha+t[\omega_X]=(1+t)\pi^*\alpha-t\sum_{j\in J} \ep_j c_1(F_j).\] 
The Ricci curvature of $\omega_{t,\ep}$ is given by 
\begin{equation}
\label{Ricci}
\mathrm{Ric}(\omega_{t,\ep})=\omega_{t,\ep}-t \omega_{\tX}+ dd^c(\psi_\ep-\varphi_{t,\ep})-\Theta_\ep+[\tDe]
\end{equation}
where $\Theta_\ep=\sum_{j\in J} c_j  \Theta_{\ep, j}$ and $\Theta_{\ep, j}:=\Theta(F_j, h_{F_j})+dd^c \log (|t_j|^2+\ep^2)$. Note that there exists $C>0$ such that for all $j\in J$, we have
\begin{equation}
\label{min theta}
\Theta_{\ep,j}+g_\ep \omega_{\tX}\ge 0, \quad \mbox{where} \quad g_\ep:=C\sum_{j\in J} \frac{\ep^2}{\ep^2+|t_j|^2}.
\end{equation}

By \cite[Theorem~A]{GP-conic}, $\omega_{t,\ep}$ is a K\"ahler metric with cone singularities along $\tDe$. In particular, $f^*\omega_{t,\ep}$ induces a singular hermitian metric $h_{t,\ep}$ on $\Omega_{(\tX, \tDe, f)}$ which is smooth away from $f^{-1}(B)$ and {\it bounded} globally on $\tY$. That is, given a smooth hermitian metric $\tilde h$ on $\Omega_{(\tX, \tDe, f)}^1$, we have 
\begin{equation}
\label{bounded}
C\tilde h \ge h_{t,\ep}\ge C^{-1}\tilde h
\end{equation} for some $C>0$. In particular, we have 
\begin{equation}
\label{upper bound 0}
f^*\omega_{t,\ep} \le C_{t, \ep}  \omega_{\tY},
\end{equation}
for some K\"ahler metric $\omega_{\tY}$ on $\tY$. Moreover, as $\ep\to 0$ but $t>0$ is fixed, we still get a domination of the form
\begin{equation}
\label{upper bound}
\omega_{t,\ep} \le C_t \, \omega_{\rm cone},
\end{equation}
where $\omega_{\rm cone}$ is a K\"ahler metric on $\tX\setminus \Supp(\tDe+F_-)$ with cone singularities along $\tDe+F_-$, where $F_-:=\sum_{c_j<0}(-c_j)F_j$, and $C_t>0$ is a constant independent of $\ep>0$. This is a consequence of the proof of \cite[Theorem~2]{GP-conic}, see also \cite[Proposition~4.1]{G12}.

\subsection{Slope computations and proof of the theorem}
\label{slope}
 Let $\tF\subset \tE$ be a saturated subsheaf of rank $r$, inducing a morphism of vector bundles $j:L\to \Lambda^r\tE$ where $L:=\det\tF=(\Lambda^r \tF)^{**}$ is a line bundle and $j$ is injective as a map of sheaves. We define $S\subset \tY$ to be the Zariski closed subset where $\tF$ is not a subbundle of $\tE$; we have $\mathrm{codim}(S)\ge 2$ since $\tF$ is saturated. 
 
 Given $t>0$, denote by $\tE_t$ the extension of $\cO_{\tY}$ by $\Omega^1_{(\tX,\tDe, f)}$ with class $\frac{1}{1+t}r(\alpha_t)$.  From \eqref{inj}, we get an injective morphism 
 \begin{equation}
 \label{inj2}
 j_t:L\to \Lambda^r \tE_t \otimes f^*\cO_{\tX}(rF).
 \end{equation}
 Note that although $j_0$ is injective (as a vector bundle morphism) away from $S$, $j_t$ is a priori only injective (as a vector bundle morphism) away from $S\cup F$.  The datum consisting of the (singular) metric $h_{t,\ep}$ on $\Omega_{(\tX, \tDe, f)}^1$, the trivial metric $h_{\cO_{\tY}}$ on $\cO_{\tY}$, the element $\beta_{t,\ep}:=\frac{1}{1+t}f^*\omega_{t,\ep}$ and a fixed $\mathcal C^\infty$ splitting of $\tE_t$ induce a (singular) metric $\tilde h_{t,\ep}$ on $\tE_t$ such that its Chern connection $\nabla_{t,\ep}$ is given by
 \[\nabla_{t,\ep}=
 \begin{pmatrix}
 d & -\beta_{t,\ep}\\
 \beta_{t,\ep}^* & \nabla_{h_{t,\ep}}
 \end{pmatrix}
 \]
where we see $\beta_{t,\ep}$ as a bounded $(0,1)$-form with values in $\Omega_{(\tX, \tDe, f)}^1$, see \eqref{sandwich}. Moreover, the restriction of $\beta_{t,\ep}$ to the complement of $f^{-1}(B)$ is an honest smooth $(0,1)$-form with values in $\Omega_{\tY}^1$. Now, let $h_F$ be any smooth hermitian metric on $\cO_{\tX}(F)$. The metrics $\tilde h_{t,\ep}$ and $h_F$ induce a (singular) hermitian metric 
\begin{equation}
\label{Vt}
h:=\Lambda^r \tilde h_{t,\ep} \otimes f^*h_F^{\otimes r} \quad \mbox{on} \quad V_t:=\Lambda^r \tE_t \otimes f^*\cO_{\tX}(rF).
\end{equation}
We have $c_1(V_t)=r(c_1(\tE_t)+{n+1 \choose r} c_1(f^*F))$. Through the map $j_t$, $h$ induces a (singular) metric $j_t^*h$ on $L$ which is smooth outside $S\cup f^{-1}(B+F)$. 
We have Griffiths' formula holding away from $S\cup f^{-1}(B+F)$:
\begin{equation}
\label{Gr}
\Theta(L, j_t^*h)=\mathrm{pr}_L \Theta(V_t, h)|_L+\tau_{t,\ep}\wedge \tau_{t,\ep}^*
\end{equation}
where $\tau_{t,\ep}$ is the second fundamental form of $j_t$; it is a smooth $(0,1)$-form with values in $\mathrm{Hom}(V_t,L)$ in the complement of $S\cup f^{-1}(B+F)$ while its adjoint $\tau_{t,\ep}^*$ is taken with respect to $h$ and $h|_L$. 
Next, we analyze the terms inside \eqref{Gr}. 

\begin{lem}
\label{lem1}
There exists a family $(\chi_\delta)_\delta$ of cut-off functions for $S\cup f^{-1}(B+F)$ such that for any $t,\ep>0$, we have
\[c_1(\tF) \cdot f^*\alpha_t^{n-1}= \lim_{\delta \to 0} \int_{\tY} \chi_\delta \, \Theta(L, j_t^*h) \wedge f^*\omega_{t,\ep}^{n-1}+O(t).\]
Moreover, the term $O(t)$ is non-positive and independent of $\ep$.
\end{lem}

\begin{proof}
The metric $j_t^*h$ develops two kinds of singularities, due to the facts that $j_t$ need not be injective as a map of vector bundles (along $S\cup F$) and that $h_{t,\ep}$ (hence $h$, too) may be (mildly) singular along $f^{-1}(B)$. We fix a smooth hermitian metric $h_L$ (resp. $h_{V_t}$) on $L$ (resp. $V_t$), write 
\[j_t^*h=h_Le^{\psi}\]
and consider the function $\|j_t\|^2$ which is the squared norm of the morphism $j_t$ with respect to $h_{V_t}$ and $h_L$. In other words, given any $e\in L\setminus \{0\}$, we have $\|j_t\|^2=\frac{|j_t(e)|^2_{h_{V_t}}}{|e|^2_{h_L}}$. We also consider the function $\rho:=\log \frac{|j_t(e)|^2_{h }}{|j_t(e)|^2_{h_{V_t}}}$. All in all, we have 
\[\psi=\log \|j_t\|^2+\rho\]
and $\rho \in L^{\infty}(X)$ thanks to \eqref{bounded}.

Let $\varrho: Z\to \tY$ be a log resolution of the ideal of $S\cup f^{-1}(B+F)$, with exceptional divisor $\Gamma=\sum_{k\in K} \Gamma_k$. It is elementary to check that we have a decomposition
\[\varrho^*\log \|j_t\|^2=\sum_{k\in K} \lambda_k \log |s_{\Gamma_k}|^2+\sum_{j\in J} \nu_j\log |s_{F'_j}|^2+ \sigma\]
where $\lambda_k \ge 0$ (resp. $\nu_j \ge 0$), $s_{\Gamma_k}$ (resp. $s_{F_j'}$) is ``the'' canonical section of $\cO_Z(\Gamma_k)$ (resp. $\cO_Z(F_j')$ where $F_j'$ is the strict transform of $f^{-1}(F_j)=f^*F_j$), measured with respect to some fixed smooth hermitian metric $h_{\Gamma_k}$ (resp. $h_{F_j'}$), and $\sigma \in \mathcal C^{\infty}(Z)$. In particular, we find on the complement of $\Gamma\cup F'$ in $Z$:
\begin{equation}
\label{id}
\varrho^*\Theta(L, j_t^*h)=\varrho^*\Theta(L,h_L)+\sum_{k\in K}\lambda_k\Theta(\Gamma_k, h_{\Gamma_k})+\sum_{j\in J}\nu_j\Theta(F_j', h_{F_j'})+dd^c(\varrho^*\rho+\sigma).
\end{equation}
Let $(\chi_\delta)_\delta$ be the family of cut-off functions for $\Gamma \cup F'\subset Z$ considered e.g. in \cite{GP-conic}. It can be interchangeably viewed as a function on $Z$ or $\tY$. It has the property that $\pm dd^c \chi_\delta \le \omega_\Gamma$ where $\omega_\Gamma$ is a K\"ahler metric on $Z\setminus \Gamma$ with Poincar\'e type singularities along $\Gamma$. Also, recall that for $t,\ep>0$ fixed, the metric $f^*\omega_{t,\ep}$ is dominated by a K\"ahler metric on $\tY$. In particular, an integration by part yields 
\begin{equation}
\label{int zero}
\left| \int_{Z} \chi_\delta \, dd^c(\varrho^*\rho+\sigma) \wedge \varrho^*f^*\omega_{t,\ep}^{n-1} \right| \le C_{t,\ep}\cdot \|\varrho^*\rho+\sigma\|_{L^{\infty}(Z)}\cdot  \int_{\mathrm{Supp}(d\chi_\delta)} \omega_\Gamma^n 
\end{equation}
and the latter goes to zero when $\delta \to 0$. 

Now, since $\omega_{t,\ep}$ has bounded potentials, we have
\[c_1(L) \cdot f^*\alpha_t^{n-1}=\int_{\tY} \Theta(L,h_L) \wedge f^*\omega_{t,\ep}^{n-1} = \lim_{\delta \to 0} \int_Z \chi_\delta \varrho^*\Theta(L,h_L) \wedge \varrho^*f^*\omega_{t,\ep}^{n-1} \]
and by \eqref{id} and \eqref{int zero}, we infer that $c_1(L) \cdot f^*\alpha_t^{n-1}$ is computed by the quantity
\[\lim_{\delta \to 0} \int_{\tY}\chi_\delta \, \Theta(L, j_t^*h)\wedge f^*\omega_{t,\ep}^{n-1}-\sum_{k\in K} \lambda_k \underbrace{c_1(\Gamma_k) \cdot \varrho^*f^*\alpha_t^{n-1}}_{=0}-\mathrm{deg}(f) \sum_{j\in J} \nu_j c_1(F_j) \cdot \alpha_t^{n-1},\]
where the vanishing of the intersection numbers above holds since the $\Gamma_k$'s are $\varrho$-exceptional. The lemma now follows from the fact that $c_1(F_j) \cdot \alpha_t^{n-1}$ is a non-negative $O(t)$ term since $F_j$ is $\pi$-exceptional and $\alpha_t=\pi^*\alpha+O(t)$. 
\end{proof}

\begin{lem}
\label{lem2}
One has 
\[\limsup_{t\to 0} \limsup_{\ep \to 0} \limsup_{\delta \to 0} \int_{\tY} \chi_\delta \, \mathrm{tr}(\mathrm{pr}_L \Theta(V_t, h)|_L) \wedge f^*\omega_{t,\ep}^{n-1} \le-\frac{r}{n+1}(f^*\alpha^n). \]
\end{lem}

\begin{proof}
The Chern curvature form of the metric $\tilde h_{t,\ep}$ on $\tE_t$ is given by
 \[\Theta(\tE_t,\tilde h_{t,\ep})=
 \begin{pmatrix}
  \beta_{t,\ep}\wedge \beta_{t,\ep}^* & -\nabla^{1,0}_{h_{t,\ep}}\beta_{t,\ep}\\
 -\bar \partial  \beta_{t,\ep}^* & \Theta(\Omega_{(\tX, \tDe, f)}^1, h_{t,\ep}) +\beta_{t,\ep}^*\wedge \beta_{t,\ep}
 \end{pmatrix}
 \]
 away from $f^{-1}(B)$. Arguing as in \cite[p. 105]{DGP-Q-Fano-decomp} one infers that, up to replacing $\beta_{t,\ep}$ by $\frac{1+t}{\sqrt{n+1}}\beta_{t,\ep}$\footnote{The constants in \cite[(32)-(33)]{DGP-Q-Fano-decomp} and the equation below are actually slightly off.}, one has 
 \begin{equation}
 \label{curv}
 \Theta(\tE_t,\tilde h_{t,\ep})\wedge f^*\omega_{t,\ep}^{n-1}=-\frac 1{n+1} f^*\omega_{t,\ep}^{n} \otimes \mathrm{Id}_{\tE_t}-f^*\omega_{t,\ep}^{n} \otimes A_{t,\ep}
 \end{equation}
 where \[A_{t,\ep}=\sharp_{f^*\omega_{t,\ep}}f^*\big[-t \omega_{\tX}+dd^c(\psi_\ep-\varphi_{t,\ep})-\Theta_\ep\big],\]
see \eqref{Ricci}. The rest also follows very closely \cite{DGP-Q-Fano-decomp} (see also \cite{GSS}), with the only difference being the appearance of the function $\chi_\delta$ to cut off the singularities of $f^*\omega_{t,\ep}$. The key elementary result that we need is that if $\alpha$ is a positive $(1,1)$-form on $\tX$, then 
\begin{equation}
\label{lin alg}
\|\sharp_{\omega_{t,\ep}} \alpha\|_{\omega_{t,\ep}}\,  \omega_{t,\ep}^n  \le  \mathrm{tr}(\sharp_{\omega_{t,\ep}} \alpha) \,  \omega_{t,\ep}^n =n\alpha \wedge \omega_{t,\ep}^{n-1}.
\end{equation}
There are $4$ terms in the expansion of $A_{t,\ep}$, say $A_1, \ldots, A_4$. Clearly, one has 
\[ \|A_1+A_2\|\, f^*\omega_{t,\ep}^n \le Ct \cdot f^*(\omega_{\tX}\wedge \omega_{t,\ep}^{n-1})\]
Next, the argument given in \cite[pp.~99-100, ``The term (I)'']{DGP-Q-Fano-decomp} shows that 
\[\limsup_{\ep \to 0} \int_{\tY} \|A_3\| f^*\omega_{t,\ep}^n=0.\]
As for $A_4$, \eqref{upper bound 0} and \eqref{lin alg} combined with the Lebesgue dominated convergence theorem show that for fixed $\ep>0$, we have
\[\lim_{\delta \to 0} \int_{\tY} \chi_\delta \|A_4\| f^*\omega_{t,\ep}^n =  \int_{\tY}  \|A_4\| f^*\omega_{t,\ep}^n.\]
Then we write $\Theta_\ep=\sum_{j\in J} c_j (\Theta_{\ep, j}+g_\ep \omega_{\tX})-\big(\sum_{j\in J} c_j\big) g_\ep \omega_{\tX}$ and use \eqref{lin alg} to see that
\[\int_{\tY} \|A_4\|f^*\omega_{t,\ep}^n \le  n\cdot \mathrm{deg}(f) \left(\sum_{j\in J}|c_j| (c_1(F_j) \cdot \alpha_t^{n-1})+\Big| \sum_{j\in J} c_j \Big| \int_{\tX} g_\ep \omega_{\tX} \wedge \omega_{t,\ep}^{n-1}\right)\]
On the RHS, the first summand is independent of $\ep$ and goes to zero as $t\to 0$ while the second summand goes to zero when $\ep \to 0$ and $t>0$ is fixed thanks to \eqref{upper bound} and Lebesgue's dominated convergence theorem. Putting all the above estimates together  yields 
\[ \lim_{t\to 0} \lim_{\ep \to 0} \lim_{\delta \to 0} \int_{\tY} \chi_\delta \, \| A_{t,\ep}\| \, f^*\omega_{t,\ep}^{n} =0.\]
Given the definition of $(V_t,h)$ (see \eqref{Vt}) and the fact that $F$ is $\pi$-exceptional, we deduce from \eqref{curv} and the above identity that 
\[\limsup_{t\to 0} \limsup_{\ep \to 0} \limsup_{\delta \to 0} \int_{\tY} \chi_\delta \, \mathrm{tr}(\mathrm{pr}_L \Theta(V_t, h)|_L) \wedge f^*\omega_{t,\ep}^{n-1}\le-\frac{r}{n+1}\cdot \deg(f) \cdot \alpha^n,\]
hence the lemma.
\end{proof}

\begin{proof}[Proof of Theorem~\ref{thm:can ext appendix}]
We deal with the two cases of the theorem separately.

\smallskip

$(i)$ Here we assume that the pair $(X,\Delta)$ is K-polystable. Thanks to \cite[Theorem~5.2]{LXZ-HRFG} (see also \cite[Theorem~1.2]{Li-YTD}), there exists a K\"ahler-Einstein metric $\omega$ on $(X,\Delta)$. Therefore, one can use the approximations $\omega_{t,\ep}$ of $\pi^*\omega$ on $\tX$ introduced in Section~\ref{metrics} to compute the slopes of subsheaves $\widetilde F\subset \tE$ as explained in Section~\ref{slope}. More precisely, one combines Lemma~\ref{lem1}, Lemma~\ref{lem2} and \eqref{Gr} and infers that $\tE$ is semistable with respect to $f^*\pi^*\alpha$. Thanks to \eqref{ss}, this implies that $E$ is semistable with respect to $g^*\alpha$.

\medskip

It remains to show polystability. Let $G\subset E$ be a subsheaf such that $\mu_{g^*\alpha}(G)=\mu_{g^*\alpha}(E)$. As explained in the paragraph above \eqref{ss}, $G$ induces a saturated subsheaf $\tilde G\subset \tE$ such that $\mu_{f^*\pi^*\alpha}(\tilde G)=\mu_{f^*\pi^*\alpha}(\tE)$. From the proof of semistability of $\tE$, one infers the vanishing
\begin{equation}
\label{II}
\limsup_{t\to 0}\limsup_{\ep\to0} \int_{\tY^\circ} \mathrm{tr}(\tau_{t,\ep}\wedge \tau_{t,\ep}^*)\wedge f^*\omega_{t,\ep}^{n-1}=0,
\end{equation}
where $\tY^\circ:=\tY\setminus (S\cup f^{-1}(B+F))$. Note that the integrand above is pointwise non-negative and vanishes if and only if $\tau_{t,\ep}$ does. Classical stability estimates in Monge-Amp\`ere theory show that $\omega_{t,\ep}\to \pi^*\omega$ locally smoothly as tensors on $\pi^{-1}(X_{\rm reg}\setminus \Supp(\Delta))$, a locus on which $\omega$ (or $\pi^*\omega$) is known to be a genuine K\"ahler metric. Because of \eqref{II}, the second fundamental form of $j_0:L\to V_0=\Lambda^r\tE\otimes f^*\cO_{\tX}(rF)$ (see \eqref{inj2}-\eqref{Vt}) with respect to the metrics induced by $f^*\pi^*\omega$ and $h_F$ vanishes on $\tY^\circ$. In particular, $L\subset V_0$ is a parallel subbundle on $\tY^\circ$ (with respect to $j_0$ and the metrics $f^*\pi^*\omega$ and $h_F$), hence $L\subset \Lambda^r\tE$ is parallel, too (with respect to $j$ and the metric induced by $f^*\pi^*\omega$). By basic linear algebra, the same holds for $\tilde G\subset \tE$. In particular, the orthogonal $\tilde G^\perp$ is a holomorphic subbundle of $\tE$ yielding a holomorphic decomposition $\tE=\tilde G\oplus \tilde G^\perp$ on $\tY^\circ$.

Pushing forward by $\mu$, which is an isomorphism over $\tY^\circ$, we get $E=G\oplus G^\perp$ on the open set
\[
Y^\circ:=\mu(\tY^\circ)=g^{-1}(X_{\rm reg}\setminus \Supp(\Delta+\pi(A)))\setminus \mu(S),
\]
over which $g^*\omega$ is a genuine K\"ahler form and $G\subset E$ is a subbundle. We claim that the vector bundle $G^\perp$ on $Y^\circ$ extends to a reflexive sheaf on the whole $Y$ such that $E=G\oplus G^\perp$ globally on $Y$. In order to prove it, one argues as follows. Thinking of $G^\perp$ as a map $p:E\to E$ defined on $Y^\circ$ such that $p^2=p$ and $\mathrm{Im}(p)=G$, it is enough to show that $p$ extends to a morphism $E\to E$ on $Y$. Since $E$ and $G$ are reflexive and $g^{-1}(X_{\rm sing})\cup \mu(S)$ has codimension at least two, it is enough to prove the extension property at a general point of each component of $g^{-1}(\Delta+\pi(A))$. By construction, we have
\begin{equation}
\label{bounded-projection}
|p(e)|^2_{h_\omega}\le |e|^2_{h_\omega},
\end{equation}
where $h_\omega$ denotes the hermitian metric on $E$ induced by $g^*\omega$ over $Y^\circ$. Thanks to \cite[Theorem~2]{GP-conic}, $h_\omega$ is quasi-isometric to a fixed smooth hermitian metric on $E$ at a general point of a component of $g^{-1}(\Delta+\pi(A))$. Near such a point, $E$ is a vector bundle and one can think of $p$ as a holomorphic map $\C^{n+1}\to \C^{n+1}$ defined away from a hyperplane. By the ``boundedness'' of $h$ and \eqref{bounded}, one infers that $p$ is bounded near the said hyperplane, hence extends holomorphically across it. This ends the proof of the polystability of $E$.

\bigskip

$(ii)$ We now only require the pair $(X,\Delta)$ to be K-semistable. Pick $H\in |-m(K_X+\Delta)|$ general for $m$ large and divisible and consider the log Fano pair $(X,\Delta_\ep:=\Delta+\ep H)$ for $\ep\in \Q_{>0}$ small. Let $h_\ep:Y_\ep\to Y$ be an adapted cover for the pair $(Y,\ep g^* H)$ and denote by $g_\ep : Y_\ep \to X$ the induced map. We have a natural injection 
\begin{equation}
\label{inj form}
h_\ep^{[*]}\Omega_{(X,\Delta, g)}^{[1]}\simeq \Omega_{(X,\Delta,g_\ep)}^{[1]}\subset \Omega_{(X,\Delta_\ep,g_\ep)}^{[1]}.
\end{equation} Next, let $E_\ep$ be the extension of $\cO_{Y_\ep}$ by $\Omega_{(X,\Delta_\ep, g_\ep)}^{[1]}$ with extension class induced by $-(K_X+\Delta)$ under the map $\mathrm{Pic}(X)_{\Q}\to H^1(Y_\ep, \Omega_{(X,\Delta_\ep, g_\ep)}^{[1]})$, see \eqref{ext map}.  It is easy to check directly using the transition maps \eqref{transition} that \eqref{inj form} induces an injective map 
\begin{equation}
\label{inj ext}
h_\ep^{[*]}E\subset E_\ep,
\end{equation}
where $E:=E_{(X,\Delta,g)}$ is the canonical extension of $(X,\Delta, g)$. Moreover, since $K_X+\Delta_\ep$ is proportional to $K_X+\Delta$, $E_\ep$ is abstractly isomorphic to the canonical extension of $(X,\Delta_\ep, g_\ep)$, i.e. the extension of $\cO_{Y_\ep}$ by $\Omega_{(X,\Delta_\ep, g_\ep)}^{[1]}$ with extension class induced by $-(K_X+\Delta_\ep)$. Thanks to \cite[Lemma 5.5]{LXZ-HRFG} (the same proof works verbatim when $t=0$), the pair $(X,\Delta_\ep)$ is uniformly K-stable (hence K-polystable, too) for any $\ep>0$ small enough. Let $F\subset E$ be a coherent subsheaf of rank $r$, which induces a subsheaf $h_\ep^{[*]}F\subset E_\ep$ via \eqref{inj ext}.  By the previous case, $E_\ep$ is semistable with respect to $g_\ep^*\alpha$ hence we get
\begin{eqnarray*}
\frac{\deg(h_\ep)}{r}(c_1(F) \cdot g^*\alpha^{n-1})&=& \frac 1 r (c_1(h_\ep^{[*]}F) \cdot h_\ep^*g^*\alpha^{n-1})\\
& \le& \frac{1}{n+1}(c_1(E_\ep) \cdot g_\ep^*\alpha^{n-1})\\
&=&-\frac{1-\ep}{n+1}  \cdot \deg(h_\ep) \cdot g^*\alpha^n.  
\end{eqnarray*}
Dividing both sides by $\deg(h_\ep)$ and letting $\ep$ go to zero yields the desired slope semistability of $E$ with respect to $g^*\alpha$.
\end{proof}

\clearpage
\markboth{}{}

\bibliography{ref}

\begin{bibdiv}
\begin{biblist}

\bib{ABHLX-reductivity}{article}{
      author={Alper, Jarod},
      author={Blum, Harold},
      author={Halpern-Leistner, Daniel},
      author={Xu, Chenyang},
       title={Reductivity of the automorphism group of {$K$}-polystable {F}ano varieties},
        date={2020},
        ISSN={0020-9910,1432-1297},
     journal={Invent. Math.},
      volume={222},
      number={3},
       pages={995\ndash 1032},
      review={\MR{4169054}},
}

\bib{AGV-GW-DM}{article}{
      author={Abramovich, Dan},
      author={Graber, Tom},
      author={Vistoli, Angelo},
       title={Gromov-{W}itten theory of {D}eligne-{M}umford stacks},
        date={2008},
        ISSN={0002-9327,1080-6377},
     journal={Amer. J. Math.},
      volume={130},
      number={5},
       pages={1337\ndash 1398},
         url={https://doi.org/10.1353/ajm.0.0017},
      review={\MR{2450211}},
}

\bib{AOV-tame-stack}{article}{
      author={Abramovich, Dan},
      author={Olsson, Martin},
      author={Vistoli, Angelo},
       title={Tame stacks in positive characteristic},
        date={2008},
        ISSN={0373-0956,1777-5310},
     journal={Ann. Inst. Fourier (Grenoble)},
      volume={58},
      number={4},
       pages={1057\ndash 1091},
      review={\MR{2427954}},
}

\bib{Artin-approx}{article}{
      author={Artin, Michael},
       title={Algebraic approximation of structures over complete local rings},
        date={1969},
        ISSN={0073-8301,1618-1913},
     journal={Inst. Hautes \'Etudes Sci. Publ. Math.},
      number={36},
       pages={23\ndash 58},
         url={http://www.numdam.org/item?id=PMIHES_1969__36__23_0},
      review={\MR{268188}},
}

\bib{BD-toric}{article}{
      author={Bielawski, Roger},
      author={Dancer, Andrew~S.},
       title={The geometry and topology of toric hyperk\"{a}hler manifolds},
        date={2000},
        ISSN={1019-8385,1944-9992},
     journal={Comm. Anal. Geom.},
      volume={8},
      number={4},
       pages={727\ndash 760},
         url={https://doi.org/10.4310/CAG.2000.v8.n4.a2},
      review={\MR{1792372}},
}

\bib{BdFFU-log-discrepancy}{incollection}{
      author={Boucksom, S\'ebastien},
      author={de~Fernex, Tommaso},
      author={Favre, Charles},
      author={Urbinati, Stefano},
       title={Valuation spaces and multiplier ideals on singular varieties},
        date={2015},
   booktitle={Recent advances in algebraic geometry},
      series={London Math. Soc. Lecture Note Ser.},
      volume={417},
   publisher={Cambridge Univ. Press, Cambridge},
       pages={29\ndash 51},
}

\bib{Bea-symp-sing}{article}{
      author={Beauville, Arnaud},
       title={Symplectic singularities},
        date={2000},
        ISSN={0020-9910,1432-1297},
     journal={Invent. Math.},
      volume={139},
      number={3},
       pages={541\ndash 549},
      review={\MR{1738060}},
}

\bib{BJ-delta}{article}{
      author={Blum, Harold},
      author={Jonsson, Mattias},
       title={Thresholds, valuations, and {K}-stability},
        date={2020},
        ISSN={0001-8708,1090-2082},
     journal={Adv. Math.},
      volume={365},
       pages={107062, 57},
      review={\MR{4067358}},
}

\bib{BK-nilpotent}{article}{
      author={Brylinski, Ranee},
      author={Kostant, Bertram},
       title={Nilpotent orbits, normality and {H}amiltonian group actions},
        date={1994},
        ISSN={0894-0347,1088-6834},
     journal={J. Amer. Math. Soc.},
      volume={7},
      number={2},
       pages={269\ndash 298},
         url={https://doi.org/10.2307/2152759},
      review={\MR{1239505}},
}

\bib{BLQ-convexity}{article}{
      author={Blum, Harold},
      author={Liu, Yuchen},
      author={Qi, Lu},
       title={Convexity of multiplicities of filtrations on local rings},
        date={2024},
        ISSN={0010-437X,1570-5846},
     journal={Compos. Math.},
      volume={160},
      number={4},
       pages={878\ndash 914},
      review={\MR{4716588}},
}

\bib{Blu-existence}{article}{
      author={Blum, Harold},
       title={Existence of valuations with smallest normalized volume},
        date={2018},
        ISSN={0010-437X},
     journal={Compos. Math.},
      volume={154},
      number={4},
       pages={820\ndash 849},
         url={https://doi.org/10.1112/S0010437X17008016},
      review={\MR{3778195}},
}

\bib{Cad-root-stack}{article}{
      author={Cadman, Charles},
       title={Using stacks to impose tangency conditions on curves},
        date={2007},
        ISSN={0002-9327,1080-6377},
     journal={Amer. J. Math.},
      volume={129},
      number={2},
       pages={405\ndash 427},
         url={https://doi.org/10.1353/ajm.2007.0007},
      review={\MR{2306040}},
}

\bib{Che-HRFG}{misc}{
      author={Chen, Zhiyuan},
       title={Finite generation of higher rank quasi-monomial valuations via the extended {R}ees algebra},
        date={2025},
         url={https://arxiv.org/abs/2510.10737},
        note={\href{https://arxiv.org/abs/2510.10737}{\textsf{arXiv:2510.10737}}},
}

\bib{CKT-hyperbolicity}{incollection}{
      author={Claudon, Beno\^it},
      author={Kebekus, Stefan},
      author={Taji, Behrouz},
       title={Generic positivity and applications to hyperbolicity of moduli spaces},
        date={2021},
   booktitle={Hyperbolicity properties of algebraic varieties},
      series={Panor. Synth\`eses},
      volume={56},
   publisher={Soc. Math. France, Paris},
       pages={169\ndash 208},
      review={\MR{4457669}},
}

\bib{CM-nilpotent-orbit}{book}{
      author={Collingwood, David~H.},
      author={McGovern, William~M.},
       title={Nilpotent orbits in semisimple {L}ie algebras},
      series={Van Nostrand Reinhold Mathematics Series},
   publisher={Van Nostrand Reinhold Co., New York},
        date={1993},
        ISBN={0-534-18834-6},
      review={\MR{1251060}},
}

\bib{CS-cone}{article}{
      author={Collins, Tristan~C.},
      author={Sz\'{e}kelyhidi, G\'{a}bor},
       title={K-semistability for irregular {S}asakian manifolds},
        date={2018},
        ISSN={0022-040X,1945-743X},
     journal={J. Differential Geom.},
      volume={109},
      number={1},
       pages={81\ndash 109},
         url={https://doi.org/10.4310/jdg/1525399217},
      review={\MR{3798716}},
}

\bib{Dai-MY}{misc}{
      author={Dailly, Louis},
       title={Miyaoka-{Y}au equality and uniformization of log {F}ano pairs},
        date={2025},
         url={https://arxiv.org/abs/2501.05887},
        note={\href{https://arxiv.org/abs/2501.05887}{\textsf{arXiv:2501.05887}}},
}

\bib{Dai-stability}{misc}{
      author={Dailly, Louis},
       title={Stability properties of adapted tangent sheaves on {K}\"ahler--{E}instein log {F}ano pairs},
        date={2026},
         url={https://arxiv.org/abs/2601.19608},
        note={\href{https://arxiv.org/abs/2601.19608}{\textsf{arXiv:2601.19608}}},
}

\bib{DGP-Q-Fano-decomp}{article}{
      author={Druel, St\'{e}phane},
      author={Guenancia, Henri},
      author={P\u{a}un, Mihai},
       title={A decomposition theorem for {$\mathbb{Q}$}-{F}ano {K}\"{a}hler-{E}instein varieties},
        date={2024},
        ISSN={1631-073X,1778-3569},
     journal={C. R. Math. Acad. Sci. Paris},
      volume={362},
      number={Special issue S1},
       pages={93\ndash 118},
         url={https://doi.org/10.5802/crmath.612},
      review={\MR{4762190}},
}

\bib{DS-degeneration2}{article}{
      author={Donaldson, Simon},
      author={Sun, Song},
       title={Gromov-{H}ausdorff limits of {K}\"{a}hler manifolds and algebraic geometry, {II}},
        date={2017},
        ISSN={0022-040X},
     journal={J. Differential Geom.},
      volume={107},
      number={2},
       pages={327\ndash 371},
         url={https://doi.org/10.4310/jdg/1506650422},
      review={\MR{3707646}},
}

\bib{ELS03}{article}{
      author={Ein, Lawrence},
      author={Lazarsfeld, Robert},
      author={Smith, Karen~E.},
       title={Uniform approximation of {A}bhyankar valuation ideals in smooth function fields},
        date={2003},
     journal={Amer. J. Math.},
      volume={125},
      number={2},
       pages={409\ndash 440},
}

\bib{FO-delta}{article}{
      author={Fujita, Kento},
      author={Odaka, Yuji},
       title={On the {K}-stability of {F}ano varieties and anticanonical divisors},
        date={2018},
        ISSN={0040-8735,2186-585X},
     journal={Tohoku Math. J. (2)},
      volume={70},
      number={4},
       pages={511\ndash 521},
      review={\MR{3896135}},
}

\bib{GKKP-extend-forms}{article}{
      author={Greb, Daniel},
      author={Kebekus, Stefan},
      author={Kov\'acs, S\'andor~J.},
      author={Peternell, Thomas},
       title={Differential forms on log canonical spaces},
        date={2011},
        ISSN={0073-8301,1618-1913},
     journal={Publ. Math. Inst. Hautes \'Etudes Sci.},
      number={114},
       pages={87\ndash 169},
         url={https://doi.org/10.1007/s10240-011-0036-0},
      review={\MR{2854859}},
}

\bib{GP-conic}{article}{
      author={Guenancia, Henri},
      author={P\u{a}un, Mihai},
       title={Conic singularities metrics with prescribed {R}icci curvature: general cone angles along normal crossing divisors},
        date={2016},
        ISSN={0022-040X,1945-743X},
     journal={J. Differential Geom.},
      volume={103},
      number={1},
       pages={15\ndash 57},
         url={http://projecteuclid.org/euclid.jdg/1460463562},
      review={\MR{3488129}},
}

\bib{GT-orbifold-stab}{article}{
      author={Guenancia, Henri},
      author={Taji, Behrouz},
       title={Orbifold stability and {M}iyaoka-{Y}au inequality for minimal pairs},
        date={2022},
        ISSN={1465-3060,1364-0380},
     journal={Geom. Topol.},
      volume={26},
      number={4},
       pages={1435\ndash 1482},
         url={https://doi.org/10.2140/gt.2022.26.1435},
      review={\MR{4504444}},
}

\bib{G12}{article}{
      author={Guenancia, Henri},
       title={K\"ahler-{E}instein metrics with mixed {P}oincar\'e{} and cone singularities along a normal crossing divisor},
        date={2014},
        ISSN={0373-0956,1777-5310},
     journal={Ann. Inst. Fourier (Grenoble)},
      volume={64},
      number={3},
       pages={1291\ndash 1330},
         url={https://doi.org/10.5802/aif.2881},
      review={\MR{3330171}},
}

\bib{GSS}{article}{
      author={Guenancia, Henri},
       title={Semistability of the tangent sheaf of singular varieties},
        date={2016},
        ISSN={2313-1691,2214-2584},
     journal={Algebr. Geom.},
      volume={3},
      number={5},
       pages={508\ndash 542},
         url={https://doi.org/10.14231/AG-2016-024},
      review={\MR{3568336}},
}

\bib{HL-book-sheaf}{book}{
      author={Huybrechts, Daniel},
      author={Lehn, Manfred},
       title={The geometry of moduli spaces of sheaves},
     edition={Second},
      series={Cambridge Mathematical Library},
   publisher={Cambridge University Press, Cambridge},
        date={2010},
        ISBN={978-0-521-13420-0},
      review={\MR{2665168}},
}

\bib{Hausel-Sturmfels}{article}{
      author={Hausel, Tam\'{a}s},
      author={Sturmfels, Bernd},
       title={Toric hyper{K}\"{a}hler varieties},
        date={2002},
        ISSN={1431-0635,1431-0643},
     journal={Doc. Math.},
      volume={7},
       pages={495\ndash 534},
      review={\MR{2015052}},
}

\bib{Huang-thesis}{book}{
      author={Huang, Kai},
       title={K-stability of log {F}ano cone singularities},
        date={2022},
        note={Thesis (Ph.D.)--Massachusetts Institute of Technology},
}

\bib{Humphreys-alg-gp}{book}{
      author={Humphreys, James~E.},
       title={Linear algebraic groups},
      series={Graduate Texts in Mathematics},
   publisher={Springer-Verlag, New York-Heidelberg},
        date={1975},
      volume={No. 21},
      review={\MR{396773}},
}

\bib{JM-val-ideal-seq}{article}{
      author={Jonsson, Mattias},
      author={Musta\c{t}\u{a}, Mircea},
       title={Valuations and asymptotic invariants for sequences of ideals},
        date={2012},
        ISSN={0373-0956,1777-5310},
     journal={Ann. Inst. Fourier (Grenoble)},
      volume={62},
      number={6},
       pages={2145\ndash 2209},
      review={\MR{3060755}},
}

\bib{Kaledin-sym}{article}{
      author={Kaledin, D.},
       title={Symplectic singularities from the {P}oisson point of view},
        date={2006},
        ISSN={0075-4102,1435-5345},
     journal={J. Reine Angew. Math.},
      volume={600},
       pages={135\ndash 156},
         url={https://doi.org/10.1515/CRELLE.2006.089},
      review={\MR{2283801}},
}

\bib{Kaledin-survey}{incollection}{
      author={Kaledin, D.},
       title={Geometry and topology of symplectic resolutions},
        date={2009},
   booktitle={Algebraic geometry---{S}eattle 2005. {P}art 2},
      series={Proc. Sympos. Pure Math.},
      volume={80},
   publisher={Amer. Math. Soc., Providence, RI},
       pages={595\ndash 628},
         url={https://doi.org/10.1090/pspum/080.2/2483948},
      review={\MR{2483948}},
}

\bib{Kau-group-action}{article}{
      author={Kaup, Wilhelm},
       title={Infinitesimale {T}ransformationsgruppen komplexer {R}\"aume},
        date={1965},
        ISSN={0025-5831,1432-1807},
     journal={Math. Ann.},
      volume={160},
       pages={72\ndash 92},
      review={\MR{181761}},
}

\bib{KM98}{book}{
      author={Koll\'ar, J\'anos},
      author={Mori, Shigefumi},
       title={Birational geometry of algebraic varieties},
      series={Cambridge Tracts in Mathematics},
   publisher={Cambridge University Press, Cambridge},
        date={1998},
      volume={134},
        ISBN={0-521-63277-3},
        note={With the collaboration of C. H. Clemens and A. Corti, Translated from the 1998 Japanese original},
      review={\MR{1658959}},
}

\bib{Kol-Seifert-bundle}{article}{
      author={Koll\'ar, J\'anos},
       title={Seifert {$G_m$}-bundles},
        date={2004},
        note={\href{https://arxiv.org/abs/math/0404386}{\textsf{arXiv:0404386}}},
}

\bib{Kol13}{book}{
      author={Koll\'ar, J\'anos},
       title={Singularities of the minimal model program},
      series={Cambridge Tracts in Mathematics},
   publisher={Cambridge University Press, Cambridge},
        date={2013},
      volume={200},
        ISBN={978-1-107-03534-8},
        note={With a collaboration of S\'andor Kov\'acs},
      review={\MR{3057950}},
}

\bib{Kol-cone-auto}{article}{
      author={Koll\'ar, J\'anos},
       title={Automorphisms and twisted forms of rings of invariants},
        date={2024},
        ISSN={1424-9286,1424-9294},
     journal={Milan J. Math.},
      volume={92},
      number={2},
       pages={473\ndash 499},
         url={https://doi.org/10.1007/s00032-024-00403-x},
      review={\MR{4836147}},
}

\bib{Kronheimer-nilpotent}{article}{
      author={Kronheimer, P.~B.},
       title={Instantons and the geometry of the nilpotent variety},
        date={1990},
        ISSN={0022-040X,1945-743X},
     journal={J. Differential Geom.},
      volume={32},
      number={2},
       pages={473\ndash 490},
         url={http://projecteuclid.org/euclid.jdg/1214445316},
      review={\MR{1072915}},
}

\bib{KS-hyperkahler-potential}{article}{
      author={Kobak, Piotr},
      author={Swann, Andrew},
       title={Hyper{K}\"ahler potentials via finite-dimensional quotients},
        date={2001},
        ISSN={0046-5755,1572-9168},
     journal={Geom. Dedicata},
      volume={88},
      number={1-3},
       pages={1\ndash 19},
         url={https://doi.org/10.1023/A:1013174027792},
      review={\MR{1877205}},
}

\bib{KS-extend-forms}{article}{
      author={Kebekus, Stefan},
      author={Schnell, Christian},
       title={Extending holomorphic forms from the regular locus of a complex space to a resolution of singularities},
        date={2021},
        ISSN={0894-0347,1088-6834},
     journal={J. Amer. Math. Soc.},
      volume={34},
      number={2},
       pages={315\ndash 368},
         url={https://doi.org/10.1090/jams/962},
      review={\MR{4280862}},
}

\bib{KS-classical-nilpotent}{article}{
      author={Kobak, Piotr},
      author={Swann, Andrew},
       title={Classical nilpotent orbits as hyper-{K}\"ahler quotients},
        date={1996},
        ISSN={0129-167X,1793-6519},
     journal={Internat. J. Math.},
      volume={7},
      number={2},
       pages={193\ndash 210},
         url={https://doi.org/10.1142/S0129167X96000116},
      review={\MR{1382722}},
}

\bib{Li-valuative-criterion}{article}{
      author={Li, Chi},
       title={K-semistability is equivariant volume minimization},
        date={2017},
        ISSN={0012-7094,1547-7398},
     journal={Duke Math. J.},
      volume={166},
      number={16},
       pages={3147\ndash 3218},
      review={\MR{3715806}},
}

\bib{Li-nv}{article}{
      author={Li, Chi},
       title={Minimizing normalized volumes of valuations},
        date={2018},
        ISSN={0025-5874,1432-1823},
     journal={Math. Z.},
      volume={289},
      number={1-2},
       pages={491\ndash 513},
      review={\MR{3803800}},
}

\bib{Li-Fano-cone-YTD}{article}{
      author={Li, Chi},
       title={Notes on weighted {K}\"ahler-{R}icci solitons and application to {R}icci-flat {K}\"ahler cone metrics},
        date={2021},
        note={\href{https://sites.math.rutgers.edu/~cl1412/notes/soliton-SE.pdf}{link}},
}

\bib{Li-extension-stab}{article}{
      author={Li, Chi},
       title={On the stability of extensions of tangent sheaves on {K}\"ahler-{E}instein {F}ano/{C}alabi-{Y}au pairs},
        date={2021},
        ISSN={0025-5831,1432-1807},
     journal={Math. Ann.},
      volume={381},
      number={3-4},
       pages={1943\ndash 1977},
         url={https://doi.org/10.1007/s00208-020-02099-x},
      review={\MR{4333434}},
}

\bib{Li-YTD}{article}{
      author={Li, Chi},
       title={{$G$}-uniform stability and {K}\"ahler-{E}instein metrics on {F}ano varieties},
        date={2022},
        ISSN={0020-9910,1432-1297},
     journal={Invent. Math.},
      volume={227},
      number={2},
       pages={661\ndash 744},
      review={\MR{4372222}},
}

\bib{LLX-nvsurvey}{incollection}{
      author={Li, Chi},
      author={Liu, Yuchen},
      author={Xu, Chenyang},
       title={A guided tour to normalized volume},
        date={2020},
   booktitle={Geometric analysis---in honor of {G}ang {T}ian's 60th birthday},
      series={Progr. Math.},
      volume={333},
   publisher={Birkh\"auser/Springer, Cham},
       pages={167\ndash 219},
      review={\MR{4181002}},
}

\bib{LWX-tangent-cone}{article}{
      author={Li, Chi},
      author={Wang, Xiaowei},
      author={Xu, Chenyang},
       title={Algebraicity of the metric tangent cones and equivariant {K}-stability},
        date={2021},
        ISSN={0894-0347,1088-6834},
     journal={J. Amer. Math. Soc.},
      volume={34},
      number={4},
       pages={1175\ndash 1214},
      review={\MR{4301561}},
}

\bib{LX-higher-rank}{article}{
      author={Li, Chi},
      author={Xu, Chenyang},
       title={Stability of valuations: higher rational rank},
        date={2018},
        ISSN={2096-6075,2524-7182},
     journal={Peking Math. J.},
      volume={1},
      number={1},
       pages={1\ndash 79},
      review={\MR{4059992}},
}

\bib{LX-Kol-comp-stab}{article}{
      author={Li, Chi},
      author={Xu, Chenyang},
       title={Stability of valuations and {K}oll\'ar components},
        date={2020},
        ISSN={1435-9855,1435-9863},
     journal={J. Eur. Math. Soc. (JEMS)},
      volume={22},
      number={8},
       pages={2573\ndash 2627},
      review={\MR{4118616}},
}

\bib{LXZ-HRFG}{article}{
      author={Liu, Yuchen},
      author={Xu, Chenyang},
      author={Zhuang, Ziquan},
       title={Finite generation for valuations computing stability thresholds and applications to {K}-stability},
        date={2022},
        ISSN={0003-486X,1939-8980},
     journal={Ann. of Math. (2)},
      volume={196},
      number={2},
       pages={507\ndash 566},
      review={\MR{4445441}},
}

\bib{LZ-Tian-sharpness}{article}{
      author={Liu, Yuchen},
      author={Zhuang, Ziquan},
       title={On the sharpness of {T}ian's criterion for {K}-stability},
        date={2022},
        ISSN={0027-7630,2152-6842},
     journal={Nagoya Math. J.},
      volume={245},
       pages={41\ndash 73},
         url={https://doi.org/10.1017/nmj.2020.28},
      review={\MR{4413362}},
}

\bib{Moser-trick}{article}{
      author={Moser, J\"urgen},
       title={On the volume elements on a manifold},
        date={1965},
        ISSN={0002-9947,1088-6850},
     journal={Trans. Amer. Math. Soc.},
      volume={120},
       pages={286\ndash 294},
      review={\MR{182927}},
}

\bib{Namikawa-deformation-terminal}{article}{
      author={Namikawa, Yoshinori},
       title={Flops and {P}oisson deformations of symplectic varieties},
        date={2008},
        ISSN={0034-5318,1663-4926},
     journal={Publ. Res. Inst. Math. Sci.},
      volume={44},
      number={2},
       pages={259\ndash 314},
         url={https://doi.org/10.2977/prims/1210167328},
      review={\MR{2426349}},
}

\bib{Namikawa-deformation}{article}{
      author={Namikawa, Yoshinori},
       title={Poisson deformations of affine symplectic varieties},
        date={2011},
        ISSN={0012-7094,1547-7398},
     journal={Duke Math. J.},
      volume={156},
      number={1},
       pages={51\ndash 85},
         url={https://doi.org/10.1215/00127094-2010-066},
      review={\MR{2746388}},
}

\bib{Namikawa-equivalence-symp}{article}{
      author={Namikawa, Yoshinori},
       title={Equivalence of symplectic singularities},
        date={2013},
        ISSN={2156-2261,2154-3321},
     journal={Kyoto J. Math.},
      volume={53},
      number={2},
       pages={483\ndash 514},
      review={\MR{3079311}},
}

\bib{Namikawa-finite}{article}{
      author={Namikawa, Yoshinori},
       title={A finiteness theorem on symplectic singularities},
        date={2016},
        ISSN={0010-437X,1570-5846},
     journal={Compos. Math.},
      volume={152},
      number={6},
       pages={1225\ndash 1236},
      review={\MR{3518310}},
}

\bib{Namikawa-notes-on-deformation}{misc}{
      author={Namikawa, Yoshinori},
       title={Notes on {P}oisson deformations of symplectic varieties},
        date={2026},
         url={https://arxiv.org/abs/2601.01131},
        note={\href{https://arxiv.org/abs/2601.01131}{\textsf{arXiv:2601.01131}}},
}

\bib{Namikawa-torichyperkahler}{article}{
      author={Namikawa, Yoshinori},
       title={Towards a characterization of toric hyperk\"{a}hler varieties among symplectic singularities},
        date={2026},
        ISSN={1022-1824,1420-9020},
     journal={Selecta Math. (N.S.)},
      volume={32},
      number={1},
       pages={Paper No. 4, 49},
         url={https://doi.org/10.1007/s00029-025-01118-6},
      review={\MR{5002795}},
}

\bib{Namikawa-torichyperkahlerII}{article}{
      author={Namikawa, Yoshinori},
       title={Towards a characterization of toric hyperk\"{a}hler varieties among symplectic singularities {II}},
        date={2026},
        note={\href{ https://arxiv.org/abs/2603.27454}{\textsf{arXiv:2603.27454}}},
}

\bib{NO-symp}{article}{
      author={Namikawa, Yoshinori},
      author={Odaka, Yuji},
       title={Canonical torus action on symplectic singularities},
        date={2025},
        note={\href{ https://arxiv.org/abs/2503.15791}{\textsf{arXiv:2503.15791}}},
}

\bib{Pan-nil-orbit-symp}{article}{
      author={Panyushev, D.~I.},
       title={Rationality of singularities and the {G}orenstein property of nilpotent orbits},
        date={1991},
        ISSN={0374-1990},
     journal={Funktsional. Anal. i Prilozhen.},
      volume={25},
      number={3},
       pages={76\ndash 78},
         url={https://doi.org/10.1007/BF01085494},
      review={\MR{1139878}},
}

\bib{Proudfoot-survey}{incollection}{
      author={Proudfoot, Nicholas~J.},
       title={A survey of hypertoric geometry and topology},
        date={2008},
   booktitle={Toric topology},
      series={Contemp. Math.},
      volume={460},
   publisher={Amer. Math. Soc., Providence, RI},
       pages={323\ndash 338},
         url={https://doi.org/10.1090/conm/460/09027},
      review={\MR{2428365}},
}

\bib{stacks-project}{misc}{
      author={{Stacks Project Authors}, The},
       title={Stacks project},
         url={https://stacks.math.columbia.edu/},
        note={\href{https://stacks.math.columbia.edu/}{\textsf{Link}}},
}

\bib{Sumihiro}{article}{
      author={Sumihiro, Hideyasu},
       title={Equivariant completion},
        date={1974},
        ISSN={0023-608X},
     journal={J. Math. Kyoto Univ.},
      volume={14},
       pages={1\ndash 28},
      review={\MR{337963}},
}

\bib{Tian-extension}{article}{
      author={Tian, Gang},
       title={On stability of the tangent bundles of {F}ano varieties},
        date={1992},
        ISSN={0129-167X,1793-6519},
     journal={Internat. J. Math.},
      volume={3},
      number={3},
       pages={401\ndash 413},
         url={https://doi.org/10.1142/S0129167X92000175},
      review={\MR{1163733}},
}

\bib{Xu-pi_1-finite}{article}{
      author={Xu, Chenyang},
       title={Finiteness of algebraic fundamental groups},
        date={2014},
        ISSN={0010-437X,1570-5846},
     journal={Compos. Math.},
      volume={150},
      number={3},
       pages={409\ndash 414},
      review={\MR{3187625}},
}

\bib{Xu-quasimonomial}{article}{
      author={Xu, Chenyang},
       title={A minimizing valuation is quasi-monomial},
        date={2020},
        ISSN={0003-486X},
     journal={Ann. of Math. (2)},
      volume={191},
      number={3},
       pages={1003\ndash 1030},
         url={https://doi-org.libproxy.mit.edu/10.4007/annals.2020.191.3.6},
      review={\MR{4088355}},
}

\bib{Xu-book}{book}{
      author={Xu, Chenyang},
       title={K-stability of {F}ano varieties},
      series={New Mathematical Monographs},
   publisher={Cambridge University Press, Cambridge},
        date={2025},
      volume={50},
        ISBN={978-1-009-53877-0; [9781009538763]},
      review={\MR{4893062}},
}

\bib{XZ-uniqueness}{article}{
      author={Xu, Chenyang},
      author={Zhuang, Ziquan},
       title={Uniqueness of the minimizer of the normalized volume function},
        date={2021},
        ISSN={2168-0930,2168-0949},
     journal={Camb. J. Math.},
      volume={9},
      number={1},
       pages={149\ndash 176},
      review={\MR{4325260}},
}

\bib{XZ-SDC}{article}{
      author={Xu, Chenyang},
      author={Zhuang, Ziquan},
       title={Stable degenerations of singularities},
        date={2025},
        ISSN={0894-0347,1088-6834},
     journal={J. Amer. Math. Soc.},
      volume={38},
      number={3},
       pages={585\ndash 626},
      review={\MR{4890655}},
}

\bib{XZ-open}{article}{
      author={Xu, Chenyang},
      author={Zhuang, Ziquan},
       title={Open problems in {K}-stability of {F}ano varieties},
        date={2026},
        note={\href{ https://arxiv.org/abs/2601.15576}{\textsf{arXiv:2601.15576}}},
}

\bib{Z-product-K}{article}{
      author={Zhuang, Ziquan},
       title={Product theorem for {K}-stability},
        date={2020},
        ISSN={0001-8708,1090-2082},
     journal={Adv. Math.},
      volume={371},
       pages={107250, 18},
      review={\MR{4108221}},
}

\bib{Z-equivariant-K}{article}{
      author={Zhuang, Ziquan},
       title={Optimal destabilizing centers and equivariant {K}-stability},
        date={2021},
        ISSN={0020-9910,1432-1297},
     journal={Invent. Math.},
      volume={226},
      number={1},
       pages={195\ndash 223},
      review={\MR{4309493}},
}

\bib{Z-survey-klt-stab}{incollection}{
      author={Zhuang, Ziquan},
       title={Stability of klt singularities},
        date={2025},
   booktitle={Handbook of geometry and topology of singularities {VII}},
   publisher={Springer, Cham},
       pages={501\ndash 551},
      review={\MR{4890435}},
}

\end{biblist}
\end{bibdiv}

\end{document}